\documentclass[12pt]{amsart}
\usepackage{fullpage}
\usepackage{bm}
\usepackage{amsmath, amsthm, amscd, amssymb, amsfonts, latexsym}
\usepackage{mathrsfs}
\usepackage{bbm}
\usepackage{dsfont}
\usepackage{bbm}
\usepackage{relsize}
\usepackage{comment}
\usepackage{xcolor}
\usepackage[mathscr]{euscript}
\usepackage{upgreek}
\usepackage[hidelinks]{hyperref}

\newtheorem{theorem}{Theorem}[section]
\newtheorem{lemma}[theorem]{Lemma}
\newtheorem{proposition}[theorem]{Proposition}
\newtheorem{corollary}[theorem]{Corollary}
\newtheorem{conjecture}[theorem]{Conjecture}

\theoremstyle{definition}
\newtheorem{remark}[theorem]{Remark}

\numberwithin{equation}{section}

\DeclareMathOperator{\re}{Re}
\DeclareMathOperator{\im}{Im}

\newcommand{\A}{{\color{red}AAA}}
\newcommand{\Q}{\mathbb{Q}}
\newcommand{\R}{\mathbb{R}}
\newcommand{\Z}{\mathbb{Z}}
\newcommand{\F}{\mathbb{F}}
\newcommand{\N}{{\rm N}}
\newcommand{\C}{\mathbb{C}}
\newcommand{\p}{\mathfrak{p}}

\newcommand{\sgn}{\mbox{sgn}}

\newcommand{\G}{\mathcal{G}}
\newcommand{\rad}{\mathrm{rad}}
\newcommand{\quartic}[2]{\left( \frac{#1}{#2} \right)_4}

\newcommand{\oquartic}[2]{\overline{\left( \frac{#1}{#2} \right)}_4}

\newcommand{\hecke}[1]{\frac{\overline{#1}}{| #1 |}}
\newcommand{\ppmod}[1]{\hspace{-0.15cm}\pmod{#1}}

\newcommand{\sumsplit}{\sum_{\substack{p \equiv 1 \bmod 4\\ p = \pi\overline{\pi} \\ \pi \equiv 1 \bmod (\lambda^3)\\ k \geq 1}}}
\newcommand{\newsumsplit}{ \sum_{\substack{\N(\pi) = p \equiv 1 \bmod 4\\ \pi \equiv 1 \bmod (\lambda^3) \\ k \geq 1}}}

\def\fp{\mathfrak{p}}

\def\fp{\mathfrak{p}}

\def\cF{\mathcal{F}}

\def\Tr{\mathrm{Tr}}
\def\tr{\mathrm{Tr}}

\renewcommand{\d}{\mathrm{d}}
\newcommand{\dd}{\mathrm{d}}

\newcommand{\kommentar}[1]{}

\title{Average analytic rank for the $L$-functions of the elliptic curves $y^2=x^3-dx$}

\author{Chantal David}
\address{Department of Mathematics and Statistics, Concordia University\\
1455 de Maisonneuve West, Montreal, H2G 1M8, Qu\'ebec, Canada.
}
\email{chntl.david@gmail.com}

\author{Lucile Devin}
\address{Univ. Littoral C\^ote d'Opale, UR 2597
	LMPA, Laboratoire de Math\'ematiques Pures et Appliqu\'ees Joseph Liouville,
	F-62100 Calais, France
}
\address{CNRS -- Université de Montréal CRM - CNRS}
\email{lucile.devin@univ-littoral.fr} 
	
\author{Alessandro Fazzari}
\address{DIMA - Dipartimento di Matematica, Via Dodecaneso, 35, 16146 Genova, Italy}
\email{alessandro.fazzari@unige.it}

\author{Ezra Waxman}
\address{Unit of Mathematics, Afeka — The Academic College of Engineering in Tel Aviv, Mivtsa Kadesh St 38, Tel Aviv-Yafo 6998812, Israel}
\address{Department of Mathematics, University of Haifa, 199 Aba Khoushy Ave., Mt. Carmel, Haifa 3498838, Israel}
\email{ezraw@afeka.ac.il}

\begin{document}

\begin{abstract} 
We study the average analytic rank in the family of $L$-functions $L(s, E_d)$ associated with the elliptic curves $E_d : y^2=x^3-dx$, as $d$ varies over fourth-power-free odd integers. Since this is a family of curves with complex multiplication, we have $L(s, E_d)=L(s - \frac12, \xi_d)$, where $\xi_d$ is a Hecke character over $\Z[i]$. Assuming the Generalized Riemann Hypothesis, we compute the one-level density of the low-lying zeros of this family for test functions whose Fourier transform is supported in $(-\frac35, \frac35)$. As a consequence, we obtain the upper bound $\frac{13}{6}$ for the average analytic rank $r(E_d)$ over the family. Under the additional assumption of a conjecture on the distribution of quartic Gauss sums at prime elements (a quartic analogue of Patterson's conjecture for cubic Gauss sums), we extend the admissible support to~$(-1, 1)$ and improve the upper bound for the average analytic rank to~$\frac32$.
Both results imply that a positive proportion of twists satisfy $r(E_d) =1$, while the second also yields a positive proportion of twists with $r(E_d)=0$.\end{abstract}

\maketitle

\section{Introduction}

Let $E$ be an elliptic curve over $\Q$ with $L$-function $L(s, E)$. The analytic rank of $E$, denoted by $r(E)$, is the order of vanishing of $L(s, E)$ at $s=1$. According to the Birch and Swinnerton-Dyer conjecture,  $r(E)$ is equal to the rank of the finitely generated abelian group $E(\Q)$. Given an elliptic curve $E: y^2 = x^3 + ax + b$, the curve $E'_d: dy^2 = x^3 + ax +b$ is called a \textit{quadratic twist} of $E$. The parity of $r(E'_{d})$ is determined by the root number, and as $d$ runs over the square-free integers, asymptotically half of the quadratic twists have root number $+1$ and half have root number $-1$. Goldfeld's \textit{minimalist} conjecture \cite{Goldfeld} predicts that the rank of almost all quadratic twists is as small as possible while respecting this parity restriction, i.e. that the average analytic rank across this family is~$\tfrac{1}{2}$.

Under the assumption of the Generalized Riemann Hypothesis (GRH), Goldfeld proved that the average rank is bounded from above by $\tfrac{13}{4}$. This bound was later improved by Heath-Brown \cite{HB}, who showed, again under GRH, that the average analytic rank is at most $\tfrac{3}{2}$ in each of the two subfamilies of quadratic twists with positive and negative root numbers, respectively. From this it follows that at least $\frac14$ of the quadratic twists with positive root number have analytic rank~$0$, and that at least $\frac34$ of those with negative root number have analytic rank~$1$. Function field analogues were established in~\cite{comeau-lapointe,Meisner-Sodergren}, while analogous results for quadratic twists of a general modular form of weight $k$ and level $1$ were obtained in~\cite{GZ-Compositio}. Heath-Brown's results in~\cite{HB} follow from proving that the Katz--Sarnak universality conjecture \cite{KSbook,Katz-Sarnak,SST} for the one-level density of the family of $L$-functions $\{L(s, E_{d}'): d \textnormal{ square-free}\}$ holds for test functions $\phi$ such that supp$(\widehat{\phi})\subseteq (-1, 1)$. Some lower-order terms for the one-level density were obtained in~\cite{FPS}.
A full proof of the universality conjecture would imply Goldfeld's minimalist conjecture. It would also follow from a very strong hypothesis (implying GRH) on the distribution of the zeros of the $L$-functions \cite{Fiorilli-2016}.

In a series of breakthrough papers, Smith \cite{SmithJAMS1, SmithJAMS2, Smith BSD implies Goldfeld} proved that in the family of quadratic twists the average $2^\infty$-Selmer rank (which bounds from above the average algebraic rank) is~$\frac12$. Assuming the Birch and Swinnerton-Dyer conjecture, this implies Goldfeld's conjecture. 
More recently, Burungale and Tian~\cite{Burungale-Tian} showed that half of the quadratic twists $d y^2 = x^3 - x$ of the congruent number elliptic curve have analytic rank~$0$, without using the Birch and Swinnerton-Dyer conjecture. Their argument combines a $p$-converse theorem with Smith's work.

For elliptic curves over $\Q$, there exist only two families of higher order twists up to isomorphism: the cubic twists $y^2 = x^3 + d$ (where $d \in \Z$ is cube-free) and the quartic twists $y^2=x^3 - dx$ (where $d \in \Z$ is fourth-power-free), both of which have {\it complex multiplication}. Their associated $L$-functions are $L$-functions of Hecke characters over $\Q(\sqrt{-3})$ and $\Q(\sqrt{-1})$, respectively. 
The analogue of Goldfeld's minimalist conjecture predicts that the average rank across each of these families is $ \tfrac{1}{2}$.
For the cubic family, Koymans and Smith \cite{Koymans-Smith} proved that the proportion of curves $y^2 = x^3 + dn^2$ (with $d$ fixed and $n$ ranging over the integers) with {\it algebraic rank}~0 is greater than $0.3195$. This improves upon the previous lower bound of~$\tfrac{1}{6}$ obtained by Alp\"oge, Bhargava and Shnidman~\cite{Al-Bh-Sc}, who also obtained results for rank~$1$ curves in this family. Assuming the Birch and Swinnerton-Dyer conjecture, these results imply the same respective lower bound for the proportion of curves of analytic rank~$0$.

For the quartic twists $y^2=x^3-dx$, Castillo, de Faveri and Dunn~\cite{CDD} computed the first two {\it mollified moments} of the associated $L$-functions when $d \in \Z[i]$ is square-free, and obtained a positive proportion of analytic rank 0, without any hypothesis. Their work is a quartic analogue of the results of~\cite{Sound-quad} for Dirichlet quadratic characters over~$\Q$, and of~\cite{DDDS} for Dirichlet cubic characters over $\Q(\sqrt{-3})$, since the $L$-functions quartic twists are Hecke $L$-functions over $\Z[i]$, and the averaging is performed over all square-free $d \in \Z[i]$. The results of~\cite{CDD} and~\cite{DDDS} were generalized to Hecke characters of general order $r$ in~\cite{DIPP}.

In this paper, we study the one-level density {of} the family of elliptic curves $E_d : y^2 = x^3 - dx$, when the parameter {ranges} over fourth-power-free $d \in \Z$, together with applications to average analytic rank and the proportion of non-vanishing at the central point. Since the $L$-functions $L(s, E_d)$ are given by  Hecke characters over $\Z[i]$, this leads to unbalanced sums between $\Z$ and $\Z[i]$.
All of our results are conditional on GRH.

\subsection{Statement of the main results}

For each odd integer $d$, let $E_d$ be the elliptic curve with affine equation $y^2=x^3 - dx.$ Writing $d = d_0^4 d_1$, {with} $d_1$ fourth-power-free, the conductor of $E_d$ is $N_{E_d} = 2^\ell \rad(d_1)^2$, where $\ell =5$ if $d_1 \equiv 1 \bmod 4$ and $\ell =6$ if $d_1 \equiv 3 \bmod 4$. The associated $L$-function is defined by 
\begin{align*} 
	L(s, E_d) &:= \prod_{p \nmid 2d} \left( 1 - \frac{a_p(E_d)}{p^s} + \frac{1}{p^{2s-1}} \right)^{-1}, \quad \text{for } \mathrm{Re}(s) > \tfrac{3}{2},
\end{align*}
where, for each prime $p \nmid 2d$, one sets $a_p(E_d):= p + 1 - \# E_d(\F_p) $.
The curves $E_d$ are called {\it quartic twists} of $E_1 : y^2 = x^3 - x$, and 
$L(s, E_d) = L(s-\frac{1}{2}, \xi_d)$, where $\xi_d$ is a {\it quartic Hecke character} over $\Z[i]$. 
Let $W(E_d) = \pm 1$ denote the root number of $E_d$, that is, the sign of the functional equation relating $L(E_d, s)$ to $L(E_d, 2-s)$. The value of $W(E_d)$ is explicitly computed in Lemma~\ref{lemma-sign-s1}. Let $r(E_d)$ denote the analytic rank of $E_d$. Throughout the paper, the term rank refers to the analytic rank unless specified otherwise.
Let also
\begin{align*}
	\cF &= \{ d \in \Z: d \text{ is odd and fourth-power-free} \} \\
	\cF^{\pm} &= \{ d \in \Z: d \text{ is odd, fourth-power-free, and } W(E_d) = \pm 1\} .
\end{align*}
For convenience, we write $\cF^*$ for any of the three families $\cF, \cF^+$, or $\cF^{-}$.

Let $w$ be a test function, not identically zero, satisfying the following conditions:
\begin{equation}\label{w-conditions}\begin{split}
	(i)\;& w \text{ is a non-negative Schwartz function and } w(0)=0,\\
	(ii)\;& \textstyle \text{the Fourier transform } x\mapsto \widehat{w}(x):= \int_{-\infty}^\infty w(t) e^{-2 \pi i t x} \d t \text{ is supported in } (-\eta,\eta).
\end{split}\end{equation}
We define the total weight
\begin{align*}
	\mathcal{S}_\mathcal{F^*}(w, D) &:= \sum_{\substack{d \in \mathcal{F^*}}} w \left( \frac{d}{D} \right).
\end{align*} 
We have, by Lemma~\ref{F-with-root-number},
$$
\mathcal{N}_{\cF}(D) := \# \{ 1 \leq d \leq D \;:\; d \in \cF \} \sim \frac{8}{15\zeta(4)} D,$$
and the density of $\cF^+$ or $\cF^-$ in $\cF$ is $\tfrac12$.
Then, 
\begin{align}\label{Eq total sum}
	\mathcal{S}_{\cF}(w, D) 
	=\int_{0}^{\infty}w \left( \frac{t}{D} \right)\d \mathcal{N}_\cF(t)
	+\int_{0}^{\infty}w \left( \frac{-t}{D} \right)\d \mathcal{N}_\cF(t)
	= \frac{8\widehat w(0)}{15\zeta(4)} D + o(D).
\end{align}
In particular, $\mathcal{S}_{\cF^*}(w,D)\asymp_w D$.

The main result of this paper is the following theorem.
\begin{theorem} \label{main-theorem} 
	Let $w$ satisfy the conditions in \eqref{w-conditions}, and assume that $L(s, E_d)$ satisfies the Riemann Hypothesis for all $d \in\cF^*$. Then 
	\begin{align*}
		\limsup_{D\to \infty}	\frac{1}{\mathcal{S}_\mathcal{F^*}(w, D)} \sum_{d \in \cF^{*}} r(E_d) w \left( \frac{d}{D} \right)
		\leq\frac{13}{6}.
 	\end{align*}
	Assuming also a Patterson-type conjecture on the distribution of quartic Gauss sums at prime elements (see Conjecture \ref{conj-DDHL}), we have 
	\begin{align*}
		\limsup_{D\to \infty}	\frac{1}{\mathcal{S}_\mathcal{F^*}(w, D)} \sum_{d \in \cF^{*}} r(E_d) w \left( \frac{d}{D} \right)
		\leq\frac{3}{2}.
 	\end{align*}
\end{theorem}

Our strategy to bound the average rank is to understand the behavior of the (properly normalized) low-lying zeros across the families $\cF^{*}$. Let $\phi$ be an even Schwartz function whose Fourier transform $\widehat{\phi}$ is compactly supported.
For any $d \in \mathcal{F}$ and $D >1$, we define
\begin{align}\label{13Jan.1}
	\mathcal{D}(\phi, E_d) :&= \sum_{L(1+ i \gamma, E_d)=0} \phi \left( \frac{\gamma \log{D}}{\pi} \right).
 \end{align}
Moreover, we introduce the following finite weighted averages:
\begin{align*}
	\mathscr{D}_{\cF^*}(\phi, w, D) :&= \frac{1}{\mathcal{S}_{\cF^*}(w, D) } 
	\sum_{d \in \mathcal{F^*} } \mathcal{D}(\phi, E_d)w\left(\frac{d}{D}\right) .
\end{align*}
The {\it one-level density} is then defined to be the limit of $\mathscr{D}_{\cF^*}(\phi, w, D)$ as $D\to\infty$. 
We remark that we approximate the conductor of each $L(s, E_d)$ by $D^2$ when normalizing $\gamma$. This is justified by Lemma \ref{Lem Conductor condition}, which shows that 
computing the one-level density with this approximation, or with the true conductor, yield the same limit.

Theorem~\ref{main-theorem} follows from computing the one-level density across the families $\cF, \cF^{+}$ and~$\cF^{-}$.
\begin{theorem} \label{main-theorem OLD} 
	Let $w$ satisfy the conditions in \eqref{w-conditions}, and let $\phi$ be an even Schwartz function such that the support of $\widehat{\phi}$ is contained in $(- \nu, \nu)$. 
	Assume that $L(s, E_d)$ satisfies the Riemann Hypothesis for all $d \in \cF^*$.
	Then, as $D \rightarrow \infty$,
	\begin{align*}
		\mathscr{D}_{\cF^*}(\phi, w,D)
		=\widehat{\phi}(0) + \frac{1}{2} \int_\R \widehat{\phi}(u) \d u+ O \left( \frac{1}{\log{D}} \right),
	\end{align*}
	provided that $\nu < \frac35$. \\
	Assuming also Conjecture~\ref{conj-DDHL}, the support condition can be relaxed to~$\nu < 1$.
\end{theorem}

From the Katz--Sarnak conjectures \cite{Katz-Sarnak, KSbook, SST}, it is expected that the results of Theorem~\ref{main-theorem OLD} hold for any {even Schwartz} test function~$\phi$. This would imply that the average analytic rank is~$\frac12$.
The upper bound on the average rank given in Theorem~\ref{main-theorem} is obtained by choosing an appropriate and admissible test function $\phi$ in Theorem~\ref{main-theorem OLD}.
When the resulting upper bound for the average rank is sufficiently small, we obtain lower bounds for the proportion of elliptic curves in the family with minimal rank.
\begin{corollary} \label{Corollary non vanishing}
	Assuming the hypotheses of Theorem \ref{main-theorem OLD},
	\begin{align*}
		\liminf_{D\rightarrow \infty} \frac{\#\lbrace d \in \cF^-\cap [1,D] : r(E_d) =1 \rbrace }{\#(\cF^- \cap [1,D])}
		\geq \frac5{12}.
	\end{align*}
	Assuming also Conjecture \ref{conj-DDHL}, we have
	\begin{align*}
		&\liminf_{D\rightarrow \infty} \frac{\#\lbrace d \in \cF^-\cap [1,D]: r(E_d) =1 \rbrace }{\#(\cF^- \cap [1,D])} 
		\geq \frac34, \\
		&\liminf_{D\rightarrow \infty} \frac{\#\lbrace d \in \cF^+\cap [1,D]: r(E_d) = 0 \rbrace }{\#(\cF^+ \cap [1,D])} 
		\geq \frac14.
	\end{align*}
\end{corollary}

There is a lengthy history of one-level density computations with applications to average rank and non-vanishing in families of $L$-functions. We concentrate here on some families of $L$-functions attached to elliptic curves. The two-parameter family of all elliptic curves $y^2 = x^3 + a x + b,$ with $a, b \in \Z$, for example, was studied by Young~\cite{Young}, who used one-level density computations to show that, under GRH, the average rank across such curves is bounded from above by $25/14$. This improved upon the upper bound 2.3 obtained by Brumer \cite{Brumer} and of 2 obtained by Heath-Brown ~\cite{HB-AR} (both under GRH). Averages over special {two}-parameters families were also considered by Young~\cite{Young} and Cho and Jeong~\cite{CJ-2023}, and over one-parameter families by Miller~\cite{Miller2004}. Phillips \cite{Phillips-2025} considered {two}-parameters families over number fields of degree $k$, and showed that under GRH, the average rank is bounded by $(9k+1)/2$.

\subsection{Sketch of proof of Theorem \ref{main-theorem OLD}}

Since $E_d$ is a CM curve, we may write $L(s, E_d) = L(s- \frac12, \xi_d)$, where $\xi_d$ is a quartic Hecke character over $\Z[i]$ of conductor $\mathfrak{f}_d$ (see Section \ref{section-EF}).
Applying the explicit formula for $L(s, \xi_d)$, we write
\begin{align} \label{also3.9}
	\mathcal{D}(\phi, E_d) = U_{\infty}(\phi, \xi_d) + U_{\text{inert}}(\phi, \xi_d) + U_{\text{split}}(\phi, \xi_d),
\end{align}
where $U_{\infty}(\phi, \xi_d), U_{\text{inert}}(\phi, \xi_d), U_{\text{split}}(\phi, \xi_d)$ are defined in \eqref{define-U}.
Averaging the first two terms yields
\begin{align*}
	\frac{1}{\mathcal{S}_{\cF^*}(w, D)} \sum_{d \in \mathcal{F^*}} \left( U_{\infty}(\phi, \xi_d) + U_{\text{inert}}(\phi, \xi_d) \right) w(\tfrac{d}{D}) = \widehat{\phi}(0) + \frac{1}{2} \int_\R \widehat{\phi}(u) \; \d u+ O \left( \frac{1}{\log{D}} \right).
 \end{align*}
In particular, their average is the same over each of the three families $\cF, \cF^+$, and $\cF^-$. Therefore, to prove Theorem \ref{main-theorem OLD}, it suffices to show that the average of
\begin{align} \label{also-rem3}
	U_{\text{split}}(\phi, \xi_d) 
	&= - \frac{1}{\log{D}}  \sum_{\substack{p \equiv 1 \bmod 4\\ p = \pi \overline{\pi} \\ k \geq 1}}  \frac{\log{p}}{p^{\frac{k}{2}}} 
	\big( \xi_{d}^k (({\pi})) + \overline{\xi_d}^k((\pi)) \big) \widehat{\phi} \left( \frac{k \log{p}}{2 \log{D}} \right),
\end{align}
over $d \in \cF^*$, weighted by $w(\tfrac{d}{D})$, is~$O(\tfrac{1}{\log D})$ as~$D \rightarrow \infty$.

Our first step is to approximate the sum over $d \in \cF^*$ with a sum over $d \in \Z$, using GRH (Lemma \ref{Lemma sieve before Poisson}). The desired estimate then follows from the P\'olya--Vinogradov inequality whenever the support of $\widehat{\phi}$ is contained in $(- \frac12, \frac12)$, see Remark~\ref{Remark TO Prove 2}. To enlarge the admissible support, we flip the two summations and apply Poisson summation to the $d$-sum
\begin{align}
	\sum_{d \in \Z} \quartic{d}{\pi^k} w\left(\frac{d}{D}\right).
\end{align}
The resulting summation involves \textit{quartic Gauss sums}, $g_{4}(r,\pi^k)$ (defined in \eqref{Def-Gauss-sum}). Interchanging the order of summation once more, we are led to estimate averages of quartic Gauss sums at prime-power arguments. More precisely, we seek to bound the sums
\begin{align} \label{sum-over-primes}
	H_\beta(X, Y, r) := \sum_{\substack{c \in \Z[i] \\ c \equiv \beta \bmod 4\\ (r,c)=1}}  \frac{g_4(r, c)}{\N(c)^{\frac{3}{2}}} \Lambda(c) \frac{\overline{c}}{|c|} R_{X, Y}(\N(c)),
\end{align} 
where $\beta \in \{ 1 , 1 + \lambda^3 \}, r \in \Z[i]$, and the test function $R_{X, Y}$ is defined by \eqref{DefR_X}. A serious challenge in computing such estimates is that the Gauss sums are not multiplicative and their associated $L$-functions have no Euler product decomposition. 

The following conjecture is a slight generalization of \cite[Conjecture 1.2]{DDHL}, which makes explicit the dependence on the conductor $r$. 
The corresponding conjecture for cubic Gauss sums was originally formulated by Patterson \cite{Patterson78}, and later refined by Patterson and Heath-Brown \cite{HBP}. A smooth version of the original Patterson conjecture for cubic Gauss sums (namely, the asymptotic for $\ell=0$ with the correct leading constant, but without an explicit error term) was recently proved by Dunn and Radziwi\l\l~\cite{DR}, assuming GRH. 
In \cite{DDH}, some of the results of \cite{DR} about lower bounds for the cubic large sieve are obtained unconditionally.

\begin{conjecture}[Conjecture 1.2 from \cite{DDHL}] \label{conj-DDHL}
	For $\beta \in \{ 1, 1+ \lambda^3 \}$ and $r\in \mathbb Z[i]$, there exist constants $b_{r,\beta}$ such that for any $\varepsilon >0$ and $\ell \in \Z$, we have
	\begin{align*}
		\sum_{\substack{c \in \Z[i] \\ c \equiv \beta \bmod 4\\ \N(c) \leq X}} \frac{g_4(r, c)}{\N(c)^{\frac{1}{2}}} \Lambda(c) \left( 				\frac{\overline{c}}{|c|} \right)^\ell =
		\begin{cases}  
			b_{r,\beta} X^{\frac{3}{4}} + 
			O \left( \N(r)^\varepsilon X^{\frac{1}{2}+\varepsilon} \right) & \ell 
			= 0    \\ O \left(\N(r)^\varepsilon X^{\frac{1}{2}+\varepsilon} \right) & \ell \neq 0 
		\end{cases}
	\end{align*}
as $X \rightarrow \infty$.
\end{conjecture}

The role of Conjecture~\ref{conj-DDHL} in our setting is analogous to that of the Riemann Hypothesis in Heath-Brown's work~\cite{HB-AR} concerning quadratic twists. For quadratic twists, the sums arising after Poisson summation involve {\it quadratic} Gauss sums, which are highly regular and in fact essentially constant (when properly normalized). The resulting sums are essentially equivalent to those which appear before Poisson summation (albeit in a different range of summation). They can then be bounded using the Riemann Hypothesis for the $L$-functions attached to the quadratic twists of the elliptic curve.

Most of the work of this paper consists of bounding the sums $H_\beta(X, Y, r)$ without using Conjecture \ref{conj-DDHL}.
We first use Vaughan's identity (in the form given in~\cite{HB2000}) to decompose the sum~\eqref{sum-over-primes} into several sums over integers in different ranges. Our goal is then to evaluate sums of the form
\begin{align*} 
\Sigma_{j,\beta}(X,Y,r,U):=\sum_{\substack{c \in \Z[i] \\ c \equiv \beta \bmod 4\\ (r,c) =1}} \frac{\overline{c}}{|c|} \; \frac{g_4(r, c)}{\N(c)^{\frac{3}{2}}} \; a_{j}(c,U)  \; R_{X,Y} (\N(c))
\end{align*}
where the $a_{j}(c,U)$ coefficients are sums over divisors of $c$ in certain ranges determined by the parameter~$U$. Depending on these ranges, the resulting sums fall into two classes, which are referred to as Type~I and Type~II sums. To bound Type~I sums, we prove a {\it Lindel\"of-on-average} result across a family of metaplectic $L$-functions (see Proposition \ref{prop-Lindelof-on-average}). As the average runs over squares, the desired result is obtained by applying the \textit{quadratic} large sieve. In \cite{DDH}, some related Lindel\"of-on-average results are proven by a different technique, which involves a Rankin-Selberg method. To bound Type~II sums, we factor the quartic Gauss sums and exploit oscillations of the resulting quadratic characters by again using the quadratic large sieve. The limitation on the support of $\widehat{\phi}$ in Theorem \ref{main-theorem OLD} then comes from the Type~II sums (see Remark \ref{Type_II_contr}). 

This approach for bounding sums of Gauss sums was first used by Heath-Brown and Patterson \cite{HBP} to prove that cubic Gauss sums at prime arguments are equidistributed, and in \cite{P-general} to prove the equidistribution of general Gauss sums at prime arguments. It was also used to compute the one-level density of cubic and quartic families of Dirichlet $L$-functions, with application to non-vanishing: in \cite{DG} (resp. \cite{GZ-Canadian-2020}), the authors prove that under GRH, there is a positive density of cubic (resp. quartic) Dirichlet characters $\xi$ over $\Z[i]$ such that $L(\frac12, \xi) \neq 0$; in \cite{DFL-general-r}, the authors generalize this result to characters $\xi_r$ of any order $r \geq 3$ over $\F_q[t]$.

In our application, we apply Vaughan's identity directly to the smooth sums $H_\beta(X, Y, r)$ rather than the sharp summation appearing in Conjecture \ref{conj-DDHL} and in the previous literature.

\subsection{Structure of the paper}

In Section~\ref{Section-background}, we review the necessary background on quartic characters and quartic Gauss sums.
Next, in Section~\ref{section-EF}, we describe the Hecke character $\xi_d$ such that $L(s, E_d) = L(s - \frac12, \xi_d)$, and use the explicit formula for $L(s, \xi_d)$ to derive the formula \eqref{also3.9} for the one-level density.
At the end of Section~\ref{section-EF}, the proof of Theorem~\ref{main-theorem OLD} is reduced to bounding the prime sum \eqref{also-rem3}.
Section~\ref{Section Poisson} removes the fourth-power-free condition by a sieving argument that reduces sums over the family $d \in \cF_{a, b}$ defined in \eqref{definition-Fab} to sums over all integers $d \in \Z$. We can then apply Poisson summation to the $d$-sum in Lemma~\ref{Poisson}; this yields an $m$-sum of shifted Gauss sums $g_4(m, \pi)$. 
Upon flipping the sums, the problem boils down to bounding the double sum of Remark~\ref{Proposition_reduction}. Finally, we show that, assuming Conjecture~\ref{conj-DDHL}, the required bound holds for $\nu < 1$, thereby proving the second part of Theorem~\ref{main-theorem OLD}.
Section~\ref{section-main} states the key bound for $H_\beta(X, Y, r)$ (see Proposition~\ref{BoundForHAfterLindelof}), and shows how it implies Theorem~\ref{main-theorem OLD} and Corollary~\ref{Corollary non vanishing} without assuming Conjecture~\ref{conj-DDHL}.

The rest of the paper is devoted to the proof of Proposition~\ref{BoundForHAfterLindelof}.
We use Vaughan's identity in Section~\ref{Vaughan} to decompose the double sum into Type~I and Type~II sums involving the test function $R_{X, Y}(t)$ defined in~\eqref{DefR_X}.
The Type~II sums are bounded in Proposition~\ref{TypeII_bound}, while the Type~I sums are bounded in terms of an integral involving Dirichlet series of Gauss sums in Lemmas~\ref{lemma_S_1} and~\ref{Vaughan_bounds}. The main focus of the subsequent sections is to study the analytic properties of these Dirichlet series, in particular to bound them in terms of the conductor in the critical strip. The known convexity bound is not enough to make the Type~I sums as small as the Type~II sums, but the {\it Lindel\"of-on-average} bound of Proposition~\ref{prop-Lindelof-on-average} is. Assuming Proposition~\ref{prop-Lindelof-on-average}, we prove Proposition~\ref{BoundForHAfterLindelof} at the end of Section~\ref{after-Vaughan}. Section~\ref{section-LOA} contains the proof of Proposition~\ref{prop-Lindelof-on-average}, and Section~\ref{root-number} contains the proof of Lemma~\ref{lemma-sign-s1}, which gives the sign of the functional equation for the $L$-functions $L(s, \xi_d)$.

\subsection*{Acknowledgments} 
The authors wish to thank Alex Dunn and Alex de Faveri for useful comments on previous versions of this manuscript. 
This project was completed when LD was visiting the CRM in Montréal thanks to a délégation CNRS and during the special semester ``Universal Statistics in Number Theory'' and she thanks them for their hospitality. 
CD is supported by the NSERC (RGPIN-2019-05536) and the FRQNT (Projets de recherche en équipe 300951).
AF is a member of the INdAM group GNAMPA, and part of this work was completed while he was supported by the FRQNT (Projet de recherche en \'equipe 300951). EW was supported by the Israel Science Foundation (Grant No. 1881/20), and by a Chateaubriand Fellowship through the French Embassy in Israel, which supported an extended research visit to the LMPA, Université du Littoral Côte d’Opale, hosted by LD.

\section{Quartic characters and quartic Gauss sums} \label{Section-background}

\subsection{Quartic Residue Symbol} 

Let $\Z[i]$ be the ring of Gaussian integers whose discriminant is~$-4$. All ideals of $\Z[i]$ are principal, and its unit group is $\{\pm1,\pm i\}$. Throughout the paper, we denote by $\lambda :=1+i$ a chosen generator of the unique ramified prime ideal of $\Q(i)$ lying above 2. Any Gaussian integer $n$ may be uniquely written as $n = u \lambda^k m$, where $u$ is a unit, $k \geq 0$ and $m \equiv 1 \bmod \lambda^3$; equivalently, $m$ is {\it primary}. Every primary Gaussian integer $m$ 
satisfies either $m \equiv 1 \bmod 4$ or $m \equiv 1 + \lambda^3 \bmod 4$.

For any Gaussian prime $\pi \in \Z[i]$ such that $(\pi,\lambda)=1$, we define the \textit{quartic residue symbol} modulo $\pi$ as the quartic character $\chi_{(\pi)}: \left( \Z[i]/(\pi) \right)^\times \longrightarrow \{ \pm 1, \pm i \}$ such that
\begin{align}\label{quartic_symbol}
	\chi_{(\pi)}(\alpha) := \left( \frac{\alpha}{\pi} \right)_4 \equiv \alpha^{(\N(\pi)-1)/4} \bmod \pi,
\end{align}
and extended to $\Z[i]$ by $\chi_{(\pi)}(\alpha)  =0$ when $(\alpha,\pi) \neq 1$. We note, in particular, that $\left( \frac{\alpha}{\pi} \right)_4 =1$ if and only if $\alpha$ is a fourth power in $\Z[i]/(\pi)$. Likewise, $\left( \frac{\alpha}{\pi} \right)_4 =-1$ if and only if $\alpha$ is a square but not a fourt -power in $\Z[i]/(\pi)$. Moreover, $\chi_{(\pi)}(\alpha)$ is independent of the chosen generator of $(\pi)$.

Consider a non-unit $\beta \in \Z[i]$, $(\beta, \lambda)=1$, with prime factorization given by $\beta=  u \pi_1^{e_1} \dots \pi_s^{e_s}$. We define the quartic residue symbol $\chi_{(\beta)}$ by multiplicativity, for any $\alpha \in \Z[i]$,
$$ \chi_{(\beta)}(\alpha) := \left( \frac{\alpha}{\beta} \right)_4 = \prod_{i=1}^s \left( \frac{\alpha}{\pi_i} \right)_4^{e_i}.$$
If $\beta \in \Z[i]$  is fourth-power-free, then $\chi_{(\beta)}$ is a primitive character of conductor $(\rad(\beta))$, the product of primes dividing $\beta$ without multiplicity.

For  $\alpha,\gamma \in \mathbb{Z}[i]$ non-units with $\alpha, \gamma$ primary and $(\alpha,\gamma)=1$, 
the quartic reciprocity law  \cite[Theorem~6.9]{Lemmermeyer} states that
\begin{equation} \label{bilaw}
	\Big( \frac{\alpha}{\gamma} \Big)_4=(-1)^{C(\alpha,\gamma)} \Big( \frac{\gamma}{\alpha} \Big)_4,
\end{equation}
where 
\begin{equation} \label{Cdef}
	C(\alpha,\gamma)=\frac{\N(\alpha)-1}{4} \frac{\N(\gamma)-1}{4}.
\end{equation}
Observe that 
$(-1)^{C(\alpha,\gamma)}$ depends only on $\alpha,  \gamma \bmod{4}$.
There are also supplementary laws for the ramified prime and units. If  $\alpha = a + 2b i$ is primary
we have
\begin{equation}\label{Eq suppl quartic law}
	\quartic{i}{\alpha} 
	= i^{\frac{1-a}{2}}, \quad \quartic{\lambda}{\alpha} = i^{\frac{a - 2b -4b^2 -1}{4}}, \quad \text{ and } \quartic{2}{\alpha} 
	= i^{-b}.
\end{equation}	
Finally, by \cite[Prop. 9.8.5]{IreRos}, for odd $d \in \mathbb{N}$ and $n \in \Z_{\neq 0}$ such that $(n,d) = 1$,
\begin{equation}\label{chi_at_integers}
	\chi_{(d)}(n) =1.
\end{equation}

\subsection{Quartic Gauss sums}

For $c,r \in \Z[i]$ with $c \equiv 1 \bmod \lambda^3$, using the standard notation $e(x) := e^{2 \pi i x}$, we define the quartic Gauss sum as 
\begin{align}\label{Def-Gauss-sum}
	g_4(r,c) := \sum_{\alpha \bmod c} \quartic{\alpha}{c} e\big(\Tr\big(\tfrac{\alpha r}{c}\big)\big).
\end{align}
If $(r,c)=1$, the change of variable $\beta = \alpha r$ yields
\begin{equation} \label{rel1}
	g_4(r v,c)=\oquartic{r}{c} g_4(v,c).
\end{equation}
As shown in \cite[(3.12)]{DDHL}, we moreover have that
\begin{equation}\label{gauss_sum_squarefree}
	|g_{4}(1,c)|=\mu^{2}(c)\N(c)^{\frac{1}{2}}.
\end{equation}
In particular, by setting $v = 1$ in \eqref{rel1} it follows that for $(r,c) =1$,
\begin{equation}\label{gauss_sum_bound}
	|g_{4}(r,c)| \leq \mu^2(c) \N(c)^{\frac{1}{2}}.
\end{equation}
For $v \in \Z[i]$, $c,c^{\prime} \in \Z[i]$ with $(cc^{\prime},\lambda)=1$ and $(c,c^{\prime})=1$, 
the Chinese Remainder Theorem implies the twisted multiplicativity relation,
\begin{align} \label{rel2}
	g_4(v,cc^{\prime})&= \quartic{c}{c^{\prime}} \quartic{c^{\prime}}{c}  g_4(v,c) g_4(v,c^{\prime}) \\
	&= (-1)^{C(c,c')}g_{4}(v,c)g_{4}(c^{2}v,c'), \nonumber
\end{align}
where the last equality follows from 
\eqref{bilaw} and \eqref{rel1}.  

By \eqref{rel1} and \eqref{rel2} it suffices to understand $g_4(\pi^n,\pi^k)$ for primes
$\pi \equiv 1 \bmod{\lambda^3}$ and $k,n \in \mathbb{N}$.
As in \cite[(3.7)]{DDHL} we will use the following:
\begin{equation} \label{rel4}
	g_4(\pi^n,\pi^k)=\begin{cases}
	\N(\pi)^{k-1} g_4(1,\pi) & \text{if } n=k-1, \quad k \equiv 1 \bmod{4},\\
	\N(\pi)^{k-1} g_2(1,\pi) & \text{if } n=k-1, \quad k \equiv 2 \bmod{4}, \\ 
	\N(\pi)^{k-1} \quartic{-1}{\pi} \overline{g_4(1,\pi)} & \text{if } n=k-1, \quad k \equiv 3 \bmod{4}, \\
	-\N(\pi)^{k-1} & \text{if } n=k-1, \quad k \equiv 0 \bmod{4}, \\
	\varphi(\pi^{k}) & \text{if } n \geq k, \quad k \equiv 0 \bmod{4}, \\
	0 & \text{otherwise},
\end{cases}
\end{equation}
where
\begin{equation*}
	g_2(v,c):=\sum_{d \ppmod{c}} \Big( \frac{d}{c} \Big)_2 e \big(\Tr\big(\tfrac{v d}{c}\big) \big),
\end{equation*}
is the (auxiliary) quadratic Gauss sum. In particular, the relations~\eqref{rel1} and~\eqref{rel2} remain valid under weaker conditions whenever both sides of the equations vanish.

\section{Properties of the $L$-function $L(s, \xi_d)$}\label{section-EF}

Our motivation for studying quartic Hecke characters is the following.
Consider the curve $E_d: y^2=x^3-dx$ where $d$ is fourth-power-free. 
As in \cite[Thm.~7 in Ch.~18]{IreRos}, its associated~$L$-function is $L(s, E_d) = L(s-\tfrac12, \xi_d)$ where $\xi_d$ is the\footnote{In fact, there are two such Hecke characters, $\xi_d$ and $\overline{\xi_d}$, whose associated $L$-functions coincide.} Hecke character over $\Z[i]$ defined on prime ideals by
\begin{equation}\label{Def_xi_d}
	\xi_d : \mathfrak{p} \mapsto 
	\begin{cases}
		\oquartic{d}{\pi}	\frac{\pi}{|\pi|} &\text{ if } \mathfrak{p} = (\pi) \text{ with } \pi \equiv 1 \bmod \lambda^3, \\
		0 &\text{ if } \mathfrak{p} = (\lambda),
	\end{cases} 
\end{equation}
and extended by multiplicativity. 
Note that if $\N(\mathfrak{p})=p^2$, and $(p,d) =1$, then $\xi_d(\mathfrak{p}) = -1.$ 
Thus,
\begin{align*}
	L(s,E_d)
	&= L(s-{\textstyle{\frac12}}, \xi_d)
	= \prod_{\mathfrak{p}} \left( 1 -  \frac{\xi_d(\mathfrak{p})}{\N(\mathfrak{p})^{s-\frac12}} \right)^{-1} \\
	&=  \prod_{\substack{p \equiv 1 \bmod 4\\ (p) = \p\overline{\p} }} 
	\left( 1 -  \frac{\sqrt{p} \,\xi_d(\p)}{p^s} \right)^{-1} \left( 1 -  \frac{{\sqrt{p} \, \overline{\xi_d}(\p)}}{p^s} \right)^{-1}
	\prod_{\substack{p \equiv 3 \bmod 4 \\ p\nmid d}} \left( 1 +  \frac{p}{p^{2s}} \right)^{-1} .
\end{align*}
We will study the one-level density of the family of $L$-functions $L(s, \xi_d)$ at $s=\frac12$, which corresponds to $L(s,E_d)$ at $s=1$.

\subsection{Conductor and sign of the functional equation}
 
The parameter~$d$ runs over the set of odd fourth-power-free integers. One can adapt the proof of \cite[Lem.~2.3]{DDW} to fourth-power-free $d$ by counting the contribution of prime factors with multiplicities.
Upon doing so, we have that the conductor of $\xi_d$ is given by
\begin{align} \label{conductor}
	\mathfrak{f}_d := \text{cond}(\xi_d) = 
	\begin{cases}
		(\lambda^3\rad(d)), & \text{if } d \equiv 1 \bmod 4 \\
 		(4\rad(d)), & \text{if } d \equiv 3 \bmod 4,
 	\end{cases}
\end{align}
where $\rad(d) := \prod_{p\mid d} p$ denotes the radical of $d$. 
The completed~$L$ function is defined by (see e.g. \cite[(2.24)]{DDW})
$$\Lambda(s, \xi_d) := (4 \N(\mathfrak{f}_d))^{\frac{s}{2}} (2 \pi)^{-s} \Gamma\big(s+\tfrac12 \big) L(s, \xi_d). $$ 
It satisfies the functional equation
\begin{align*}
	\Lambda(s, \xi_d) &=  W(\xi_d)  \Lambda(1-s, {\xi_d}) 
\end{align*}
where $W(\xi_d)$ denotes the root number of $L(s, \xi_d)$. 
 
The value of the root number depends on the congruence classes of both $d$ and $\rad(d)$. The following generalizes the result of Birch and Stephens~\cite{BirchStephens} to $d\equiv 3 \bmod 4$ and \cite[Lem.~2.5]{DDW} to fourth-power-free integers.	 

\begin{lemma} \label{lemma-sign-s1} 
	Let $d$ be an odd fourth-power-free integer, then the root number of $\xi_d$ is
	\begin{align} \label{d=1mod4-s1}
		W(\xi_d) = (-1)^{\frac{\rad(d)-1}{2}} 
		\begin{cases} 
			1 & \textnormal{ if } d \equiv 1,7,13,15 \bmod {16}\\
			-1 & \textnormal{ if } d \equiv 3,5,9,11 \bmod {16}.
		\end{cases}
\end{align}
\end{lemma}

For ease of exposition, we delay the proof to Section~\ref{root-number}.
This motivates the following notation. Let $a,b$ be positive odd integers, and let 
\begin{align}\label{definition-Fab}
	\mathcal{F}_{a,b} &= \{ d \in \cF \;:\; d \equiv a \bmod 16, \;\rad(d) \equiv b \bmod 4\}.
\end{align}
We define
\begin{align*}
	\mathcal{S}_{\mathcal F_{a,b}}(w, D)  &= \sum_{\substack{d \in \mathcal{F}_{a,b}}} w \left( \frac{d}{D} \right), \\
	\mathscr{D}_{a,b}(\phi, w,D) &= \frac{1}{\mathcal{S}_{\cF_{a,b}}(w, D) } 
	\sum_{d \in \mathcal{F}_{a,b}} \mathcal{D}(\phi, E_d) w\bigg(\frac{d}{D}\bigg). 
\end{align*}
Our goal is to prove that
\begin{align} \label{prove-that}
	\mathscr{D}_{a,b}(\phi, w,D) 
	=\widehat{\phi}(0) + \frac{1}{2} \int_\R \widehat{\phi}(u) \; \d u + O \bigg( \frac{1}{\log{D}} \bigg)
\end{align}
under the conditions of Theorem \ref{main-theorem OLD}.

We begin with an estimate for the size of the families.
\begin{lemma} \label{F-with-root-number} Let $a,b$ be positive odd integers, one has
$$\mathcal{N}_{\cF_{a,b}}(D) := \# \{ 1 \leq d \leq D \;:\; d \in \cF_{a,b} \} \sim \frac{8}{15 \zeta(4)} \frac{D}{16}  \left( 1 + \chi_4(ab) P_3 \right),$$
as $D\to \infty$,
where $\chi_4$ is the non-trivial Dirichlet character modulo~$4$ and $P_3$ is the convergent Euler product
\begin{align*}
 P_3 := \prod_{p \equiv 3 \bmod 4} \Big( 1 - \frac{2}{p^2(1 + \frac{1}{p} + \frac{1}{p^2} + \frac{1}{p^3})}\Big). 
\end{align*}
In particular, it follows that
\begin{align*}
\mathcal{N}_{\cF}(D) \sim \frac{8}{15 \zeta(4)} D \quad \text{and} \quad \mathcal{N}_{\cF^{\pm}}(D) \sim \frac{4}{15 \zeta(4)} D.
\end{align*}
\end{lemma}
\begin{proof}
For $0 \leq i \leq 7$, let $\chi_i$ be the Dirichlet characters modulo $16$, where $\chi_0$ is the principal character, and  $\chi_4$ is the non-trivial character modulo 4. 
Using orthogonality relations and applying the Wiener--Ikehara Tauberian theorem to pick up the residue at $s=1$, we write
\begin{align*}
\mathcal{N}_{\cF_{a,b}}(D)  &= \sum_{\substack{1 \leq d \leq D\\\text{$d$ odd}\\ \text{fourth-power-free}}}\frac{1}{2}  \big(1 + \chi_4(b^{-1} \rad(d))\big) \Big(\frac{1}{8} \sum_{i=0}^{7} \chi_i(a^{-1} d) \Big)\\
&\sim \frac{D}{16}  \mathrm{Res}_{s=1} \bigg( \sum_{\substack{\text{$d$ odd}\\ \text{fourth-power-free}}} \big(1 + \chi_4(b^{-1} \rad(d))\big) \Big( \sum_{i=0}^{7} \chi_i(a^{-1} d) \Big) d^{-s} \bigg).
\end{align*}
Writing the generating series above as a linear combination of $16$ Dirichlet series, observe that only 2 of those series have a pole at $s =1$ and thus contribute to the count, namely
\begin{equation*}
L_0(s):=\sum_{\substack{\text{$d$ odd}\\ \text{fourth-power-free}}} \chi_0(d) d^{-s} = \prod_{p\neq 2} \left( 1 + \frac{1}{p^s} + \frac{1}{p^{2s}} + \frac{1}{p^{3s}} \right)=\frac{\zeta(s)}{\zeta(4s)}\left( 1 + \frac{1}{2^s} + \frac{1}{2^{2s}} + \frac{1}{2^{3s}} \right)^{-1}
\end{equation*}

\begin{align*}
\text{and} \quad L_4(s) &:= \sum_{\substack{\text{$d$ odd}\\ \text{fourth-power-free}}} \chi_4(\rad(d)) \chi_4(d) d^{-s}  = \prod_{p\neq 2} \left( 1 + \frac{\chi_4(p^2)}{p^s} + \frac{\chi_4(p^3)}{p^{2s}} + \frac{\chi_4(p^4)}{p^{3s}} \right).\\
\end{align*}
We deduce that
\begin{align*}
\mathcal{N}_{\cF_{a,b}}(D)  &\sim \frac{D}{16} \big(  \mathrm{Res}_{s=1} L_0(s) +   \chi_4(ab)  \mathrm{Res}_{s=1} L_4(s) \big) \\
&=  \frac{8}{15 \zeta(4)} \frac{D}{16}  \left( 1 + \chi_4(ab) \prod_{p \equiv 3 \bmod 4} \frac{ 1 + \frac{1}{p} - \frac{1}{p^2} + \frac{1}{p^3} }{ 1 + \frac{1}{p} + \frac{1}{p^2} + \frac{1}{p^3}} \right).
\end{align*}
The count for $\mathcal{N}_{\cF^\pm}(D)$ follow immediately from Lemma~\ref{lemma-sign-s1}.
\end{proof}

Let us then show that on average the logarithm of the conductor is $\log D^2$, which accounts for the normalization in \eqref{13Jan.1}. This result is referred to as the \lq\lq conductor condition\rq\rq\, in the work of Young \cite{Young}.

\begin{lemma}[Conductor condition]\label{Lem Conductor condition}
	Let $w$ be a non-zero Schwartz function with $w(0)=0$.
	For positive odd integers $a,b$ we have as $D$ tends to infinity, 
	\begin{equation}\label{ConductorCondition} 
		\frac{1}{\mathcal{S}_{\mathcal F_{a,b}} (w, D)} 
		\sum_{d\in \cF_{a,b}} w\left(\frac{d}{D}\right)\frac{\log \N(\mathfrak f_d)}{\log D^2} 
		= 1+O\bigg(\frac{1}{\log D}\bigg). 
	\end{equation}
\end{lemma}

\begin{proof}
	The desired asymptotic formula can be proven by direct computation. 
	The main idea is to compare the logarithms of $\mathfrak f_d$ and $d^2$ by writing
	$$\log \N(\mathfrak{f}_d) = \log D^2 
	+ \log (d^2/D^2)  - \log (d^2/\N(\mathfrak{f}_d)).$$
	The first term on the right-hand side gives the main term.
	For the second term, since $w$ is Schwartz and $w(0) =0$, the function $t \mapsto w(t) \log(t^{2}) $ is also Schwartz and vanishing at 0, so, as in~\eqref{Eq total sum}, one has
	$$ \sum_{d\in \mathcal F_{a,b}} w\left(\frac{d}{D}\right) \log (d^2/D^2) 
	\ll_{\cF_{a,b},w} D \ll_{\cF_{a,b},w} \mathcal{S}_{\mathcal F_{a,b}} (w, D).  $$ 	
	
	For the last term, employing \eqref{conductor}, we write $\N(\mathfrak{f}_d) = 2^\ell \rad(d)^2$. 
	Since $\N(\mathfrak f_d)$ and $d^2$ have the same odd prime divisors, we get
	$$\log (d^2/\N(\mathfrak f_d)) = \log \bigg( 2^{-\ell}\prod_{\substack{p^\alpha||d \\  \alpha\geq 1}} p^{2\alpha-2} \bigg)
	=  \sum_{\substack{p^\alpha||d  \\ \alpha \geq 2}} \log p^{2\alpha-2} +O(1).$$
	Summing over fourth-power-free $d\in \cF_{a,b}$, we then have
	\begin{equation}\begin{split}\label{ConductorCondition2}\notag
			\sum_{d\in \cF_{a,b}} w\left(\frac{d}{D}\right) & \log (d^2/\N(\mathfrak f_d)) 
			=  \sum_{d\in \mathcal F_{a,b}} w\left(\frac{d}{D}\right)  
			\sum_{\substack{p^\alpha||d \\ \alpha \in \{2,3\}}} \log p^{2\alpha-2} 
			+ O(\mathcal{S}_{\mathcal F_{a,b}}(w,D)).
	\end{split}\end{equation}
	We then bound the first term above as follows:
	\begin{equation}\begin{split}\notag
			\sum_{d\in \mathcal F_{a,b}} w\left(\frac{d}{D}\right)  
			\sum_{\substack{p^\alpha||d \\  \alpha \in \{2,3\}}} \log p^{2\alpha-2} 
			&\ll \sum_{d\in \mathcal F_{a,b}} w\left(\frac{d}{D}\right)  \sum_{p^2||d } \log p^{2} 
			+ \sum_{d\in \mathcal F_{a,b}} w\left(\frac{d}{D}\right)  \sum_{p^3||d } \log p^{4} \\
			&\ll   \sum_{\substack{p >2}} \log p \sum_{\substack{ d'p^2 \in \mathcal F_{a,b}  }} w\left(\frac{d'p^2}{D}\right) 
			+ \sum_{\substack{p >2}} \log p \sum_{\substack{ d'p^3 \in \mathcal F_{a,b}  }} w\left(\frac{d'p^3}{D}\right)  \\
			&\ll D  \sum_{\substack{p >2}} \frac{\log p}{p^2} + D\sum_{\substack{p >2}} \frac{\log p}{p^3}
			\ll   D.  
	\end{split}\end{equation}  
	which concludes the proof.
\end{proof}

\subsection{Explicit formula} 
 
For a fixed $d$ odd fourth-power-free integer,
taking the logarithmic derivative of the functional equation for $L(s,\xi_d)$, we obtain
$$ \frac{L'}{L}(s, \xi_d) = \frac{X_d'}{X_d}(s) - \frac{L'}{L}(1-s, \xi_d), $$
where
\begin{align*}
	X_d(s) &	= W(\xi_d) (\N(\mathfrak{f}_d))^{\frac12 - s} \pi^{2s-1} 
	\frac{\Gamma\left( \frac32 - s \right)}{\Gamma \left( s+\frac12 \right)}.
\end{align*}
Computing the logarithmic derivatives of the Euler product for $L(\cdot,\xi_d)$ and and of $X_d$, we find that
\begin{align} \label{log-derivative-1}
	- \frac{L'}{L}(s, \xi_d) 
	&=  \sum_{\substack{p \equiv 1 \bmod 4\\ p = \pi\overline{\pi} \\ \pi \equiv 1 \bmod \lambda^3\\ k \geq 1}} 
	\frac{ \left( \xi_d((\pi))^k + \overline{\xi_d}((\pi))^k \right)\log{p}}{p^{ks}}
	+  2 \sum_{\substack{p \equiv 3 \bmod 4 \\ p\nmid d \\k \geq 1 }}  \frac{(-1)^k \log{p}}{p^{2sk}} \\
	\frac{X'_d}{X_d}(s) &= - \log \N(\mathfrak{f}_d) + 2 \log{\pi} - \frac{\Gamma'}{\Gamma}\big(\tfrac32 - s \big)
	- \frac{\Gamma'}{\Gamma}\big( s  + \tfrac12 \big).
 \end{align}
 
Let $D>0$ and $\phi$ be an even test function with compact support, and $D>0$. Since $L(s,\xi_d)$ has no trivial zeros for $\re(s) \geq -\tfrac14$, we obtain
\begin{align}\begin{split}\label{U_in_OLD}
	\mathcal{D}(\phi, E_d) &= \frac{1}{2 \pi i} \left( \int_{(\frac54)} - \int_{(-\frac14)}  \right) \frac{L'}{L}(s, \xi_d) \phi \left( \frac{\log{ D}}{\pi} \frac{s-\frac12}{i} \right) \d s \\
	&= \frac{1}{2 \pi i} \int_{(\frac54)}  \left( \frac{L'}{L}(s, \xi_d) -   \frac{L'}{L}(1-s, \xi_d) \right) \phi \left( \frac{\log D}{\pi} \frac{s-\frac12}{i} \right) \d s \\
	&=   \frac{1}{2 \pi i} \int_{(\frac54)}  \left( 2 \frac{L'}{L}(s, \xi_d) -   \frac{X'_d}{X_d}(s) \right)\phi \left( \frac{\log{D}}{\pi} \frac{s-\frac12}{i} \right) \d s \\
	&=  U_{\infty}(\phi, \xi_d)  +  U_{\text{inert}}(\phi, \xi_d) + U_{\text{split}}(\phi, \xi_d),
\end{split}\end{align}
where
\begin{align} \begin{split} \label{define-U}
	U_{\infty}(\phi, \xi_d) 
	&:= \frac{1}{2 \pi i} \int_{(\frac54)}  - \frac{X'_d}{X_d}(s) \phi \left( \frac{\log{D}}{\pi} \frac{s-\frac12}{i} \right) \d s \\
	U_{\text{inert}}(\phi, \xi_d) 
	&:= -\frac{2}{\pi i} \int_{(\frac54)}  \Big(\sum_{\substack{p \equiv 3 \bmod 4 \\ p\nmid d \\ k \geq 1}}  \frac{(-1)^k \; \log{p} }{p^{2ks}} \Big) \phi \left( \frac{\log{D}}{\pi} \frac{s-\frac12}{i} \right) \d s \\
	U_{\text{split}}(\phi, \xi_d)   &:= - \frac{1}{\pi i} \int_{(\frac54)}  \Big( \sumsplit \frac{\big( \xi_{d}^k (({\pi})) + \overline{\xi_d}^k((\pi)) \big) \log{p}}{p^{ks}} \Big) \phi \left( \frac{\log{D}}{\pi} \frac{s-\frac12}{i} \right) \d s.
\end{split}\end{align}
Similarly to \cite[Lem.~3.1]{DDW} with $k=1$, we have
\begin{align*}
	U_{\infty}(\phi, \xi_d) 
	= \frac{\log{\N({\mathfrak{f}_d)}}}{\log{D^2}}\widehat\phi(0) +O\bigg(\frac{1}{\log{D}}  \bigg) 
\end{align*}
from which it then follows by Lemma~\ref{Lem Conductor condition} that
\begin{align} \label{U-infty}
	\frac{1}{\mathcal{S}_{\mathcal F_{a,b}} (w, D)} \sum_{\substack{ d\in \cF_{a,b}}}  U_{\infty}(\phi, \xi_d) w\left(\frac{d}{D}\right)
	&= \widehat{\phi}(0) + O \left( \frac{1}{\log{D}} \right). 
\end{align}

We now move to the study of the term $U_{\text{inert}}$ in \eqref{U_in_OLD}.
Shifting the contour to the line $\re(s)= \frac12$, and applying the change of variables $s=\frac12 + \frac{ \pi i t}{\log{D}}$, we write
\begin{align*}
	U_{\text {inert}} (\phi, \xi_d) 
	&= -\frac{2}{\log{D}} \sum_{\substack{p \equiv 3 \bmod 4 \\ p\nmid d  \\ k \geq 1}}  \frac{(-1)^k \log{p}}{p^{k}}   
	\int_{-\infty}^\infty \frac{1}{p^{2k \left( \frac{i t \pi}{\log D} \right)}} \phi \left( t \right) \d t \\
	&= -\frac{2}{\log{D}} \sum_{\substack{p \equiv 3 \bmod 4 \\ p\nmid d  \\ k \geq 1}}  \frac{(-1)^k \log{p}}{ p^{k}}  \widehat{\phi} \left( \frac{k \log{p}}{\log{D}} \right) .
\end{align*}
Working as in \cite[Lemma 3.5]{DDW}, this gives 
\begin{align}\label{U-inert}
	U_{\text {inert}} (\phi, \xi_d) = \frac{1}{2} \int_\R \widehat{\phi}(u) \; \d u + O \left( \frac{1}{\log{D}} \right).
\end{align}
Plugging \eqref{U-infty} and \eqref{U-inert} into \eqref{U_in_OLD} yields
\begin{align*}
	\mathscr{D}_{a,b}(\phi; w,D) 
	=\widehat{\phi}(0) + \frac{1}{2} \int_\R \widehat{\phi}(u) \; \d u + \frac{1}{\mathcal{S}_{\cF_{a,b}}(w,D)} 
	\sum_{\substack{d\in \cF_{a,b}}} U_{\text {split}} (\phi, \xi_d) w\left(\frac{d}{D}\right) + O \left( \frac{1}{\log{D}} \right).
\end{align*}

Proceeding as in \cite[Lemma 3.2]{DDW},  we write
\begin{align*}
	U_{\text {split}} (\phi, \xi_d) 
 	&=-\frac{1}{\log D}  \newsumsplit  \frac{\log{p}}{p^{\frac{k}{2}}}  \;  \xi_d^{k}((\pi))   \; \widehat{\phi} \; \left( \frac{ k \log{p}}{2\log{D}} \right),
\end{align*}
where we now sum over the primes $\pi \in \Z[i]$ and note that $\xi_d^{k}((\overline{\pi})) = \overline{\xi}_d^{k}((\pi))$.
Changing the variable $\pi \rightarrow \overline \pi$ in the sum, the contribution of the split primes to $\mathscr{D}_{a,b}(\phi, w,D)$ is then
\begin{align*}\begin{split}
	\sum_{d \in \mathcal{F}_{a,b}} U_{\text {split}} (\phi, \xi_d) w\left(\frac{d}{D}\right) 
	=-\frac{1}{\log D} \newsumsplit 
	\frac{\log{p}}{p^{\frac{k}{2}}}  \sum_{d \in \mathcal{F}_{a,b}} \overline{\xi}_d^{k}((\pi)) \widehat{\phi} \left( \frac{ k \log{p}}{2\log{D}} \right) w\left(\frac{d}{D}\right).
\end{split}\end{align*}
Since ${\mathcal{S}_{\cF_{a,b}}(w,D)} \asymp_{\cF_{a,b},w} D$, using the definition of $\xi_d$ in~\eqref{Def_xi_d}, we may summarize the results of this section as follows.
\begin{remark}\label{Prop To Prove}
	Let $w$ satisfy the conditions in \eqref{w-conditions}, 
	and $\phi$ be an even Schwartz function such that the support of $\widehat{\phi}$ is contained in $(- \nu, \nu)$.  
	Theorem \ref{main-theorem OLD} then follows provided that
	\begin{align} \label{TO-PROVE-3}
		\newsumsplit
		\frac{\log{p}}{p^{\frac{k}{2}}}  \left( \frac{\overline{\pi}}{|\pi|} \right)^k 
		\widehat{\phi} \left( \frac{k \log{p}}{2\log{D}} \right) 
		\sum_{d \in \mathcal{F}_{a,b}} \quartic{d}{\pi}^k  w\left(\frac{d}{D}\right) 
		\ll D
	\end{align}
	as $D\to \infty$.
\end{remark}

\section{Poisson summation formula}\label{Section Poisson}

In this section, we apply Poisson summation to the $d$-sum in \eqref{TO-PROVE-3}. To do so, we first use the Riemann Hypothesis together with a M\"obius inversion argument to replace the sum over fourth-power-free integers to a more tractable sum over all integers.

\begin{lemma} \label{Lemma sieve before Poisson}
	Let $w$ satisfy the conditions in \eqref{w-conditions}, and $\phi$ be an even Schwartz function such that the support of $\widehat{\phi}$ is contained in $(- \nu, \nu)$.  
	Let $a,b$ be positive odd integers, and assume the Riemann Hypothesis for~$L(s, \xi_d)$ for all odd integers $d$. Then, for any $\varepsilon >0$ and any $D,y >0$, we have 
	\begin{align*}
		\sum_{\substack{\pi \; \mathrm{primary} \\ \N(\pi) =  p \equiv 1 \bmod 4\\k \geq 1}} &  \frac{\log{p}}{p^{\frac{k}{2}}} \left( \frac{\overline{\pi}}{|\pi|} \right)^k \widehat{\phi} \left( \frac{ k \log{p}}{2\log{D}} \right) \sum_{d \in \mathcal{F}_{a,b}}  \quartic{d}{\pi}^k  w\left(\frac{d}{D}\right) \\
		&= \sum_{d_2,d_3,\ell < y} \sum_{\substack{\delta\mid d_2d_3 \\A \equiv B \bmod{4} }} \mu(\ell)\mu(\delta) \sum_{\substack{d \in \Z \\ d\equiv A \bmod 16}} 
	w \left( \frac{d \delta  \ell^2 d_{2}^{2}d_{3}^{3}}{D} \right) \\
		& \hspace{1cm} \times \sum_{\substack{\pi \; \mathrm{primary} \\ \N(\pi) =  p \equiv 1 \bmod 4\\k \geq 1}} \quartic{d \delta  \ell^2 d_{2}^{2}d_{3}^{3}}{\pi}^k \frac{\log{p}}{p^{\frac{k}{2}}}  
		\left( \frac{\overline{\pi}}{|\pi|} \right)^k \widehat{\phi} \left( \frac{ k \log{p}}{2\log{D}} \right) + O \left( \frac{D^{1+\varepsilon}}{y} \right)
	\end{align*}
	as $D\to\infty$,
	where $A = a\delta^{3}\ell^{2}d_{2}^{2}d_{3}$ and $B= b\delta d_{2}d_{3}$.
\end{lemma}

\begin{proof}
	To sieve out the desired conditions, we write each fourth-power-free $d\in \cF_{a,b}$ as $d = d_{1}d_{2}^{2}d_{3}^{3}$, where $d_{1} \in \Z$ and  $d_{2}, d_{3} \in \Z_{\geq 1}$ are square-free odd and pairwise coprime, so that
	\begin{align} \nonumber
		\sum_{d \in \mathcal{F}_{a,b}}  \quartic{d}{\pi}^k  w\left(\frac{d}{D}\right) &=
		\sum_{\substack{d \in \Z \\ \text{fourth-power-free}\\
		d \equiv a \bmod{16}\\ \rad(d) \equiv b \bmod{4}\\}}  \quartic{d}{\pi}^k    w \left( \frac{d}{D} \right)\\
		\label{sieve-out}
		&= \sum_{\substack{d_{3}  ~\square \text{-free}}}\quartic{d_{3}^{3}}{\pi}^k \sum_{\substack{d_{2}  ~\square \text{-free}\\ (d_{2},d_{3}) = 1}}\quartic{d_{2}^{2}}{\pi}^k \sum_{\substack{d_{1} \in \Z, ~\square \text{-free}\\(d_{1},d_{2}d_{3}) = 1 \\d_{1} \equiv ad_{2}^{-2}d_{3}^{-3}\bmod{16}\\d_{1} \equiv bd_{2}^{-1}d_{3}^{-1}\bmod{4}}}\quartic{d_{1}}{\pi}^k w \left( \frac{d_{1}d_{2}^{2}d_{3}^{3}}{D} \right).
	\end{align}
	We have
	\begin{align*}
		&\sum_{\substack{d_{1} \in \Z, ~\square \text{-free}\\(d_{1},d_{2}d_{3}) = 1 \\d_{1} \equiv ad_{2}^{-2}d_{3}^{-3}\bmod{16}\\d_{1} \equiv bd_{2}^{-1}d_{3}^{-1}\bmod{4}}}\quartic{d_{1}}{\pi}^k w \left( \frac{d_{1}d_{2}^{2}d_{3}^{3}}{D} \right) 
		= \sum_{\substack{d_{1} \in \Z, ~\square \text{-free}\\d_{1} \equiv ad_{2}^{-2}d_{3}^{-3}\bmod{16}\\d_{1} \equiv ad_{2}^{-1}d_{3}^{-1}\bmod{4}}}\sum_{\delta \mid (d_{1},d_{2}d_{3})}\mu(\delta)\quartic{d_{1}}{\pi}^k w \left( \frac{d_{1}d_{2}^{2}d_{3}^{3}}{D} \right)\\
		&= \sum_{\delta \mid d_{2}d_{3}}\mu(\delta)\quartic{\delta}{\pi}^k \sum_{\substack{d'_{1} \in \Z \\d'_{1} \equiv a\delta^{-1}d_{2}^{-2}d_{3}^{-3}\bmod{16}\\d'_{1} \equiv b\delta^{-1}d_{2}^{-1}d_{3}^{-1}\bmod{4}}}\sum_{\substack{\ell^{2} \mid d'_{1}}}\mu(\ell)\quartic{d'_{1}}{\pi}^k w \left( \frac{d'_{1}\delta d_{2}^{2}d_{3}^{3}}{D} \right)\\
		&= \sum_{\delta \mid d_{2}d_{3}}\mu(\delta)\quartic{\delta}{\pi}^k \sum_{\ell}\mu(\ell)\quartic{\ell^{2}}{\pi}^k \sum_{\substack{d \in \Z \\d \equiv a\delta^{3}\ell^{2}d_{2}^{2}d_{3}\bmod{16}\\d \equiv b\delta d_{2}d_{3}\bmod{4}}}\quartic{d}{\pi}^k w \left( \frac{ d\delta \ell^2 d_{2}^{2}d_{3}^{3}}{D} \right).
	\end{align*}
	where in the last line we moreover note that $x^{4} \equiv 1 \bmod{16}$ and $x^{2}\equiv 1\bmod{4}$ for any odd $x \in \Z$.
	Replacing in \eqref{sieve-out} and then in the prime sum, we obtain
	\begin{multline}\label{Eq sieve prime sum}
		\sum_{\substack{\pi \; \mathrm{primary} \\ \N(\pi) =  p \equiv 1 \bmod 4\\k \geq 1}}   \frac{\log{p}}{p^{\frac{k}{2}}} \left( \frac{\overline{\pi}}{|\pi|} \right)^k \widehat{\phi} \left( \frac{ k \log{p}}{2\log{D}} \right) \sum_{d \in \mathcal{F}_{a,b}}  \quartic{d}{\pi}^k  w\left(\frac{d}{D}\right) \\	
		= \sum_{\substack{d_2,d_3 \\ \square \text{-free} \\ (d_2, d_3)=1}} \sum_{\ell}\sum_{\substack{\delta\mid d_2d_3 \\ A \equiv B \bmod{4} }} \mu(\ell)\mu(\delta) \sum_{\substack{d \in \Z \\ d\equiv A \bmod 16}} w \left( \frac{ d\delta \ell^2 d_{2}^{2}d_{3}^{3}}{D} \right) \\
		\times \sum_{\substack{\pi \text{ primary} \\ \N(\pi) = p \equiv 1 \bmod 4\\k \geq 1}}
		\frac{\log{p}}{p^{\frac{k}{2}}}  \left( \frac{\overline{\pi}}{|\pi|} \right)^k \quartic{ d \delta\ell^2 d_{2}^{2}d_{3}^{3}}{\pi}^k\widehat{\phi} \left( \frac{ k \log{p}}{2\log{D}} \right).
	\end{multline}
	Taking $d_2, d_3, \ell < y$ gives the main term of the lemma. 

	For the remainder sum, we denote
	\begin{align*}
		E
		&:= \sum_{\substack{\pi \text{ primary} \\ \N(\pi) = p \equiv 1 \bmod 4\\k \geq 1}} \frac{\log{p}}{p^{\frac{k}{2}}} \overline{\xi}^k_{ d \delta\ell^2 d_{2}^{2}d_{3}^{3}}((\pi)) \;
		\widehat{\phi} \left( \frac{ k\log{p}}{2\log{D}} \right) 
	\end{align*}
	the inner prime sum in~\eqref{Eq sieve prime sum}.
	We also define 
	\begin{align*}
		F_{X}(t) := \widehat{\phi} \left( \frac{ \log{t}}{\log{X}} \right) .
	\end{align*}
	and we remark that for $\sigma > 0$, its Mellin transform $\widetilde F_X : s \mapsto \int_{0}^{\infty} F_X(u) u^{s-1} \dd u$ is such that
	\begin{align} \label{bound-MT}
		\widetilde F_X(\sigma + it) \ll X^{\nu \sigma} \min{(1, \lvert t\rvert^{-2})}.
	\end{align}
	This is similar to the proof of Lemma \ref{TransformBounds} below for the more complicated function $R_{X, Y}(t) = \widehat{\phi} ( \frac{ \log{t}}{\log{X}} )  \widehat{w} ( \frac{Y}{t} ) .$
	Then by Mellin inversion, and taking $X=D^2$, we get
	\begin{align*} 
		E =  \sum_{\substack{\pi \text{ primary} \\ \N(\pi) = p \equiv 1 \bmod 4\\k \geq 1}} \frac{\log{p}}{p^{\frac{k}{2}}} \; \overline{\xi}^k_{ d \delta\ell^2 d_{2}^{2}d_{3}^{3}}((\pi)) \;
		F_{D^2}(p^k)	
		= \frac{1}{2 \pi i} \int_{(2)} G_{ d\delta \ell^2 d_{2}^{2}d_{3}^{3}}(s+\tfrac12) \widetilde F_{D^2}(s) {\d s}
	\end{align*}
	where by \eqref{log-derivative-1}
	$$G_{ d \delta\ell^2 d_{2}^{2}d_{3}^{3}}(s) = - \frac{L'}{L} (s, \overline{\xi}_{ d \delta\ell^2 d_{2}^{2}d_{3}^{3}}) -  2 \sum_{\substack{\pi \text{ primary}\\N (\pi ) = p^2 \\ p \equiv 3 \bmod 4\\ k \geq 1}} \frac{\overline{\xi}_{d\delta  \ell^2 d_{2}^{2}d_{3}^{3}}^k((\pi)) \log{p}}{p^{2 k s}}.$$
	Observe that the sum is absolutely convergent for $\re(s) >\frac12$.
	As we assume $L(s, \overline{\xi}_{d})$ satisfies the Riemann Hypothesis for all $d$, then $\frac{L'}{L}(s, \overline{\xi}_d)$ satisfies the Lindel{\"o}f Hypothesis for all $d$. In particular, for any $\varepsilon>0$, we have
	$$ \frac{L'}{L}(s, \overline{\xi}_d) \ll_\sigma (t \lvert d\rvert)^\varepsilon, $$
	for any $s = \sigma+it$ such that $\frac12 < \sigma < \frac54$ by~\cite[Theorem 5.17]{IK}. 
	Then, moving the integral to $\re(s) = \varepsilon$, and using \eqref{bound-MT}, we get that 
	\begin{align*}
		E  \ll_{\varepsilon} D^{2\nu\varepsilon} ( \lvert d \rvert \delta \ell d_2 d_3)^\varepsilon.
	\end{align*}
	In the case $\ell \geq y$, we use $w(t) \ll \min(1,\lvert t\rvert^{-2})$, since $w$ is a Schwartz function, to obtain  
	\begin{align*}
		\sum_{\substack{d_2,d_3 \geq 1 \\ \square \text{-free} \\ (d_2, d_3)=1}} \sum_{\ell \geq y} & \sum_{\substack{\delta\mid d_2d_3 \\ A \equiv B \bmod{4} }} \mu(\ell)\mu(\delta) \sum_{\substack{d \in \Z \\ d\equiv A \bmod 16}} w \left( \frac{ d\delta \ell^2 d_{2}^{2}d_{3}^{3}}{D} \right) E  \\
		& \ll_\varepsilon D^{2\nu\varepsilon}  \sum_{\ell \geq y}	\sum_{d_3,d_2}\sum_{\delta\mid d_2d_3 }  \bigg(\sum_{ \lvert d\rvert \leq  \frac{D}{\delta\ell^2d_2^2d_3^3} }  
		( d \delta\ell d_2 d_3)^{\varepsilon} + \sum_{\lvert d\rvert >  \frac{D}{\delta \ell^2 d_2^2d_3^3} } D^2 ( d \delta\ell^2 d_2^2 d_3^3)^{\varepsilon -2}\bigg) \\ 
		& \ll_\varepsilon D^{2\nu\varepsilon} \sum_{\ell \geq y}	\sum_{d_3,d_2}\sum_{\delta\mid d_2d_3 }  \bigg( \frac{D^{1+\varepsilon}}{\delta \ell^{2+\varepsilon} d_2^{2+\varepsilon} d_3^{3 + \varepsilon}} +  \frac{D^{1+\varepsilon}} {\delta \ell^{2+\varepsilon} d_2^{2+\varepsilon} d_3^{3 + \varepsilon}}\bigg)   
		\ll_{\varepsilon} \frac{D^{1+2(\nu+1)\varepsilon}}{y}.
	\end{align*}
	Estimating loosely the sum over divisors $\delta$ as $O((d_2d_3)^{\varepsilon})$, the contributions for $d_2 \geq y$ and $d_3 \geq y$ are bounded as above.
	Since $\varepsilon>0$ is arbitrary, the claim follows.
\end{proof}
 
By applying Lemma~\ref{Lemma sieve before Poisson} to Remark~\ref{Prop To Prove} with $y= D^{\varepsilon}$, our desired goal reduces to the following.

\begin{remark}\label{Remark TO Prove 2}
	Let $w$ satisfy the conditions in \eqref{w-conditions}, $\phi$ be an even Schwartz function such that the support of $\widehat{\phi}$ is contained in $(- \nu, \nu)$, and $\varepsilon > 0$.
	Theorem \ref{main-theorem OLD} then follows provided that
	there exist $\delta >30\varepsilon$ such that
	for any integer $M < D^{9\varepsilon}$ and any odd integer $\alpha$, one has 
	\begin{align} \label{sieve-is-gone}
		\sum_{\substack{\pi \text{ primary} \\ \N(\pi) = p \equiv 1 \bmod 4\\k \geq 1}} \frac{\log{p}}{p^{\frac{k}{2}}} \left(  \frac{\overline{\pi}}{|\pi|} \right)^k  \quartic{M}{\pi}^k  
		\widehat{\phi} \left( \frac{ k \log{p}}{2\log{D}} \right)  \sum_{d \equiv \alpha \bmod16} \quartic{d}{\pi}^k   w \left( \frac{d M}{D} \right)   \ll M^2 D^{1-\delta}.
	\end{align}
 	In particular, using the P\'olya--Vinogradov inequality~\cite[\S~23]{Dav} to bound the inner sum, we deduce that  Theorem~\ref{main-theorem OLD} is true under the restricted assumption that $\nu < \frac12$ (under GRH).
\end{remark} 
  
Having dealt with the fourth-power-freeness of $d$, we now sieve out the congruence condition modulo $16$. This is achieved via the following identity a direct consequence of the orthogonality of Dirichlet characters:

\begin{equation}\label{d_character_sum}
	\sum_{\substack{ d \equiv \alpha \bmod 16}}  \quartic{d}{\pi}^k w \left( \frac{d M }{D} \right) = \frac{1}{8} \sum_{\chi \bmod 16}\chi(\alpha^{-1})  \sum_{d \in \Z}\chi(d)  \quartic{d}{\pi^k} w \left( \frac{d M}{D} \right), 
\end{equation}
where $\chi$ runs over the Dirichlet characters modulo $16$ (over $\Z$).

\begin{lemma}[Poisson]\label{Poisson} 	
	Let $w$ satisfy the conditions in \eqref{w-conditions}. Let $\chi$ be one of the  Dirichlet character modulo $16$ over $\Z$, $k \in \Z_{\geq 1}$, and let $\pi$ be a primary Gaussian prime with $\N(\pi) =p$.  
	Then, 
	\begin{align*} 
		S_{D,\chi}(\pi, k) &:=  \sum_{d \in \mathbb Z}    \chi(d) \quartic{d}{\pi^k} w \left( \frac{d}{D} \right) \\
		& = \frac{D}{16 p^k}  \quartic{\overline{\pi}^k}{\pi^k} \chi(p^k) \sum_{m \in \Z}  g_4(m, \pi^k) G(m, \chi, 16)  \widehat{w} \left( \frac{m D }{16 p^k} \right).
	\end{align*}
	where the Gauss sums $g_4(m,\pi^k)$ and $G(m, \chi, 16)$ are respectively defined by \eqref{Def-Gauss-sum} and~\eqref{weird-Gauss-sum}.
\end{lemma}

\begin{proof} 
	By the Poisson summation formula, we have
	\[\sum_{m=-\infty}^{\infty}f(x+mc) = \frac{1}{c}\sum_{m=-\infty}^{\infty}e\big({\tfrac{m x}{c}}\big)\widehat{f}\big(\tfrac{m}{c}\big),\]
	for continuous integrable complex functions $f$.
	Since $\chi(\cdot) \quartic{\cdot}{\pi^k}$ is a character modulo $16 p^k$, we get using Poisson that
	\begin{align}\label{Poisson_sum_D} \begin{split}
		S_{D,\chi}(\pi, k) &= \sum_{a \bmod 16 p^k} \chi(a) \quartic{a}{\pi^k} \sum_{m \in \Z} w\left( \frac{16 p^km+a}{D} \right)\\
		&= \frac{D}{16 p^k} \sum_{m \in \Z} \sum_{a \bmod 16 p^k} \chi(a) \quartic{a}{\pi^k}
		e \left(\frac{ma}{16 p^k} \right)\widehat{w} \left( \frac{m D}{16 p^k} \right)\\
		&= \frac{D}{16 p^k} \sum_{m \in \Z} \sum_{b \bmod p^k} \sum_{c \bmod 16}\chi(cp^k)\quartic{16 b}{\pi^k}
		e \left(\frac{mb}{p^k} \right) e \left(\frac{mc}{16} \right) \widehat{w} \left( \frac{m D}{16 p^k} \right) \\
		&= \frac{D}{16 p^k} \chi(p^k) \quartic{16}{\pi^k}\;  \sum_{m \in \Z} G\big(m, \quartic{\cdot}{\pi^k}, p^k\big) G(m, \chi, 16) \widehat{w} \left( \frac{m D}{16 p^k} \right),
	\end{split}\end{align}
	where for any Dirichlet character $\omega$ of modulus $N$ over $\Z$,
	\begin{align} \label{weird-Gauss-sum}
		G(m, \omega, N) := \sum_{b \bmod N} \omega(b) e \left( \frac{bm}{N} \right). 
	\end{align}
	Using the ring isomorphism $\Z/p^k\Z \cong \Z[i]/(\pi^k)$ given explicitly by $a \bmod {p^k} \mapsto a \bmod{\pi^k}$, we find that the Gauss sum from~\eqref{Def-Gauss-sum} can be re-written as (cf. \cite[(2.11)]{Diaconu})
	\begin{align}\label{g4_and_G1}\begin{split}
		g_4(m, \pi^k) &:= \sum_{\alpha \bmod \pi^k} \quartic{\alpha}{\pi^k} e \left(\Tr \left( \frac{\alpha m}{\pi^k} \right) \right) 
		= \sum_{a \bmod p^k} \quartic{a}{\pi^k} e \left(  \frac{am}{p^k} ({\pi}^k +  \overline{\pi}^k)  \right) \\
		&= \oquartic{\overline{\pi}^k}{\pi^k} \sum_{a \bmod p^k} \quartic{a}{\pi^k} e \left(  \frac{am}{p^k} \right) =  \oquartic{\overline{\pi}^k}{\pi^k} G\big(m,\quartic{\cdot}{\pi^k}, p^k \big).
	\end{split}\end{align}
	Replacing \eqref{g4_and_G1} in \eqref{Poisson_sum_D}, the result follows.
\end{proof}

We conclude this section by offering alternative expressions for $\quartic{\overline{\pi}^k}{\pi^k}$ and $\chi(p^k)$ to include their contribution in the quartic symbol.

\begin{lemma} \label{weird-character} 
	Let $\alpha \in \Z[i]$ be primary.  Then
	\begin{align*}
		\quartic{ \overline{\alpha}}{\alpha} &=
		\begin{cases}
			\quartic{2i}{\alpha} & \textnormal{ if } \alpha \equiv 1 \bmod{4}\\ 
			i\quartic{2i}{\alpha} & \textnormal{ if }  \alpha \equiv 1 + \lambda^3 \bmod{4}.
		\end{cases}
	\end{align*}
\end{lemma}

\begin{proof}
	To begin, observe that since $\alpha = a + 2bi$ is primary, we have the following 	correspondences between congruence classes of $\alpha, a$, and $b$:
	\begin{center}\begin{tabular}{c|c|c}
		$\alpha \bmod{4}$ & $a \bmod{4}$ & $b \bmod{2}$\\ \hline 
		$1$ & $1$ & $0$\\
		$1 + \lambda^3$ & $3$ & $1$
	\end{tabular}\end{center}
	Note that 
	\begin{equation*}
		\quartic{ \overline{\alpha}}{\alpha} = \quartic{ \overline{\alpha}+\alpha}{\alpha} = 
		\quartic{2}{\alpha}\quartic{a}{a+2bi}.
	\end{equation*}
	When $a \equiv 1 \bmod{4}$, using \eqref{bilaw}, \eqref{chi_at_integers}, and the first supplementary quartic reciprocity law from~\eqref{Eq suppl quartic law}, we write
	\begin{equation*}
		\quartic{a}{a+2bi} = \quartic{a+2bi}{a}=\quartic{i}{a}\quartic{2b}{a} = \quartic{i}{\alpha}.
	\end{equation*}
	Similarly, when $a \equiv 3 \bmod{4}$, we have
	\begin{equation*}
		\quartic{a}{a+2bi}= \quartic{-1}{a+2bi}\quartic{-a}{a+2bi} = i^{1-a}\quartic{a+2bi}{-a}= - \quartic{i}{-a} = -i^{\frac{1+a}{2}} = - \quartic{i}{\alpha}i^a.
	\end{equation*}
	The claim follows.
\end{proof}

\begin{lemma} \label{chi-value}
	Let $\chi$ be a Dirichlet character modulo $16$. Then, there exists $\ell = \ell(\chi) \in \lbrace0,1,2,3\rbrace$ such that
	$$\chi : p \mapsto  \quartic{i}{\pi}^\ell, \text{ whenever } p =\pi\overline{\pi}.$$
\end{lemma}	

\begin{proof}
	Note that the group of Dirichlet characters modulus 16 on $\Z$ 	is isomorphic to $\Z/2\Z \times \Z/4\Z$ and generated by  $\quartic{1+i}{\cdot}$ and $\left(\tfrac{-1}{\cdot}\right)_{2}$. Thus for every $\chi \in \widehat{(\Z/16\Z)^{\times}}$, there exists $\ell \in \{0,1,2,3\}$ and $\ell' \in \{ 0, 1 \}$ such that
	\[\chi(\cdot) = \chi_{\ell, \ell'}(\cdot) = \left(\frac{-1}{\cdot}\right)_{2}^{\ell'}\quartic{1+i}{\cdot}^{\ell}.\]
	As in \cite[p. 195]{Lemmermeyer}, we further note that when $p = \pi \overline{\pi}$ splits, then 
	\begin{align*}
		\quartic{1+i}{p} =\quartic{i}{\pi} \;\mbox{and}\;  \left(\frac{-1}{p}\right)_{2} = 1, 
	\end{align*}
	from which the lemma follows.
\end{proof}

Using lemmas \ref{Poisson}, \ref{weird-character} and \ref{chi-value} in \eqref{d_character_sum},  we get
\begin{align*}
	\sum_{d \equiv \alpha \bmod16} \quartic{d}{\pi}^k   w \left( \frac{d M}{D} \right) = \frac{ c(\pi^k)  D}{2^7 p^k M} 
	\sum_{\chi \bmod 16} \chi(\alpha^{-1})
	\sum_{\substack{m \in \Z}}  g_4 (8 i^{3(\ell(\chi)+1)}  m, \pi^k)  G(m, \chi, 16) 
	\widehat{w} \left( \frac{m D}{16 p^k M} \right),
\end{align*}
where $c(\pi^k):=1$ if $\pi^k \equiv 1 \bmod 4$ and $c(\pi^k) := i$ if $\pi^k \equiv 1 + \lambda^3 \bmod 4$.
Replacing in Remark~\ref{Remark TO Prove 2}, the proof of Theorem \ref{main-theorem OLD} reduces to the following.

\begin{remark}\label{Proposition_reduction}
	Let $w$ satisfy the conditions in \eqref{w-conditions}, and let $\phi$ be an even Schwartz function such that the support of $\widehat{\phi}$ is contained in $(- \nu, \nu)$. Fix $\beta \in \lbrace 1, 1 + \lambda^3\rbrace$ and let 
	\begin{align*} 
		S_{\text{split}}(D, M) := \sum_{\substack{m \in \Z}}   \bigg|  \sum_{\substack{\pi \text{primary} \\ \N(\pi) = p \equiv 1 \bmod 4 \\ (\pi, M)=1 \\ k \geq 1 \\ \pi^k \equiv \beta \bmod 4}}  \frac{g_4\big(8 i^{3(\ell+1)} M^3 m, \pi^k \big)}{\N(\pi^k)^{\frac{3}{2}}}   \log \N(\pi) \frac{\overline{\pi^k}}{|\pi^k|}  \,
		\widehat{\phi} \left( \frac{  k \log{\N(\pi)}}{2\log{D}} \right)  \widehat{w} \left( \frac{m D}{16 \N(\pi)^k M} \right) \bigg|
	\end{align*}
	Then, to prove
	Theorem \ref{main-theorem OLD}, it suffices to show that there exists $\delta >0$ such that
	\begin{align}\label{reduction before Vaughan}
		S_{\text{split}}(D, M) 
		\ll \frac{M^{3}}{D^{\delta}}
	\end{align}
	for any $\beta \in \{ 1, 1 + \lambda^3 \}$, $\ell \in \{0,1,2,3\}$, uniformly for all $\varepsilon>0$ and $M < D^{\varepsilon}$. 
 \end{remark}
 
\begin{proof}[Proof of Theorem~\ref{main-theorem OLD} assuming Conjecture \ref{conj-DDHL}]
	We want to show that Conjecture \ref{conj-DDHL} for $\ell=1$ implies the bound~\eqref{reduction before Vaughan} when $\nu < 1$. We first show that
	\begin{align}\nonumber
		S_{\text{inert}}(D, M) &:= \sum_{m \in \Z} \bigg| \sum_{\substack{ \pi  \text{ primary} \\ \N(\pi) = p^2 \\ (\pi, M)=1 \\ k \geq 1 \\ \pi^k \equiv \beta \bmod 4}}  
		\frac{g_4(8 i^{3(\ell+1)} M^3 m, \pi^k)}{\N(\pi^k)^\frac32} \log{\N(\pi)} 
		\frac{\overline{\pi^k}}{|\pi^k|} \widehat{\phi} \left( \frac{  k \log{\N(\pi)}}{2\log{D}} \right)  \widehat{w} \left( \frac{m D}{16 \N(\pi)^k M} \right)   \bigg| \\
		&\ll M D^{\nu -1} \log{D},\label{inert-primes} 
	\end{align}
	which satisfies the bound of Remark \ref{Proposition_reduction} for $\nu < 1$. 
	Since $\widehat{\phi}$ is supported on $(-\nu, \nu)$ and $\widehat{w}$ is supported on $(-\eta, \eta)$, we indeed have
	\begin{align*}
		S_{\text{inert}}(D, M)   \ll  \sum_{\substack{|m| \ll \frac{p^{2k} M}{D} \\ p^{2k} \leq D^{2 \nu} \\ 1 \leq k \leq \nu \log{D} }}   \frac{|g_4(8 i^{3(\ell+1)} M^3 m, \pi^k)|}{\N(\pi^k)^\frac32} \log{\N(\pi)}.
	\end{align*}
	We use the identities~\eqref{rel4}. 
	When $k \not\equiv 0 \bmod 4$, we have  $|g_4(8 i^{3(\ell+1)} M^3 m, \pi^k)| \leq \N(\pi)^{k-\frac12}$ if $p^{k-1} \parallel m$, and it is $0$ otherwise. We get 
	\begin{align*} 
		\sum_{\substack{|m| \ll \frac{p^{2k} M}{D} \\ p^{2k} \leq D^{2 \nu} \\ 1 \leq k \leq \nu \log{D} \\ k \not\equiv 0 \bmod 4}}  
		\frac{|g_4(8 i^{3(\ell+1)} M^3 m, \pi^k)|}{\N(\pi^k)^\frac32} \log{\N(\pi)} &\ll
		\sum_{\substack{|m| \ll \frac{p^{2k} M}{D} \\ p^{2k} \leq D^{2 \nu} \\ p^{k-1} \parallel m \\ 1 \leq k \leq \nu \log{D} \\ k \not\equiv 0 \bmod 4 }} 
		\N(\pi)^{-\frac{k}{2} - \frac{1}{2}} \log p \\
		&\ll \frac{M}{D} \sum_{\substack{p \leq D^{\frac{\nu}{k}} \\ 1 \leq k \leq \nu \log{D}}} \log p \ll M D^{\nu-1} \log{D}.
	\end{align*}
	When $k \equiv 0 \bmod 4$, we have the same term as above coming from $p^{k-1}\parallel m$ and supplementary terms for $p^k\mid m$, this yields
	\begin{align*} 
		&\sum_{\substack{|m| \ll \frac{p^{2k} M}{D} \\ p^{2k} \leq D^{2 \nu} \\ 1 \leq k \leq \nu \log{D} \\ k \equiv 0 \bmod 4}}  
		\frac{|g_4(8 i^{3(\ell+1)} M^3 m, \pi^k)|}{\N(\pi^k)^\frac32} \log{\N(\pi)} \\
		&\ll  \sum_{\substack{1 \leq k \leq \nu \log{D} \\ k \equiv 0 \bmod 4}}  \sum_{\substack{|m| \ll \frac{p^{2k} M}{D} \\ p^k \mid m \\ p^{2k} \leq D^{2 \nu} }}  
		\frac{\varphi(\pi^k)}{\N(\pi)^{3k/2}} + O \left(   M D^{\nu-1} \log{D}   \right)\\
		&\ll  \sum_{\substack{1 \leq k \leq \nu \log{D} \\ k \equiv 0 \bmod 4}}   \sum_{\substack{|m| \ll \frac{p^{k} M}{D} \\ p\leq D^{\nu/k} }}  
		p^{-k} + O \left(   M D^{\nu-1} \log{D}   \right) =  O \left(   M D^{\nu-1} \log{D}   \right).
	\end{align*}
	This proves \eqref{inert-primes}, and by Remark \ref{Proposition_reduction}, it suffices to show that Conjecture \ref{conj-DDHL} implies that
	\begin{align} \label{all-primes}
		\sum_{m \in \Z} \bigg| \sum_{\substack{ c \in \Z[i] \\ c \text{ primary} \\c \equiv \beta \bmod 4}} \frac{g_4(8 i^{3(\ell+1)} M^3 m, c)}{\N(c)^\frac32} \Lambda(c)
		\frac{\overline{c}}{|c|} 
		\widehat{\phi} \left( \frac{\log{\N(c)}}{2\log{D}} \right)  \widehat{w} \left( \frac{m D}{16 \N(c) M} \right) \bigg| \ll \frac{M^3}{D^\delta}
	\end{align}
	when $\nu < 1$.
	By partial summation and Conjecture \ref{conj-DDHL}, the sum over $c \in \Z[i]$ in \eqref{all-primes} is bounded by
	\begin{align*}
		& \int_{\frac{\lvert m\rvert D}{16\eta M}}^{D^{2 \nu}}  \Big( \sum_{\substack{1 \leq \N(c) \leq t \\ c \text{ primary } \\ c \equiv \beta \bmod 4}} 
		\frac{g_4(8 i^{3(\ell+1)} M^3 m, c)}{\N(c)^\frac12} \Lambda(c)
		\frac{\overline{c}}{|c|} \Big) \;
		\frac{\dd}{\dd t} \left( \frac{1}{t}  \widehat{\phi} \left( \frac{\log{t}}{2\log{D}} \right)  \widehat{w} \left( \frac{m D}{16 t M} \right) \right) \dd t \\
		&\ll  \N(M^3 m)^\varepsilon  \left(  \int_{\frac{\lvert m\rvert D}{16\eta M}}^{D^{2 \nu}}  \frac{t^{\frac12 + \varepsilon}}{t^2} \dd t +
		\frac{\lvert m\rvert D}{M} \int_{\frac{\lvert m\rvert D}{16\eta M}}^{D^{2 \nu}}  \frac{t^{\frac12 + \varepsilon}}{t^3} \dd t \right) \\
		&\ll \N(M^3 m)^\varepsilon  \Big( \frac{\lvert m\rvert D}{M} \Big)^{-\frac12 + \varepsilon} .
	\end{align*}
	Summing over $|m| \ll MD^{2\nu-1}$, we get the bound $M^{1+\varepsilon} D^{\nu-1+\varepsilon}$, which is smaller than the bound of Remark \ref{Proposition_reduction} for $\nu < 1$. This completes the proof of Theorem~\ref{main-theorem OLD} assuming Conjecture~\ref{conj-DDHL}.
\end{proof}

In the next sections, we prove Theorem~\ref{main-theorem OLD} without assuming Conjecture ~\ref{conj-DDHL} for $\nu < \frac35$.  We will need to bound the contribution of the prime powers that are not coprime to $m$ in $S_{\mathrm{split}}(D,M)$, which is done in the next lemma.

\begin{lemma}\label{Lem bound non-coprimality}
	Let $w$ satisfy the conditions in \eqref{w-conditions}, and let $\phi$ be an even Schwartz function such that the support of $\widehat{\phi}$ is contained in $(- \nu, \nu)$.
	For any $\beta \in \lbrace 1, 1 + \lambda^3\rbrace$, we have
	\begin{align*}
		\sum_{\substack{m \in \Z}}   \bigg|  \sum_{\substack{\pi \text{primary} \\ \N(\pi) = p \equiv 1 \bmod 4 \\ (\pi, M)=1 \\ k \geq 1 \\ \pi^k \equiv \beta \bmod 4 \\ (\pi,m)\neq 1}}  \frac{g_4\big(8 i^{3(\ell+1)} M^3 m, \pi^k \big)}{\N(\pi^k)^{\frac{3}{2}}}   \log \N(\pi) \frac{\overline{\pi^k}}{|\pi^k|}  \,
		\widehat{\phi} \left( \frac{  k \log{\N(\pi)}}{2\log{D}} \right)  \widehat{w} \left( \frac{m D}{16 \N(\pi)^k M} \right) \bigg| \ll M D^{\frac\nu2 -1} \log D.
	\end{align*}
\end{lemma}

\begin{proof}
	Similarly to the previous proof, it suffices to bound the triple sum
	\begin{align*}
		\sum_{1\leq k \leq \nu\log D} \sum_{\substack{p^k \leq D^{2\nu} \\ p \equiv 1 \bmod 4 \\ (p,M)=1}}	\sum_{\substack{ \lvert m\vert  \ll \frac{p^k M}{D}  \\ (\pi,m)\neq 1}}     \frac{\lvert g_4\big(8 i^{3(\ell+1)} M^3 m, \pi^k \big)\rvert }{ p^{\frac{3}{2}k}}   \log p 
	\end{align*}
	When $k \not\equiv 0 \bmod 4$, we have by \eqref{rel4} that  $|g_4(8 i^{3(\ell+1)} M^3 m, \pi^k)| \leq \N(\pi)^{k-\frac12}$ if $p^{k-1} \parallel m$, and it is $0$ otherwise.
	We get
	\begin{align*}
		\sum_{\substack{1\leq k \leq \nu\log D \\ k \not\equiv 0 \bmod 4 }} \sum_{\substack{p^k \leq D^{2\nu} \\ p \equiv 1 \bmod 4 \\ (p,M)=1}}	\sum_{\substack{ \lvert m\vert  \ll \frac{p^k M}{D}  \\ (\pi,m)\neq 1}}     \frac{\lvert g_4\big(8 i^{3(\ell+1)} M^3 m, \pi^k \big)\rvert }{ p^{\frac{3}{2}k}}   \log p 
		&\leq 	\sum_{\substack{2\leq k \leq \nu\log D \\ k \not\equiv 0 \bmod 4 }} \sum_{\substack{p^k \leq D^{2\nu} \\ p \equiv 1 \bmod 4 \\ (p,M)=1}}	\sum_{\substack{ \lvert m\vert  \ll \frac{p^k M}{D}  \\ p^{k-1} \parallel m}}     p^{-\frac{k+1}{2}}   \log p \\
		&\leq \frac{M}{D} \sum_{\substack{2\leq k \leq \nu\log D \\ k \not\equiv 0 \bmod 4 }} \sum_{\substack{p^k \leq D^{2\nu} \\ p \equiv 1 \bmod 4 \\ (p,M)=1}} p^{-\frac{k-1}{2}}   \log p \\
		&\ll M D^{\frac{\nu}{2} -1} \log D.
	\end{align*}
	When $k \equiv 0 \bmod 4$, we must also account for the contribution from the previous to last line of \eqref{rel4}, and
	\begin{multline*}
		\sum_{\substack{1\leq k \leq \nu\log D \\ k \equiv 0 \bmod 4 }} \sum_{\substack{p^k \leq D^{2\nu} \\ p \equiv 1 \bmod 4 \\ (p,M)=1}}	\sum_{\substack{ \lvert m\vert  \ll \frac{p^k M}{D}  \\ (\pi,m)\neq 1}}     \frac{\lvert g_4\big(8 i^{3(\ell+1)} M^3 m, \pi^k \big)\rvert }{ p^{\frac{3}{2}k}}   \log p \\
		\ll \sum_{\substack{1\leq k \leq \nu\log D \\ k \equiv 0 \bmod 4 }} \sum_{\substack{p^k \leq D^{2\nu} \\ p \equiv 1 \bmod 4 \\ (p,M)=1}}	\sum_{\substack{ \lvert m\vert  \ll \frac{p^k M}{D}  \\ p^k \mid m}}     \frac{p^k }{ p^{\frac{3}{2}k}}   \log p  +  M D^{\frac{\nu}{2} -1} \log D
		 \ll  M D^{\frac{\nu}{2} -1} \log D.
	\end{multline*}
	This completes the proof of Lemma~\ref{Lem bound non-coprimality}.
\end{proof}

\section{Proof of the main Theorems} \label{section-main}

Let  $r \in \Z[i]$, $\beta \in \{ 1, 1 + \lambda^3 \}$, $X, Y > 0$, and
\begin{equation}\label{Def_H}
	H_\beta(X, Y, r) := \sum_{\substack{c \equiv \beta \bmod 4\\ (r,c) =1}} \frac{\overline{c}}{|c|} \; \frac{g_4(r, c)}{\N(c)^{\frac{3}{2}}} \; \Lambda(c)  \; R_{X, Y}(\N(c)) 	
\end{equation}
where $\Lambda(c)= \log \N(\pi)$ when $c=\pi^k$ and 0 otherwise, and  
\begin{equation} \label{DefR_X}
	R_{X, Y}(t)  :=\widehat{\phi} \left( \frac{\log{t}}{\log X} \right)  \widehat{w} \left( \frac{Y}{t} \right).
\end{equation}
We denote by $\widetilde R_{X, Y}(s)$ the Mellin transform 
\begin{align} \label{mellin-transform}
	\widetilde R_{X,Y}(s) := \int_0^\infty y^{s-1} R_{X,Y}(y) \d y.
\end{align}

\begin{lemma}\label{TransformBounds}
	Let $w$ satisfy the conditions in \eqref{w-conditions}, and let $\phi$ be an even Schwartz function such that the support of $\widehat{\phi}$ is contained in $(- \nu, \nu)$.
	Let $R_{X,Y}$ be the function defined by~\eqref{DefR_X}.
	Then, for every real $\sigma>-1$, 
	\begin{align*}
		\widetilde R_{X, Y}(\sigma+it) 
		\ll_\sigma 
		\begin{cases}
			X^{\nu\sigma}\cdot\min (1,|t|^{-3}) & \text{ if } \sigma > 0,\\
			\log X \cdot\min (1,|t|^{-3})& \text{ if } \sigma=0,\\
			\frac{1}{Y}  X^{\nu(\sigma+1)}\cdot\min (1,|t|^{-3}) & \text{ if } \sigma \in (-1,0).
		\end{cases}
	\end{align*}
\end{lemma}

\begin{proof}
	Since, by hypothesis, $\widehat\phi$ is supported on $(-\nu,\nu)$ and $\widehat{w}$ is compactly supported on $(-\eta,\eta)$, it follows that $R_{X, Y}(t) = 0$ unless 
	$t \in (X^{-\nu},X^{\nu}) \cap (Y/\eta, \infty),$
	and it suffices to consider the case $Y <  \eta X^{\nu}$.
	Then, for any $t\in \R$, we have
	\begin{align}\label{eq : bound R-tilde}
		\widetilde R_{X,Y}(\sigma+it) = \int_{\max(X^{-\nu},\frac{Y}{\eta})}^{X^{\nu}} R_{X, Y}(y) y^{\sigma+it-1} \d y 
		\ll \int_{\frac{Y}{\eta}}^{X^{\nu}} y^{\sigma-1}\d y \ll_\sigma 
		\begin{cases}
			X^{\nu\sigma} & \text{ if } \sigma > 0,\\
			\log X & \text{ if } \sigma=0,\\
			\frac{1}{Y}  X^{\nu(\sigma+1)} & \text{ if } \sigma \in (-1,0).
		\end{cases}
	\end{align}
	Moreover, if $|t|\geq 1$, integrating by parts three times,
	\begin{align*}
		\widetilde R_{X, Y}(\sigma+it) 
		&= \int_{\max(X^{-\nu},\frac{Y}{\eta})}^{X^{\nu}} \widehat\phi\bigg(\frac{\log y}{\log X}\bigg)\widehat w\bigg(\frac{Y}{y}\bigg) y^{\sigma+it-1} \d y \\
		&\ll \int_{\frac{Y}{\eta}}^{X^{\nu}} \bigg| \frac{\d^3}{\d y^3}\bigg[\widehat\phi\bigg(\frac{\log y}{\log X}\bigg)\widehat w\bigg(\frac{Y}{y}\bigg)\bigg]\bigg| \frac{y^{\sigma+2}}{|t|^3} \d y.
	\end{align*} 
	Using $Y \ll y$, and the fact that $\widehat\phi, \widehat w$ and their derivatives are smooth compactly supported hence bounded functions, one has
	\begin{align*}
		\frac{\d^3}{\d y^3}\bigg[\widehat\phi\bigg(\frac{\log y}{\log X}\bigg)\widehat w\bigg(\frac{Y}{y}\bigg)\bigg] 
		\ll \frac{1}{y^3} \sum_{i+j \leq 3}  \bigg| \widehat \phi ^{(i)}\bigg(\frac{\log y}{\log X}\bigg)\bigg| \cdot\bigg| \widehat w^{(j)}\bigg(\frac{Y}{y}\bigg)\bigg|
		\ll \frac{1}{y^3}.
	\end{align*}
	Therefore,
	\begin{align*}
		\widetilde R_{X,Y}(\sigma+it) 
		&\ll  \frac{1}{|t|^3} \int_{\frac{Y}{\eta}}^{X^{\nu}} y^{\sigma -1}\d y
	\end{align*}
	for which we use \eqref{eq : bound R-tilde}, obtaining the stated bound.
\end{proof}

The rest of the paper consists of a proof of an upper bound for $H_\beta(X, Y, r)$.

\begin{proposition}\label{BoundForHAfterLindelof}
	Let $w$ satisfy the conditions in \eqref{w-conditions}, and let $\phi$ be an even Schwartz function such that the support of $\widehat{\phi}$ is contained in $(- \nu, \nu)$. 
	Let $\beta \in \lbrace 1, 1 + \lambda^3\rbrace$ and $H_\beta$ be as defined in~\eqref{Def_H}.
	For any $1 \leq U \leq X$ and any $\varepsilon > 0$,
	\begin{align}\label{H_bound_types}
		H_\beta(X, Y, r)  &\ll  
		\frac{X^{\frac{\nu}{2}+\varepsilon} U^{\frac12} \N(r)^{\frac14+\varepsilon }}{Y} 
		+ X^{\varepsilon}  \bigg(\frac{U}{Y^{\frac12}} + \frac{X^{\frac{\nu}{2}}}{Y^{\frac12} U^{\frac12}} \bigg).
	\end{align}
\end{proposition}

Assuming Proposition \ref{BoundForHAfterLindelof}, we now prove Theorem~\ref{main-theorem OLD}, Theorem~\ref{main-theorem} and Corollary~\ref{Corollary non vanishing}.

\begin{proof}[Proof of Theorem~\ref{main-theorem OLD}]
	As noticed earlier, $H_\beta(X, Y, r)=0$ unless $0 < Y < \eta X^\nu$, and we may therefore restrict to this case.
	Then, using the bound~\eqref{inert-primes}, Lemma~\ref{Lem bound non-coprimality} and Remark~\ref{Proposition_reduction}, the proof of Theorem~\ref{main-theorem OLD} reduces to showing that there exists $\delta >0$ such that
	\begin{align}\label{reduction to H}
		\sum_{\substack{|m| < {16 \eta M D^{2 \nu-1}}}}   \big\lvert H_{\beta}(D^2,\frac{mD}{16M},8i^{3(\ell+1)}M^3 m) \big\rvert
		\ll \frac{M^{3}}{D^{\delta}}.
	\end{align}
	for any $\beta \in \{ 1, 1 + \lambda^3 \}$, and uniformly for $M < D^{\varepsilon}$. 

	With the bound of Proposition \ref{BoundForHAfterLindelof}, we get 
	\begin{align*}
		&\sum_{\substack{m \ll D^{2 \nu-1} M}}  \left| H_\beta \left( D^2, \frac{mD}{16 M}, 8 i^{3(\ell+1)} M^3 m \right)  \right| \\
		&\ll (MD)^{\varepsilon} \left( M^{\frac52}  D^{\nu - 1} U^{\frac12}  + M^{\frac12}D^{-\frac12}   U 
		+ M^{\frac12} D^{ \nu -\frac12}  U^{-\frac{1}{2}} \right) \sum_{m \ll D^{2 \nu -1} M} \frac{1}{m^{\frac12}}  \\
		&\ll D^{\varepsilon} M^{3+\varepsilon} \left( D^{2\nu - \frac32} U^{\frac12} + {D^{\nu-1}}  U + {D^{2\nu - 1}} U^{-\frac{1}{2}} \right).
	\end{align*}
	To make each contribution $o(1)$, we choose $U = D^\theta$. This requires
		\begin{align*}
		\begin{cases}
			2\nu-\frac32+\frac{\theta}{2} < 0 \\
			\nu-1+\theta<0 \\
			2\nu-1-\frac{\theta}{2} < 0
		\end{cases}
		\end{align*}	
	from which
	\begin{align*}
		\begin{cases}
			4\nu-2 < 3-4\nu \\
			4\nu-2 < 1-\nu.
		\end{cases}
	\end{align*}	
	This yields $\nu< \frac{3}{5}$, and concludes the proof of Theorem~\ref{main-theorem OLD}. 
\end{proof}
\begin{remark}\label{Type_II_contr}
Note that the limitation on the support comes entirely from the second term in \eqref{H_bound_types}, i.e. from the contribution of the Type II sums (Proposition~\ref{TypeII_bound}).
\end{remark}
\begin{proof}[Proof of Theorem~\ref{main-theorem}]
	Using Theorem~\ref{main-theorem OLD}, we write
	\begin{align*}
		\frac{1}{\mathcal{S}_\mathcal{F^*}(w, D)}  \sum_{d \in \cF^{*}} r(E_d) w \left( \frac{d}{D} \right)
		\leq \frac{1}{\mathcal{S}_\mathcal{F^*}(w, D)}  \sum_{d \in \cF^*} \mathcal{D}(\phi,E_d) \xrightarrow[D\rightarrow\infty]{} \widehat\phi(0) + \frac{\phi(0)}{2}
	\end{align*}
	for any non-negative, even Schwartz function $\phi$ for which $\phi(0) \geq 1$ and $\mathrm{supp}(\widehat{\phi}) \subset (-\tfrac35,\tfrac35)$.
	As in \cite{ILS} (see also~\cite[section~7]{DDW}), we pick a sequence of such functions whose limit equals
	$$ \phi_{\nu}(x) := \left( \frac{\sin{(\pi \nu x)}}{\pi \nu x} \right)^2, \quad 	\widehat\phi_\nu(t) =  \begin{cases}  \frac{\nu-|t|}{\nu^2} & \text{if $|t| < \nu$} \\    0 & \text{otherwise,} \end{cases}. $$
	for $\nu = \tfrac35$ or $1$, from which the result follows.
\end{proof}

\begin{proof}[Proof of Corollary~\ref{Corollary non vanishing}] 

For any $\varepsilon >0$,
we choose a non-negative weight function $w_\varepsilon$ satisfying the conditions in \eqref{w-conditions}  such that
	$w_\varepsilon(t) \geq 1$ for $t \in [\frac12, 1]$, 
	and $\widehat{w}_\varepsilon(0) = \frac12 + \varepsilon.$
Let
\begin{align*} 
\mathcal{N}_{\cF^{-}} [a,b] := \# \{ d \in \cF^{-} \;:\; a \leq d \leq b \}.
\end{align*}
Since $\mathcal{N}_{\cF^{-}} [1, D] \sim \frac{4}{15 \zeta(4)} D$, we get
using  ~\eqref{Eq total sum} that
	\begin{align*}
\mathcal{S}_{\cF^{-}} (w_\varepsilon, D) &= ( \mathcal{N}_{\cF^{-}} [1, D]  + o(D) ) \widehat{w}_\varepsilon(0) + o(D) \\ 
&= ( \mathcal{N}_{\cF^{-}} [D/2, D] )(1 + 2 \varepsilon) + o(D) 
\end{align*}
		We have
	\begin{align*}
		{\#\lbrace d \in \cF^-, D/2 \leq d \leq D : \mathrm{ord}_{s=1} L(s,E_d)} >1 \rbrace  &\leq 
		\sum_{d \in \cF^- \cap [D/2, D]} \frac{r(E_{d})-1}{2}w_\varepsilon\bigg(\frac{d}{D}\bigg),
	\end{align*}
	and by Theorem~\ref{main-theorem}, we deduce that
	\begin{align*}
		\#\lbrace d \in \cF^-, D/2 \leq d \leq D : \mathrm{ord}_{s=1} {L(s,E_d)} >1 \rbrace 
				&\leq \bigg(\frac{13}{12} - \frac12\bigg) {\mathcal{S}_{\cF^-}(w_\varepsilon,D)} + o(D) \\
		&\leq \frac7{12} \mathcal{N}_{\cF^{-}} [D/2, D]  (1 + 2 \varepsilon)  + o(D).
	\end{align*}
		Summing dyadically on both sides, we get
	\begin{align*}
	\frac{{\#\lbrace d \in \cF^-, 1 \leq d \leq D : \mathrm{ord}_{s=1} L(s,E_d)} >1 \rbrace }{\#\lbrace d \in \cF^-, 1 \leq d \leq D \}  }
	 &\leq \frac{7}{12} +\varepsilon + o(1),
	\end{align*}
for any $\varepsilon>0$. The first statement of the corollary follows. The proofs of the other two statements are similiar.
	\end{proof}

\section{Vaughan's Identity} \label{Vaughan}

We now record Vaughan's identity~\cite{Vau}.
Our setting is analogous to \cite[\S2]{HB2000} and \cite{DDHL}.
Let $X, Y > 0$, we use the notation $R_{X,Y}$ as in~\eqref{DefR_X}.
For $1 \leq U \leq X$, $j \in \{0,1,2^{\prime},2^{\prime \prime},3,4\}$,  $r \in \Z[i]$, and $\beta \in \{ 1, 1 + \lambda^3 \}$, we define
\begin{equation} \label{sigmajl}
	\Sigma_{j,\beta} (X,Y, r,U):=\sum_{\substack{  a,b,c \equiv 1 \bmod{\lambda^3} \\ abc \equiv \beta \bmod{4}  \\ (abc,r)=1}}
	\Lambda(a) \mu(b) 
	\frac{g_4(r,abc)}{\N(abc)^{\frac32}}
	\frac{\overline{abc}}{|abc|} R_{X,Y} (\N(abc))
\end{equation}
where $a,b,c \in \mathbb{Z}[i]$ are subject to the conditions: 
\begin{align*}
	\N(bc) & \leq U,  \quad \quad \quad \quad \quad \quad  j=0;  \\ 
	\N(b) & \leq U,  \quad \quad \quad \quad \quad \quad   j=1;  \\ 
	\N(ab) & \leq U,  \quad \quad \quad \quad \quad \quad j=2';  \\ 
	\N(a),\N(b) \leq U &<\N(ab),   \quad \quad \quad \quad j=2'';  \\ 
	\N(b) \leq U&<\N(a),\N(bc),   \quad \hspace{0.2cm} j=3;  \\ 
	\N(a),\N(bc) & \leq U, \quad \quad \quad \quad \quad \quad j=4. 
\end{align*}
Throughout, we always have $a,b,c \equiv 1 \bmod{\lambda^3}$ even if not mentioned explicitly. 
Recall that, as discussed in the proof of Lemma~\ref{TransformBounds}, the support conditions on $\widehat{\phi}$ and $\widehat{w}$ imply that every $\Sigma_{j,\beta} (X,Y, r,U)$ vanishes unless $Y < \eta X^{\nu}$. Accordingly, this bound is assumed implicitly in the following.
Then Vaughan's identity reads,
\begin{align} \begin{split} \label{Vauid}
	\Sigma_{0,\beta}(X,Y, r,U) + \Sigma_{2',\beta}(X,Y,r,U) + \Sigma_{2^{\prime \prime},\beta}(X,Y,r,U)+\Sigma_{3,\beta}(X,Y,r,U)\\
	=\Sigma_{1,\beta}(X,Y,r,U)+\Sigma_{4,\beta}(X,Y,r,U).
\end{split} \end{align}

We now proceed to bound each $\Sigma_{j,\beta}(X,Y,r,U)$ contribution, independently.  We start by bounding the term $\Sigma_{4,\beta}(X,Y,r,U)$ trivially.

\begin{lemma}\label{bound S4}
	Let $w$ satisfy the conditions in \eqref{w-conditions}, and let $\phi$ be an even Schwartz function such that the support of $\widehat{\phi}$ is contained in $(- \nu, \nu)$.
	Let $X, Y > 0$,  $r \in \Z[i]$, and $\beta \in \{ 1, 1 + \lambda^3 \}$, and $\Sigma_{j,\beta}$ as defined in~\eqref{sigmajl}.
	For any $U\leq X$, we have that 
	\[ \Sigma_{0,\beta}(X,Y, r,U) = H_{\beta}(X,Y, r),\]
	and 
	\[ \Sigma_{4,\beta}(X,Y, r,U) \ll \frac{U}{Y} .\]
\end{lemma}

\begin{proof}
	By M\"{o}bius inversion, we first note that
	\begin{align*}
		\Sigma_{0,\beta}(X,Y,r,U)
		&= \sum_{\substack{n \equiv \beta \bmod 4\\ (r,n) =1}} \frac{\overline{n}}{|n|} \; \frac{g_4(r, n)}{\N(n)^{\frac{3}{2}}} R_{X, Y}(\N(n))\sum_{\substack{a \mid n \\ \N(\frac{n}{a}) \leq U }}\Lambda(a) \sum_{b\mid \frac{n}{a}  } \mu(b) \\
		&= \sum_{\substack{n \equiv \beta \bmod 4\\ (r,n) =1}} \frac{\overline{n}}{|n|} \; \frac{g_4(r, n)}{\N(n)^{\frac{3}{2}}} R_{X, Y}(\N(n))\Lambda(n) 
		=  H_{\beta}(X,Y,r).
	 \end{align*}
	
	Similarly, recalling the definition \eqref{DefR_X} of $R_{X, Y}$ and using the fact that $\widehat\phi$ is bounded, we obtain the trivial bound
	\begin{equation*}\begin{split}\label{S_4.28Jan}
		\Sigma_{4,\beta}(X,Y,r,U)
		&= \sum_{\substack{n \equiv \beta \bmod 4\\ (r,n) =1}} \frac{\overline{n}}{|n|} \; \frac{g_4(r, n)}{\N(n)^{\frac{3}{2}}} R_{X, Y}(\N(n))\sum_{\substack{a \mid n \\ \N(a) \leq U \\ \N(\frac{n}{a}) \leq U \\  }}\Lambda(a) \sum_{b\mid \frac{n}{a}   } \mu(b)   \\
		&\ll  \sum_{\N(n) \leq U} \frac{\Lambda(n)}{\N(n)}  \bigg|\widehat{w} \left( \frac{Y}{\N(n)} \right)\bigg|.
	\end{split}\end{equation*}
	Since $\widehat w$ is bounded and supported on $(-\eta,\eta)$, the prime number theorem yields
	\begin{equation*}
		\Sigma_{4,\beta}(X,Y,r,U) \ll
		\sum_{ \frac{Y}{\eta} \leq \N(n) \leq U} \frac{\Lambda(n)}{\N(n)} \ll \frac{1}{Y} \sum_{{Y} \leq \N(n) \leq U} \Lambda(n) \ll \frac{U}{Y},
	\end{equation*}
	which concludes the proof.
\end{proof}

To bound the type~I sums $\Sigma_{1, \beta}(X, Y, r, U)$ and $\Sigma_{2', \beta}(X, Y, r, U)$, we first integrate by parts to treat the logarithmic weights. 

\begin{lemma}\label{smooth integration by parts}
	Let $R_{X, Y}$ be a smooth function compactly supported on $(X^{-\nu},X^\nu)$ as in \eqref{DefR_X}, 
	and let {$(b_n)_{n \in \Z[i]}$} be a sequence of complex numbers with $\lvert b_n\vert \leq 1$.  Then, we have
	\begin{align*}
		\sum_{n \in \Z[i]} b_n R_{X, Y}(\N(n))\log \N(n)  
		= &\frac{\log X^\nu }{2\pi i}\int_{(2)} \mathcal{B}(s) \widetilde R_{X, Y}(s)\d s \\
		&+ \frac{1}{4\pi^2 }\int_{(2)}\int_{(2)} \mathcal{B}(s+s') X^{\nu s'}\widetilde R_{X, Y}(s)\frac{\d s'}{s'^2}\d s,
	\end{align*}
where  $\widetilde R_{X, Y}(s)$ is the Mellin transform \eqref{mellin-transform}
and
\[ \mathcal{B}(s) := \sum_{n \in \Z[i]}  b_n \N(n)^{-s}, \quad \quad \re(s) > 1.\]
\end{lemma}

\begin{proof}
	By partial summation, we note that
	\begin{align*}
		\sum_{\N(n) \leq X^\nu} b_n \N(n)^{-s}\log \N(n)  
		= \log X^\nu \sum_{\N(n) \leq X^\nu} b_n \N(n)^{-s} - \int_{0}^{X^\nu}\sum_{\N(n) \leq t} b_n \N(n)^{-s}\frac{\d t}{t}.
	\end{align*}
	Moreover, by Mellin inversion,
	\begin{align*}
		\sum_{n \text{ primary}}& b_n R_{X, Y}(\N(n))\log \N(n) 
		= \sum_{\N( n) \leq X^\nu} b_n\log \N(n) \frac{1}{2\pi i}\int_{(2)} \N(n)^{-s} \widetilde R_{X, Y}(s) \d s\\
		&= \frac{\log X^\nu }{2\pi i}\int_{(2)}\sum_{\N(n) \leq X^\nu}\frac{b_n }{\N(n)^{s}} \widetilde R_{X, Y}(s) \d s  
		- \frac{1}{2\pi i}\int_{(2)}\int_{0}^{X^\nu}\sum_{\N(n) \leq t}\frac{b_n }{\N(n)^{s}}\frac{\d t}{t} \widetilde R_{X, Y}(s)\d s.
	\end{align*}
	The first term on the second line above is
	\begin{align*} \begin{split}
		\frac{\log X^\nu }{2\pi i}\int_{(2)}\sum_{\N(n) \leq X^\nu} b_n \N(n)^{-s} \widetilde R_{X, Y}(s)\d s 
		=  \log X^\nu \sum_{\N(n) \leq X^\nu} b_n R_{X, Y}(\N(n)) \\
		=  \log X^\nu \sum_{n \in \Z[i]}  b_n R_{X, Y}(\N(n)) 
		= \frac{\log X^\nu }{2\pi i}\int_{(2)} \mathcal{B}(s) \widetilde R_{X, Y}(s)\d s, 
	\end{split}\end{align*}
	where we note that the sum may be completed since $R_{X, Y}$ is compactly supported.  
	Finally, Perron's formula yields
	\begin{align*}
		-\frac{1}{2\pi i}\int_{(2)}\int_{0}^{X^\nu}\sum_{\N(n) \leq t} b_n \N(n)^{-s}\frac{\d t}{t} \widetilde R_{X, Y}(s)\d s
		&= 	\frac{1}{4\pi^2 }\int_{(2)}\int_{(2)}\int_{0}^{X^\nu} \mathcal{B}(s+s') t^{s'} \frac{\d t}{t}\widetilde R_{X, Y}(s)\frac{\d s'}{s'}\d s \\
		&= \frac{1}{4\pi^2 }\int_{(2)}\int_{(2)} \mathcal{B}(s+s') X^{\nu s'}\widetilde R_{X, Y}(s)\frac{\d s'}{s'^2}\d s.
	\end{align*}
	Combining the two terms gives the result.
\end{proof}

We now introduce some notation for Dirichlet series of Gauss sums. 
For $s \in \C$, $r \in \Z[i]$, and $\beta \in \{1, 1 + \lambda^3 \}$, we define
\begin{align} \label{generating-phi}
	\G_{\beta}(s, r) :&=  \sum_{\substack{c \equiv \beta \bmod 4}} \frac{g_4(r, c)}{\N(c)^{\frac{1}{2} + s}} \frac{\overline{c}}{|c|}.
\end{align}
By \eqref{gauss_sum_bound} and \eqref{rel4}, 
$\G_{\beta}(s, r)$  converges  absolutely for $\re(s) > 1$.
We also define for $s \in \C$, $\beta \in \{1, 1 + \lambda^3 \}$, and $v, r, a \in \Z[i]$ with $v|r$, 
\begin{align} \label{phi-var1}
	\G_{\beta}(s, v|r, a) &:= \sum_{\substack{c \equiv \beta \bmod 4\\ (c, v)=1\\ a \mid c}} \frac{g_4(r, c)}{\N(c)^{\frac{1}{2} + s}} \frac{\overline{c}}{|c|}.
\end{align}

\begin{lemma}\label{lemma_S_1}
	Let $w$ satisfy the conditions in \eqref{w-conditions}, and let $\phi$ be an even Schwartz function such that the support of $\widehat{\phi}$ is contained in $(- \nu, \nu)$.
	Let $X, Y > 0$,  $r \in \Z[i]$, $\beta \in \{ 1, 1 + \lambda^3 \}$, $R_{X,Y}$ as defined in~\eqref{DefR_X} and $\Sigma_{1,\beta}$ as defined in~\eqref{sigmajl}.
	Then for $U \leq X$, we have
	\begin{align*}
		\Sigma_{1,\beta}(X,Y,r,U) \ll &\log X \sum_{\N(a) \leq U}\mu^{2}(a)\bigg \lvert\int_{(2)}\G_{\beta}(s, r|r, a) \widetilde R_{X, Y}(s-1)\d s \bigg \rvert \\
		& + \sum_{\N(a) \leq U}\mu^{2}(a)\bigg \lvert\int_{(2)}\int_{(2)} \G_{\beta}(s+s', r|r, a) X^{\nu s'}\widetilde R_{X, Y}(s-1)\frac{\d s'}{s'^2}\d s\bigg \rvert,
	\end{align*}
\end{lemma}

\begin{proof}
	The proof follows similarly to \cite[pp. 102$-$103]{HB2000}.  A similar result for cubic characters may moreover be found in \cite[Lemma 5.2]{DG}.  To begin, employing M\"obius inversion for the von Mangoldt function, we note that
	\begin{align*}
		|\Sigma_{1,\beta}(X,Y, r,U)|  &= \bigg|\sum_{\substack{n \equiv \beta \bmod 4\\ (r,n) =1}} \frac{\overline{n}}{|n|} \; \frac{g_4(r, n)}{\N(n)^{\frac{3}{2}}} \;R_{X, Y}(\N(n)) \sum_{\substack{b\mid n \\ \N(b) \leq U}}\mu(b) \sum_{a \mid \frac{n}{b}}\Lambda(a) \bigg| \\
	&\leq \sum_{\N(b) \leq U}\mu^{2}(b)\bigg \lvert\sum_{\substack{n \equiv \beta \bmod 4\\ (r,n) =1  \\ b\mid n}}\frac{\overline{n}}{|n|} \; \frac{g_4(r, n)}{\N(n)^{\frac{3}{2}}} \;R_{X, Y}(\N(n))\left(\log \N(n)-\log \N(b)\right)\bigg \rvert.
	\end{align*}

	Since $R_{X, Y}$ is supported on $(X^{-\nu},X^{\nu})$, by Mellin inversion, we write
	\[\sum_{\substack{n \equiv \beta \bmod 4\\ (r,n) =1  \\ b\mid n }} \frac{\overline{n}}{|n|} \; \frac{g_4(r, n)}{\N(n)^{\frac{3}{2}}} R_{X, Y}(\N(n)) \log \N(b) = \frac{\log \N(b) }{2\pi i}\int_{(2)}\G_{\beta}(s, r|r, b)\widetilde R_{X, Y}(s-1)\d s\]
	while by Lemma \ref{smooth integration by parts}, we have
	\begin{align*}
		\sum_{\substack{n \equiv \beta \bmod 4\\ (r,n) =1  \\ b\mid n}} \frac{\overline{n}}{|n|} \; \frac{g_4(r, n)}{\N(n)^{\frac{3}{2}}} \;R_{X, Y}(\N(n))&\log \N(n)= \frac{\log X^\nu }{2\pi i}\int_{(2)}\G_{\beta}(s, r|r, b) \widetilde R_{X, Y}(s-1)\d s \\
		& + \frac{1}{4\pi^2 }\int_{(2)}\int_{(2)} \G_{\beta}(s+s', r|r, b) X^{\nu s'}\widetilde R_{X, Y}(s-1)\frac{\d s'}{s'^2}\d s,
	\end{align*}
	from which the proof follows.
\end{proof}

\begin{lemma}\label{Vaughan_bounds}
	Let $w$ satisfy the conditions in \eqref{w-conditions}, and let $\phi$ be an even Schwartz function such that the support of $\widehat{\phi}$ is contained in $(- \nu, \nu)$.
	Let $X, Y > 0$,  $r \in \Z[i]$, $\beta \in \{ 1, 1 + \lambda^3 \}$, $R_{X,Y}$ as defined in~\eqref{DefR_X} and $\Sigma_{2',\beta}$ as defined in~\eqref{sigmajl}.
	Then for $U \leq X$, we have
	\[\lvert \Sigma_{2',\beta}(X,Y, r,U)\rvert \leq \log U \sum_{\N(a) \leq U}\mu^{2}(a)\bigg \lvert\int_{(2)}\G_{\beta}(s, r|r, a) \widetilde R_{X, Y}(s-1) \d s \bigg \rvert.\] 
\end{lemma}

\begin{proof}
	We write
	\begin{align*}
		\Sigma_{2',\beta}(X,Y,r,U) &=\sum_{\substack{n \equiv \beta \bmod 4\\ (r,n) =1}} \frac{\overline{n}}{|n|} \; \frac{g_4(r, n)}{\N(n)^{\frac{3}{2}}} \; R_{X, Y}(\N(n) )\sum_{\substack{\N(ab) \leq U\\ ab|n}}\Lambda(a)\mu(b)\\
		&=\sum_{\N(a) \leq U}\sum_{\substack{n \equiv \beta \bmod 4\\ (r,n) =1\\ a\mid n}} \frac{\overline{n}}{|n|} \; \frac{g_4(r, n)}{\N(n)^{\frac{3}{2}}} \;R_{X, Y}(\N(n))\sum_{b|a}\Lambda(\tfrac{a}{b})\mu(b).
	\end{align*}
	Noting that
	\begin{align*}
		\big\lvert \sum_{b|a}\Lambda(\tfrac{a}{b})\mu(b)\big\rvert \leq \log \N(a)\mu^{2}(a),
	\end{align*}
	it follows by Mellin inversion that
	\begin{align*}
		|\Sigma_{2',\beta}(X,Y, r,U)|
		&\leq \sum_{\substack{\N(a)\leq U}}\log \N(a)\mu^{2}(a)\bigg \lvert\frac{1}{2\pi i}\int_{(2)}\G_{\beta}(s, r|r, a) \widetilde R_{X, Y}(s-1)\d s\bigg \rvert.
	\end{align*}
This is the desired bound.
\end{proof}

Next, we seek a bound on the Type~II (i.e. bilinear) sums, $\Sigma_{2'',\beta}(X,Y,r,U)$ and $\Sigma_{3,\beta}(X,Y,r,U)$.  We treat the two sums in parallel. 

\begin{proposition}\label{TypeII_bound}
	Let $w$ satisfy the conditions in \eqref{w-conditions}, and let $\phi$ be an even Schwartz function such that the support of $\widehat{\phi}$ is contained in $(- \nu, \nu)$.
	Let $X, Y > 0$,  $r \in \Z[i]$, $\beta \in \{ 1, 1 + \lambda^3 \}$, $R_{X,Y}$ as defined in~\eqref{DefR_X} and $\Sigma_{j,\beta}$ as defined in~\eqref{sigmajl}.
	Then for $U \leq X$
	and for any $\varepsilon >0$, we have
	$$|\Sigma_{2'',\beta}(X,Y,r,U)| + |\Sigma_{3,\beta}(X,Y,r,U)|
	\ll X^{\varepsilon}  \bigg(\frac{U}{Y^{\frac12}} + \frac{X^{\frac\nu2}}{Y^{\frac12}U^{\frac12}} \bigg) .$$
\end{proposition}

\begin{proof}
	The proof follows similarly to \cite[Prop. 7.1]{DDHL}, which in turn is based on \cite[Lemma 2]{HB2000} (see also \cite[Lemma 5.2]{DG}).  We begin by noting that 
	\begin{align*}
		\Sigma_{2'',\beta}(X,Y, r,U)&=\sum_{\substack{n \equiv \beta \bmod 4\\ (r,n) =1}} \frac{\overline{n}}{|n|} \; \frac{g_4(r, n)}{\N(n)^{\frac{3}{2}}} \;R_{X, Y}(\N(n)) \sum_{\substack{\N(a),\N(b) \leq U\\ \N(ab) > U\\ ab|n}}\Lambda(a)\mu(b)\\
		&=\sum_{\substack{v,w \textnormal{ primary}\\
		vw \equiv \beta \bmod 4\\ (r,vw) =1\\ \N(v) > U}}\frac{\overline{vw}}{|vw|} \; \frac{g_4(r, vw)}{\N(vw)^{\frac{3}{2}}} \;R_{X, Y}(\N(vw)) \sum_{\substack{\N(a),\N(b) \leq U\\ ab=v}}\Lambda(a)\mu(b).
	\end{align*}
	By \eqref{bilaw} and \eqref{rel2}, we note that
	\[g_4(r, vw) = \quartic{v}{w}\quartic{w}{v}g_{4}(r,v)g_{4}(r,w)=(-1)^{C(v,w)}\left(\frac{w}{v}\right)_{2}g_{4}(r,v)g_{4}(r,w).\]
	It thus follows that
	\begin{align*}
		\Sigma_{2'',\beta}(X,Y,r,U)&=\sum_{\substack{v,w \textnormal{ primary}\\
		vw \equiv \beta \bmod 4\\ \N(v) > U}}(-1)^{C(v,w)}A_{2''}(v)B_{2''}(w)\left(\frac{w}{v}\right)_{2}R_{X, Y}( \N(vw)),
	\end{align*}
	where
	\begin{align*}
		A_{2''}(v)&:=\mathds{1}_{\{(v,r)=1\}}\frac{\overline{v}}{|v|} \; \frac{g_4(r, v)}{\N(v)^{\frac{3}{2}}} \sum_{\substack{\N(a),\N(b) \leq U\\ ab=v}}\Lambda(a)\mu(b)\\
		B_{2''}(w)&:=\mathds{1}_{\{(w,r)=1\}}\frac{\overline{w}}{|w|} \; \frac{g_4(r, w)}{\N(w)^{\frac{3}{2}}}.
	\end{align*}
	We moreover note that $A_{2''}(v) = 0$ whenever $\N(v) > U^{2}$.
	Similarly, we write
	\begin{align*}
		\Sigma_{3,\beta}(X,Y,r,U)&=\sum_{\substack{n \equiv \beta \bmod 4\\ (r,n) =1}} \frac{\overline{n}}{|n|} \; \frac{g_4(r, n)}{\N(n)^{\frac{3}{2}}} 
		\;R_{X, Y} ( \N(n) )\sum_{\substack{ a\mid n \\ \N(a) > U \\ \N(\frac{n}{a})>U}}\Lambda(a)\sum_{\substack{b|\frac{n}{a} \\ \N(b)\leq U}}\mu(b)\\
		&=\sum_{\substack{v,w \textnormal{ primary}\\
		vw \equiv \beta \bmod 4\\ \N(v) > U}}(-1)^{C(v,w)}A_{3}(v)B_{3}(w)\left(\frac{w}{v}\right)_{2} R_{X, Y}(\N(vw)),
	\end{align*}
	where
	\begin{align*}
		A_{3}(v)&:=\mathds{1}_{\{(v,r)=1\}}\frac{\overline{v}}{|v|} \; \frac{g_4(r, v)}{\N(v)^{\frac{3}{2}}} \Lambda(v)\\
		B_{3}(w)&:=\mathds{1}_{\{(w,r)=1\}}\mathds{1}_{\{\N(w)> U\}} \frac{\overline{w}}{|w|} \; \frac{g_4(r, w)}{\N(w)^{\frac{3}{2}}}\sum_{\substack{b|w\\ \N(b)\leq U}}\mu(b).
	\end{align*}
	By \eqref{gauss_sum_bound}, the functions $A_{2''}, A_{3}, B_{2''}$, and $B_{3}$ are each supported on square-free elements of~$\Z[i]$.

	To estimate $\Sigma_{j,\beta}(X,Y, r,U)$ for $j \in \{2'',3\}$ we dyadically partition by setting
	\begin{equation}\label{VW_def}
		V_{k}:=2^{k}U \textnormal{ and } W_{\ell,k}:=\frac{Y}{\eta} \frac{2^{\ell}}{2V_{k}}.
	\end{equation}
	For $\Sigma_{2'',\beta}(X,Y, r,U)$, since the sum over $v$ may be restricted to $U < \N(v) < U^{2}$, we find that
	\begin{align}\label{S''_bound}\begin{split}
		|\Sigma_{2'',\beta}(X,Y, r,U)|&\leq \sum_{\substack{0 \leq k \leq \left\lceil \tfrac{\log U}{\log 2}\right\rceil \\ 0 \leq \ell \leq \frac{\log (2X^{2\nu})}{\log 2}}}\sum_{\substack{
		\eta,\gamma \in \{1,1+\lambda^3\} \\ \eta\gamma \equiv \beta \bmod{4}}}\bigg\lvert \sum_{\substack{\N(v) \in (V_{k},V_{k+1}]\\
		\N(w) \in [W_{\ell,k},W_{\ell+1,k})\\
		v \equiv \eta \bmod 4\\ w \equiv \gamma \bmod 4}}A(v)B(w)\left(\frac{w}{v}\right)_{2}R_{X, Y}( \N(vw))\bigg\rvert\\
		&\ll_{\nu} \log U \cdot \log X \cdot \max_{\substack{U < V \leq U^{2} \\ \frac{Y}{2\eta V} \leq W \leq \frac{X^\nu}{V}}}\bigg\lvert \sum_{\substack{\N(v) \in (V,2V]\\
		\N(w) \in [W,2W)\\
		v \equiv \eta \bmod 4\\ w \equiv \gamma \bmod 4}}A(v)B(w)\left(\frac{w}{v}\right)_{2}R_{X, Y}( \N(vw))\bigg\rvert,
	\end{split}\end{align}
	where $A(v):=A_{2''}(v)$ and $B(w):=B_{2''}(w)$.  Similarly for $\Sigma_{3,\beta}(X,Y, r,U)$, as the sum over $w$ may be restricted to $\N(w)>U$, we may restrict the sum over $v$  to $\N(v) <\frac{X^{\nu}}{U}$, and we obtain
	\begin{align}\label{S3_bound}\begin{split}
		|\Sigma_{3,\beta}(X,Y, r,U)|&\leq \sum_{\substack{0 \leq k \leq \left\lceil \tfrac{\log (X^\nu/U^{2})}{\log 2}\right\rceil \\ 0 \leq \ell \leq \frac{\log 2X^{2\nu}}{\log 2}}}\sum_{\substack{
		\eta,\gamma \in \{1,1+\lambda^3\} \\ \eta\gamma \equiv \beta \bmod{4}}}\bigg\lvert \sum_{\substack{\N(v) \in (V_{k},V_{k+1}]\\
		\N(w) \in [W_{\ell,k},W_{\ell+1,k})\\
		v \equiv \eta \bmod 4\\ w \equiv \gamma \bmod 4}}A(v)B(w)\left(\frac{w}{v}\right)_{2} R_{X, Y}( \N(vw) )\bigg\rvert\\
		&\ll_{\nu} \log U \cdot \log X \cdot \max_{\substack{U < V \leq \frac{X^\nu}{U} \\  \frac{Y}{2\eta V} \leq W \leq \frac{X^\nu}{V}}}\bigg\lvert \sum_{\substack{\N(v) \in (V,2V]\\
		\N(w) \in [W,2W)\\
		v \equiv \eta \bmod 4\\ w \equiv \gamma \bmod 4}}A(v)B(w)\left(\frac{w}{v}\right)_{2} R_{X, Y}(\N(vw)) \bigg\rvert,
	\end{split}\end{align}
	where $A(v):=A_{3}(v)$ and $B(w):=B_{3}(w)$.  
	We now bound the inner sum in~\eqref{S''_bound} and~\eqref{S3_bound} in parallel. For $A= A_{2''}$ or $A_3$ and $B= B_{2''}$ or $B_3$, we define the Dirichlet polynomials
	\[P_{\eta,\gamma}(s;V,W):=\sum_{\substack{
	\N(v) \in (V,2V]\\
	\N(w) \in [W,2W)\\
	v \equiv \eta \bmod 4\\ w \equiv \gamma \bmod 4}}A(v)B(w)\left(\frac{w}{v}\right)_{2}\N(vw)^{-s}, \quad \quad s \in \C.\]
	Then, by Mellin inversion, we have
	\begin{align}\label{TypeII_Mellin_Bound}\begin{split}
		\sum_{\substack{\N(v) \in (V,2V]\\
		\N(w) \in [W,2W)\\
		v \equiv \eta \bmod 4\\ w \equiv \gamma \bmod 4}}A(v)B(w)\left(\frac{w}{v}\right)_{2} R_{X, Y}(\N(vw)) &= \frac{1}{2\pi}\int_{\R}\widetilde R_{X, Y}(it)P_{\eta,\gamma}(it;V,W)\d t.
	\end{split}\end{align}
	The quadratic large sieve inequality over $\Z[i]$~\cite[Thm. 1]{Onodera}, states that for an arbitrary complex sequence $\{b_{w}\},$ any $V,W \geq 1$, and $\varepsilon > 0$, we have the bound
	\[\sum_{\substack{v \in \Z[i]\\ \N(v)\leq V\\ v \textnormal{ primary}}}\mu^{2}(v)\bigg \lvert \sum_{\substack{w \in \Z[i]\\ \N(w)\leq W \\ w  \textnormal{ primary}}}\mu^{2}(w)b_{w}\left(\frac{w}{v}\right)_{2}\bigg \rvert^{2}\ll_{\varepsilon}(VW)^{\varepsilon}(V+W)\sum_{\N(w)\leq W}|b_{w}|^{2}\mu^{2}(w).\]
	Together with the Cauchy--Schwarz inequality, it follows that for all $t \in \R$,
	\begin{align*}
		|P_{\eta,\gamma}(it;V,W)|^{2} 
		&\leq \sum_{\substack{\N(v) \in (V,2V]\\
		v \equiv \eta \bmod 4}}|A(v)|^{2}\sum_{\substack{\N(v) \in (V,2V]\\
		v \equiv \eta \bmod 4}}\mu^{2}(v)\bigg\vert\sum_{\substack{\N(w) \in [W,2W)\\
		w \equiv \gamma \bmod 4}}B(w)\N(w)^{-it}\left(\frac{w}{v}\right)_{2}\bigg|^{2}\\
		&\ll_{\varepsilon} \sum_{\substack{\N(v) \in (V,2V]\\
		v \equiv \eta \bmod 4}}|A(v)|^{2}\bigg((VW)^{\varepsilon}(V+W)\sum_{\substack{
		\N(w) \in [W,2W)\\
		w \equiv \gamma \bmod 4}}|B(w)|^{2}\mu^2(w)\bigg ).
	\end{align*}
	Using the bound~\eqref{gauss_sum_bound}, and upon noting that $d(n) \ll \N(n)^{\varepsilon}$ for any $\varepsilon > 0$, we find that $A(v) \ll V^{-1+\varepsilon}$ for all $\N(v) \in (V,2V]$ and $B(w)\ll W^{-1 +\varepsilon}$ for all $\N(w) \in [W,2W)$. In the range $VW \leq X^\nu$, it follows that for all $t \in \R$,
	\begin{align*}
		|P_{\eta,\gamma}(it;V,W)|^{2} &\ll_{\varepsilon} \frac{VX^{\varepsilon}}{V^2W^2}\bigg((VW)^{\varepsilon}(V+W)WX^{\varepsilon}\bigg ) \ll X^{\varepsilon} \frac{V+W}{VW}.
	\end{align*}
	Applying this bound in~\eqref{TypeII_Mellin_Bound} we thus find that
	\begin{align}\begin{split}
		\sum_{\substack{\N(v) \in (V,2V]\\
		\N(w) \in [W,2W)\\
		v \equiv \eta \bmod 4\\ w \equiv \gamma \bmod 4}} A(v)B(w)\left(\frac{w}{v}\right)_{2}R_{X,Y}( \N(vw))
		&\ll \int_{\R}|\widetilde R_{X, Y}(it)|\cdot |P_{\eta,\gamma}(it;V,W)|\d t\\
		&\ll X^{\varepsilon}(VW)^{-1/2}(V^{\frac{1}{2}}+W^{\frac{1}{2}}),
	\end{split}\end{align}
	where we used Lemma~\ref{TransformBounds} to bound the Mellin transform $\widetilde{R}_{X, Y}$. 

	Finally, coming back to~\eqref{S''_bound} we then conclude that
	\begin{align*}
		|\Sigma_{2'',\beta}(X,Y, r,U)|
		&\ll X^\varepsilon \cdot \max_{\substack{U < V \leq U^{2} \\ \frac{Y}{2\eta V} \leq W \leq \frac{X^\nu}{V}}}(V^{-\frac12} + W^{-\frac12})\\
		&\ll X^\varepsilon \cdot \max_{\substack{U < V \leq U^{2}}} \bigg (V^{-\frac12} + \bigg(\frac{V}{Y}\bigg)^{\frac12}\bigg) 
		\ll X^{\varepsilon} \bigg(\frac{U}{Y^{\frac12}} + U^{-\frac{1}{2}}\bigg).
	\end{align*}
	Similarly, in~\eqref{S3_bound} we find that
	\begin{align*}
		|\Sigma_{3,\beta}(X,Y,r,U)|
		&\ll X^\varepsilon \cdot \max_{\substack{U < V \leq \frac{X^{\nu}}{U} \\ \frac{Y}{2\eta V} \leq W \leq \frac{X^{\nu}}{V}}} (V^{-\frac{1}{2}}+W^{-\frac{1}{2}})
		\ll X^{\varepsilon}  U^{-\frac{1}{2}} \bigg( 1 + \frac{X^{\frac\nu2}}{Y^{\frac12}}\bigg).
	\end{align*}
	Recall that we can assume $Y < \eta X^\nu$, as $\Sigma_{j,\beta}$ vanishes otherwise and the bound is trivial, then the bound announced in the statement follows.
\end{proof}

Using Lemma~\ref{bound S4}, Lemma~\ref{lemma_S_1}, Lemma~\ref{Vaughan_bounds}, and Proposition~\ref{TypeII_bound} in \eqref{Vauid}, we arrive at the following.

\begin{proposition}\label{PropositionH}
	Let $w$ satisfy the conditions in \eqref{w-conditions}, and let $\phi$ be an even Schwartz function such that the support of $\widehat{\phi}$ is contained in $(- \nu, \nu)$.
	Let $X, Y > 0$,  $r \in \Z[i]$, $\beta \in \{ 1, 1 + \lambda^3 \}$, and $R_{X,Y}$ as defined in~\eqref{DefR_X}.
	Then for $1 \leq U \leq X$ and for any $\varepsilon >0$, we have
	\begin{align*}
		H_\beta(X,Y, r) \ll_{R,\varepsilon} &\log X \sum_{\N(a) \leq U}\mu^{2}(a)\bigg \lvert\int_{(2)}\G_{\beta}(s, r|r, a) \widetilde R_{X, Y}(s-1)\d s \bigg \rvert \\
		& + \sum_{\N(a) \leq U}\mu^{2}(a)\bigg \lvert\int_{(2)}\int_{(2)} \G_{\beta}(s+s', r|r, a) X^{\nu s'}\widetilde R_{X, Y}(s-1)\frac{\d s'}{s'^2}\d s\bigg \rvert\\
		& + X^{\varepsilon}  \bigg(\frac{U}{Y^{\frac12}} + \frac{X^{\frac\nu2}}{Y^{\frac12}U^{\frac12}} \bigg).
	\end{align*}
\end{proposition}

\section{The Dirichlet series of Gauss sums} \label{after-Vaughan}

Let $X, Y > 0$,  $r \in \Z[i]$, $\beta \in \{ 1, 1 + \lambda^3 \}$. To bound $H_\beta(X, Y, r)$ via Proposition~\ref{PropositionH}, we study the Dirichlet series 
\begin{align*}
	\G_{\beta}(s, v|r, a) &= \sum_{\substack{c \equiv \beta \bmod 4\\ (c, v)=1\\ a \mid c}} \frac{g_4(r, c)}{\N(c)^{\frac{1}{2} + s}} \frac{\overline{c}}{|c|}
\end{align*}
defined in~\eqref{phi-var1} for $\re(s) > 1$,
and $v, r, a \in \Z[i]$ with $v|r$.
We express these Dirichlet series as linear combinations of the Dirichlet series from~\eqref{generating-phi}
\begin{align*} 
	\G_{\beta}(s, r) &=  \sum_{\substack{c \equiv \beta \bmod 4}} \frac{g_4(r, c)}{\N(c)^{\frac{1}{2} + s}} \frac{\overline{c}}{|c|},
\end{align*}
whose analytic properties are well-understood.

We first deal with the divisibility condition $a | c$ in the sum defining $\G_{\beta}(s, v|r, a)$.

\begin{lemma}\label{lem psi remove divisibility}
	Let $a, r \in \Z[i]$ where $a$ is primary and $(a,r)=1$, and $\beta \in \{1,1+\lambda^3\}$.  Then for $\re(s)> 1$,
	\[\G_{\beta}(s, r|r, a) = 
	(-1)^{C(a, a \beta)} \frac{\overline{a}}{|a|}  \frac{g_4(r, a)}{\N(a)^{\frac{1}{2}+s}} \G_{a\beta}(s,ar|a^{2}r,1),
	\]
	where $C(\cdot,\cdot)$ is defined in \eqref{Cdef}.
\end{lemma}

\begin{proof}
	By \eqref{gauss_sum_bound}, we recall that if $(r,c)=1$, then $g_4(r, c) \neq 0$ if and only if $c$ is square-free.  In such a case, $a|c$ moreover implies that $c = a c'$ where $a, c'$ are square-free and where $(a, c')=1$.  By \eqref{rel2}, and upon noting that $a \equiv a^{-1} \bmod{4}$ whenever $a$ is primary, we then find that 
	\begin{align*}
		\G_{\beta}(s, r|r, a) 
		&= \sum_{\substack{c \equiv \beta \bmod 4\\ (r, c)=1\\a \mid c}} \frac{\overline{c}}{|c|} \; \frac{g_4(r, c)}{\N(c)^{\frac{1}{2}+s}} 
		= \sum_{\substack{c' \equiv a \beta \bmod 4\\ (ar, c')=1}} \frac{\overline{a c'}}{|a c'|} \; \frac{g_4(r, a c')}{\N(a c')^{\frac{1}{2}+s}} \\ 
 		&=  (-1)^{C(a, a \beta)} \frac{\overline{a}}{|a|}  \frac{g_4(r, a)}{\N(a)^{\frac{1}{2}+s}} \sum_{\substack{c \equiv a \beta \bmod 4\\ (a r, c)=1}} \frac{\overline{c}}{|c|} \; \frac{g_4(a^2r, c)}{\N(c)^{\frac{1}{2}+s}} ,
	\end{align*}
	as desired.
\end{proof}

We now want to remove the coprimality condition. 
Considering~\eqref{rel4}, we see that the behaviour depends on the power at which each prime divides the arguments in the Gauss sum.
Fourth powers dividing~$r$ are the easiest case.

\begin{lemma}\label{lem psi remove fourth power coprim}
	Let $\beta \in \{1,1+\lambda^3\}$, and  $r \in \Z[i]$.  We write $r = r_1 r_2^2 r_3^3 r_4^4$, where $r_1, r_2, r_3$ are square-free and coprime, and let $r_4^*$ denote the product of the primes $\pi \in \Z[i]$ such that  $\pi \mid r_4$, but $\pi \nmid r_1 r_2 r_3$.  Then for $\textnormal{Re}(s)> 1$,
	\[\G_{\beta}(s,r|r,1) = \sum_{\substack{d \mid r_4^{*}\\d \textnormal{ primary}}} \mu(d)  (-1)^{C(\beta d, d)}   \hecke{d} \frac{g_4(r_1 r_2^2 r_3^3, d)}{\N(d)^{\frac{1}{2}+s}} 
	\; \G_{\beta d}(s,r_{1}r_{2}dr_{3}|r_{1}r_{2}^{2}d^{2}r_{3}^{3},1).\]
\end{lemma}

\begin{proof}
	For all $(c,r)=1$, it follows, upon applying the change of variables $a'=ar_{4}^{4}$, that
	\begin{align*}
		g_{4}(r, c) = g_{4}(r_1 r_2^2 r_3^3 r_4^4, c) 
		&= \sum_{a'  \bmod c} \quartic{a'}{c} \quartic{r_4}{c}^{-4} e\bigg(\Tr\bigg( \frac{r_1 r_2^2 r_3^3 a'}{c}\bigg)\bigg) 
		= g_{4}(r_1 r_2^2 r_3^3, c).
	\end{align*}
	As above, it again follows from \eqref{rel2}, and upon noting that $d \equiv d^{-1} \bmod{4}$ whenever $d$ is primary, that for $\re(s)> 1$,
	\begin{align*}
		\G_{\beta}(s,r|r,1) 
		&= \sum_{\substack{c \equiv \beta \bmod 4\\ (r, c)=1}}  \frac{\overline{c}}{|c|} \; \frac{g_4(r , c)}{\N(c)^{\frac{1}{2}+s}}  
		=  \sum_{\substack{c \equiv \beta \bmod 4\\ ( r_1 r_2  r_3, c)=1}}  \frac{\overline{c}}{|c|} \; \frac{g_4(r_1 r_2^2 r_3^3 , c)}{\N(c)^{\frac{1}{2}+s}} \sum_{\substack{d \mid (c,r_4^*)\\d \textnormal{ primary}}} \mu(d) \\
		&= \sum_{\substack{ d \mid r_4^{*}\\d \textnormal{ primary}}} \mu(d)  \hecke{d}  \sum_{\substack{c' \equiv \beta d \bmod 4\\ ( r_1 r_2 r_3, d c')=1\\ (d, c')=1}}  \hecke{c'} \frac{g_4(r_1 r_2^2 r_3^3 , c' d)}{\N(c'd)^{\frac{1}{2}+s}}  \\
		&= \sum_{\substack{d \mid r_4^{*}\\d \textnormal{ primary}}} \mu(d)  (-1)^{C(\beta d, d)} \hecke{d} \frac{g_4(r_1 r_2^2 r_3^3, d)}{\N(d)^{\frac{1}{2}+s}} \sum_{\substack{c \equiv \beta d \bmod 4\\ ( r_1 r_2 r_3 d, c )=1}} \hecke{c} \frac{g_4(r_1 (r_2 d)^2 r_3^3 , c)}{\N(c)^{\frac{1}{2}+s}},
	\end{align*}
	as desired.
\end{proof}

We now deal with the coprimality condition for smaller powers. 

\begin{lemma} 
	Let $\beta \in \{1,1+\lambda^3\}$, 
	and fix $r  = r_1 r_2^2 r_3^3  \in \Z[i]$, with $r_1, r_2, r_3 \in \Z[i]$ square-free and coprime.
	Then for $\textnormal{Re}(s)>1$, we have
	\begin{align}
		\G_{\beta}(s, r_1 r_2^2 r_3^3|r,1) 
		&= \prod_{\pi \mid r_1} \left( 1 +  \frac{\overline{\pi^2}}{|\pi^2|} g_2 \left( \frac{r}{\pi} , \pi \right)  \N(\pi)^{-2s} \right)^{-1} \G_{\beta}(s, r_2^2 r_3^3|r,1)\label{linear_part_psi}  \\
		\G_{\beta}(s, r_1 r_2^2 r_3^3|r,1) &= \prod_{\pi \mid r_3}
		\left( 1 -  \frac{\overline{\pi^4}}{|\pi^4|}  \N(\pi)^{1-4s} \right)^{-1} \G_{\beta} (s,  r_1 r_2^2|r,1) \label{cube_part_psi}
	\end{align}
	and
	\begin{align}\label{square_part_psi} 
		&\G_{\beta} (s,  r_1 r_2^2 r_3^3|r,1) = \prod_{\pi \mid r_2} \left( 1 -  
		\frac{\overline{\pi^4}}{|\pi^4|}   \N(\pi)^{1-4s} \right)^{-1} \\
		& \times \sum_{\delta \mid r_2} \mu(\delta) (-1)^{C(\delta, \delta\beta)} \quartic{-1}{\delta}   \frac{\overline{\delta}^3}{|\delta|^3} \overline{g_4 \left( \frac{r}{\delta^2}, \delta \right)} \N(\delta)^{\frac12-3s} 
		\G_{\delta\beta} \left( s, r_1 r_3^3|\tfrac{r}{\delta^2},1 \right) \nonumber
	\end{align}
\end{lemma}

\begin{proof} 
	Let $\pi \mid r_1$ be primary, and write $r = r' \pi$, with $(r', \pi)=1$. Let $r_0 \mid r'$.  By \eqref{rel1}, \eqref{rel2}, and \eqref{rel4}, and upon noting that $\pi^{2}\equiv 1 \bmod 4$, we then find  that
	\begin{align*}
		\G_{\beta}(s,  r_0\pi|r' \pi,1) &= \sum_{\substack{c \equiv \beta \bmod 4\\ (r_0, c)=1}}   \frac{\overline{c}}{|c|} \; \frac{g_4(r' \pi, c)}{\N(c)^{\frac{1}{2}+s}} -\sum_{\ell \geq 1}\sum_{\substack{c'\pi^{\ell} \equiv \beta \bmod 4\\ (c',r_0\pi )=1}}   \frac{\overline{c'\pi^{\ell}}}{|c'\pi^{\ell}|} \; \frac{g_4(r' \pi, c'\pi^{\ell})}{\N(c'\pi^{\ell})^{\frac{1}{2}+s}} \\
		&= \sum_{\substack{c \equiv \beta \bmod 4\\ (r_0, c)=1}}   \frac{\overline{c}}{|c|} \; \frac{g_4(r' \pi, c)}{\N(c)^{\frac{1}{2}+s}} -\sum_{\substack{c' \equiv \beta \bmod 4\\ (c', r_0 \pi)=1}}   \frac{\overline{c' \pi^2}}{|c' \pi^2|} \; \frac{g_4(r' \pi, c' \pi^2)}{\N(c' \pi^2)^{\frac{1}{2}+s}}  \\
		&= \G_{\beta}(s,  r_0|r,1) -\frac{\overline{\pi^2}}{|\pi^2|}\frac{g_4(r' \pi, \pi^2)}{\N(\pi)^{1+2s}}\sum_{\substack{c' \equiv \beta \bmod 4\\ (c', r_0 \pi)=1}}   \frac{\overline{c'}}{|c'|} \; \frac{g_4(r' \pi^5,c')}{\N(c')^{\frac{1}{2}+s}}.
	\end{align*} 
	By \eqref{rel1} we note that $g_4(r' \pi^5,c') = g_4(r' \pi,c')$ as $(c',\pi)=1$ and by~\eqref{rel4} we have
	\[g_{4}(r'\pi,\pi^{2}) = \oquartic{r'}{\pi^{2}}g_{4}(\pi,\pi^{2}) 
	= \left(\frac{r'}{\pi}\right)_{2}g_{2}(1,\pi)\N(\pi) = g_{2}(r',\pi)\N(\pi).\]
	It further follows that
	\begin{align*}
		\G_{\beta}(s, r_0\pi|r,1) \left( 1 +  \frac{\overline{\pi^2}}{|\pi^2|} g_2 \left( \frac{r}{\pi} , \pi \right)  \N(\pi)^{-2s} \right)
		&= \G_{\beta}(s, r_0|r,1) ,
	\end{align*}
	so that \eqref{linear_part_psi} now follows by induction on the primes dividing $r_1$.

	Let $\pi \mid r_3$ and write $r = r' \pi^3$, with $(r', \pi)=1$. Let $r_0 \mid r'$. 
	Then, using~\eqref{rel2} and~\eqref{rel4} we write
	\[g_{4}(r'\pi^{3},c'\pi^{\ell}) = (-1)^{C(c',\pi^{\ell})}g_{4}(r'\pi^{3},\pi^{\ell})g_{4}(r'\pi^{3+2\ell},c')= 
	\begin{cases}
		0 & \textnormal{ if } \ell \neq 4\\
		-g_{4}(r'\pi^{3},c')\N(\pi)^{3} & \textnormal{ if } \ell = 4.
	\end{cases}  \]
	We deduce
	\begin{align*}
		\G_{\beta}(s, r_0\pi|r' \pi^3,1) 
		&= \sum_{\substack{c \equiv \beta \bmod 4\\ (r_0, c)=1}}   \frac{\overline{c}}{|c|} \; \frac{g_4(r' \pi^{3}, c)}{\N(c)^{\frac{1}{2}+s}} -\sum_{\ell \geq 1}\sum_{\substack{c'\pi^{\ell} \equiv \beta \bmod 4\\ (c',r_0\pi )=1}}   \frac{\overline{c'\pi^{\ell}}}{|c'\pi^{\ell}|} \; \frac{g_4(r' \pi^{3}, c'\pi^{\ell})}{\N(c'\pi^{\ell})^{\frac{1}{2}+s}} \\
		&= \G_{\beta}(s, r_0|r,1) + \frac{\overline{ \pi^4}}{|\pi^4|}\N(\pi)^{1-4s} \sum_{\substack{c' \equiv \beta \bmod 4 \\ (c', r_0 \pi)=1}}   \frac{\overline{c'}}{|c' |} \; \frac{g_4(r' \pi^3, c' )}{\N(c')^{\frac{1}{2}+s}}. 
	\end{align*}
	This gives
	\begin{align*}
		\G_{\beta}(s, r_0\pi|r,1)
		\left( 1 -  \frac{\overline{\pi^4}}{|\pi^4|}  \N(\pi)^{1-4s} \right) &=
		\G_{\beta} (s, r_0|r,1) 
	\end{align*}
	and we get~\eqref{cube_part_psi} by induction on the primes dividing $r_3$.

	Finally, \eqref{square_part_psi} is given by \cite[Lemma 8.1~(8.4)]{DDHL}, in the case $\ell =1$, upon making the change of variables $d \mapsto \delta$, $\nu \mapsto r$, $\alpha \mapsto r_{2} $,  $s\mapsto s+\frac12$.
\end{proof}

Combining the previous lemmas, we are finally able to obtain the desired expression for the Dirichlet series $\G_{\beta}(s,r|r; a)$.

\begin{proposition}\label{prop_phiLC} 
	Let $\beta \in \{1,1+\lambda^3\}$, and fix $r \in \Z[i]$.  
	We write $r = r_1 r_2^2 r_3^3 r_4^4$, where $r_1, r_2, r_3 \in \Z[i]$ are square-free and coprime,
	and let $r_4^*$ denote the product of the primes $\pi \in \Z[i]$ such that  $\pi \mid r_4$, but $\pi \nmid r_1 r_2 r_3$. 
	Let $a \in \Z[i]$ be primary, square-free and coprime to~$r$.  
	Then $\G_{\beta}(s,r|r, a)$ is analytic for $\textnormal{Re}(s)> \frac12$, and
	\begin{equation}
		\G_{\beta}(s,r|r, a) = \sum_{\substack{d \mid r_4^{*}\\d \textnormal{ primary}}}\sum_{\delta \mid a d r_2}F_{\beta}(s;a,d,\delta,r)\G_{\delta \beta a d} \left( s, \frac{r_{1}a^2d^{2}r_{2}^{2}r_{3}^{3}}{\delta^2}\right),
	\end{equation}
	where
	\begin{align*}
		F_{\beta}(s;a,d,\delta,r)
		:= &(-1)^{C(a, a \beta)} \frac{\overline{a}}{|a|}  \frac{g_4(r, a)}{\N(a)^{\frac{1}{2}}} \mu(d)  (-1)^{C(a\beta d, d)}   \hecke{d} \frac{g_4(r_1 a^2r_2^2 r_3^3, d)}{\N(d)^{\frac{1}{2}}}\\
		& \times \prod_{\pi \mid r_1} \bigg ( 1 +  \frac{\overline{\pi^2}}{|\pi^2|} \frac{g_2 \big( \frac{r_{1}r_{2}^{2}a^2d^{2}r_{3}^{3}}{\pi} , \pi \big)}{\N(\pi)^{\tfrac{1}{2}}}  \N(\pi)^{\tfrac{1}{2}-2s} \bigg)^{-1}
		\prod_{\pi \mid adr_2r_3} \bigg( 1 -  \frac{\overline{\pi^4}}{|\pi^4|}  \N(\pi)^{1-4s} \bigg)^{-1}\\
		&\times \mu(\delta) (-1)^{C(\delta, \delta \beta a d)} \quartic{-1}{\delta}  \frac{\overline{\delta}^3}{|\delta|^3} \frac{\overline{g_4 \big( \frac{r_{1}a^2d^{2}r_{2}^{2}r_{3}^{3}}{\delta^2}, \delta \big)}}{\N(\delta)^{\frac{1}{2}}}   \N(ad)^{-s}\N(\delta)^{1-3s}.
	\end{align*}
\end{proposition}

\begin{proof}
	Lemma~\ref{lem psi remove divisibility} and Lemma~\ref{lem psi remove fourth power coprim} yield
	\begin{align}\notag
		\G_{\beta}(s, r|r, a)  
		&=  (-1)^{C(a, a \beta)}  \frac{\overline{a}}{|a|}  \frac{g_4(r, a)}{\N(a)^{\frac{1}{2}+s}} \G_{a\beta}(s,ar|a^{2}r,1) \\
		&=   (-1)^{C(a, a \beta)}   \frac{\overline{a}}{|a|}  \frac{g_4(r, a)}{\N(a)^{\frac{1}{2}+s}}  \\
		&\times\sum_{\substack{d \mid r_4^{*}\\d \textnormal{ primary}}} \mu(d)  (-1)^{C(a\beta d, d)}  \hecke{d} \frac{g_4(r_1 a^2r_2^2 r_3^3, d)}{\N(d)^{\frac{1}{2}+s}} \cdot \G_{a\beta d}(s,r_{1}r_{2}r_{3}ad|r_{1}r_{2}^{2}a^2d^{2}r_{3}^{3},1)
	\end{align}
	where $r_1, a d r_2$ and $r_3$ are square-free and coprime.
	The proof is completed upon noting that by \eqref{linear_part_psi},
	\begin{multline*}
		\G_{\beta ad}(s,r_{1}r_{2}adr_{3}|r_{1}r_{2}^{2}a^2d^{2}r_{3}^{3},1) \\
		= \prod_{\pi \mid r_1} \left( 1 +  \frac{\overline{\pi^2}}{|\pi^2|} g_2 \left( \frac{r_{1}r_{2}^{2}a^2d^{2}r_{3}^{3}}{\pi} , \pi \right)  \N(\pi)^{-2s} \right)^{-1}  \G_{\beta a d}(s, r_{2}adr_3|r_{1}a^2d^{2}r_{2}^{2}r_{3}^{3},1)
	 \end{multline*}
	by \eqref{cube_part_psi},
	\[ \G_{\beta a d}(s, adr_{2}r_3|r_{1}r_{2}^{2}a^2d^{2}r_{3}^{3},1) = \prod_{\pi \mid r_3}
	\left( 1 - \frac{\overline{\pi^4}}{|\pi^4|}  \N(\pi)^{1-4s} \right)^{-1}
	\G_{\beta a d} (s, ad r_2|r_{1}a^2d^{2}r_{2}^{2}r_{3}^{3},1),\]
	and by \eqref{square_part_psi}, 
	\begin{align*} 
		&\G_{\beta a d} (s, a d r_2|r_{1}a^2d^{2}r_{2}^{2}r_{3}^{3},1) = \prod_{\pi \mid a d r_2} \left( 1 - 
		\frac{\overline{\pi}^4}{|\pi^4|}  \N(\pi)^{1-4s} \right)^{-1} \\
		& \times \sum_{\delta \mid a d r_2} \mu(\delta) (-1)^{C(\delta, \delta \beta a d)} \quartic{-1}{\delta}  \frac{\overline{\delta}^3}{|\delta|^3} \overline{g_4 \left( \frac{r_{1}a^2d^{2}r_{2}^{2}r_{3}^{3}}{\delta^2}, \delta \right)} \N(\delta)^{\frac12-3s} 
		\G_{\delta \beta a d} \left( s, \frac{r_{1}a^2d^{2}r_{2}^{2}r_{3}^{3}}{\delta^2}\right).
	\end{align*}
	It remains to justify that $\G_{\beta}(s,r|r, a)$ is 
	c for $\re(s) > \frac12$, which follows from the fact that
	$F_{\beta}(s;a,d,\delta,r)$ and $\G_\beta(s,r)$ are analytic for 
	$\re(s) > \tfrac{1}{2}$ as will be seen in Section~\ref{section-LOA}.
\end{proof}

We deduce a bound for the Dirichlet series of Proposition~\ref{PropositionH} in terms of the natural Dirichlet series for Gauss sums.

\begin{corollary}\label{CorLindelof} 
	Let $\beta \in \{1,1+\lambda^3\}$,  $r \in \Z[i]$ and $s=\sigma + it$ with $\sigma> \tfrac{1}{2}$. 
	Let $a \in \Z[i]$ be primary, square-free and coprime to~$r$. 
	Then for any $\varepsilon > 0$, we have
	\begin{align*}
		|\G_{\beta}(s,r|r, a)| &\ll_{\sigma, \varepsilon} \frac{\N(ar)^\varepsilon}{\N(a)^{\sigma}} \sum_{\substack{d \mid r_4^{*}\\d \textnormal{ primary}}}\sum_{\delta \mid a d r_2}
		\N(d)^{-\sigma+\varepsilon}\N(\delta)^{1-3\sigma}\left\lvert \G_{\delta \beta a d} \left( s, \frac{r_{1}a^2d^{2}r_{2}^{2}r_{3}^{3}}{\delta^2}\right)\right\rvert .
	\end{align*}
\end{corollary}

\begin{proof} 
	This follows from Proposition \ref{prop_phiLC}  and the bound
	\[F_{\beta}(s;a,d,\delta,r) \ll \N(r_{1}r_{2}r_{3}ad)^{\varepsilon}\N(ad)^{-\sigma}\N(\delta)^{1-3\sigma}.\]
\end{proof}

Therefore, bounding $H_\beta(X, Y, r)$ in Proposition \ref{PropositionH} reduces to controlling the special values 
$\G_{\delta \beta a d} (s, a^2h)$, with $h\in \Z[i]$, in the critical strip. The convexity bound of \cite[Proposition~4.3]{DDHL} is not sufficient to reach the admissible support of Theorem~\ref{main-theorem OLD}; nevertheless, after averaging over square-free values of $a$, we obtain a bound as strong as the Lindel{\"o}f bound. 

\begin{proposition}[Lindel{\"o}f on average]\label{prop-Lindelof-on-average}
		For any $0 \neq h \in \Z[i]$, $t\in \R$, $\frac12<\sigma < \frac34$ and $U \geq 1$, we have the bound 
	\begin{equation}\notag
		\sum_{\N(a) \leq U} \mu(a)^2  |\G_\beta(\sigma+it,ha^2)|^2
		\ll_{\sigma,\varepsilon} U^{1+\varepsilon} \N(h)^{\frac12+\varepsilon}(1+|t|)^{3+\varepsilon};
	\end{equation}
	for every $\varepsilon>0$.
\end{proposition}

We prove this result in the next section. 
Assuming the proposition, we get the bound on $H_\beta(X, Y, r)$ required to establish Theorem~\ref{main-theorem OLD}. 

\begin{proof}[Proof of Proposition \ref{BoundForHAfterLindelof}]
	We first bound the contribution of the first integral in Proposition~\ref{PropositionH}. 
	By shifting the integral to the line $\re(s) = \frac{1}{2}+\varepsilon$, 
	and using the bound from Lemma~\ref{TransformBounds}, we get
	\begin{equation}\begin{split}\label{12sep.1}
		\int_{(2)} \G_{\beta}&(s,r|r,a) \widetilde R_{X, Y}(s-1) \d s 
		\ll \frac{ X^{\frac\nu2+\varepsilon}}{Y} \int_{\mathbb R} |\G_{\beta}(\tfrac{1}{2}+\varepsilon+it,r|r,a)| \min (1,|t|^{-3})  \d t.
	\end{split}\end{equation}
	We then bound $\G_{\beta}(\frac12+\varepsilon+it,r|r,a)$ with the expression from Corollary~\ref{CorLindelof} and sum over $a$ to obtain that the first integral term in Proposition~\ref{PropositionH} is bounded by
	\begin{align*}
		&  \frac{X^{\frac\nu2}}{Y}  (\N(r)X)^\varepsilon \int_{\mathbb R} \sum_{\N(a) \leq U} \sum_{\substack{d \mid r_4^{*}\\d \textnormal{ primary}}}\sum_{\delta \mid a d r_2} \frac{\mu^2(a)}{\N(ad\delta)^{\frac12}}
		\left\lvert \G_{\delta \beta a d} ( \tfrac12+\varepsilon+it, \tfrac{r_{1}a^2d^{2}r_{2}^{2}r_{3}^{3}}{\delta^2})\right\rvert      \min (1,|t|^{-3})  \d t \\
		&\ll \frac{X^{\frac\nu2}}{Y}  (\N(r)X)^\varepsilon 
		\int_{\mathbb R} \sum_{\substack{d \mid r_4^{*}\\d \textnormal{ primary}}}\sum_{\substack{\N(\delta) \leq U \\
		\delta' \mid  d r_2}} \frac{1}{\N(d\delta^2\delta')^{\frac12}} 	\\ 
		& \quad \quad \times \sum_{\N(a) \leq \frac{U}{\N(\delta)}}  \frac{\mu^2(a)}{\N(a)^{\frac12}} 
		\left\lvert \G_{\delta^2\delta' \beta a d} ( \tfrac12+\varepsilon+it, \tfrac{r_{1}a^2d^{2}r_{2}^{2}r_{3}^{3}}{\delta'^2})\right\rvert \min (1,|t|^{-3})  \d t,
	\end{align*}
	where we wrote $\delta \mid adr_2$ as $\delta = \delta_a\delta'$ with $\delta_a\mid a$ and $\delta'\mid dr_2$ and used positivity: all the terms of the sum in the first line are included in the second line.
	By the Cauchy--Schwarz inequality and Proposition~\ref{prop-Lindelof-on-average}, we have
	\begin{align*}
		&\sum_{\N(a) \leq \frac{U}{\N(\delta)}} \frac{\mu^2(a)}{\N(a)^{\frac12}}
		\left\lvert \G_{\delta^2\delta' \beta a d} ( \tfrac12+\varepsilon+it, \tfrac{r_{1}a^2d^{2}r_{2}^{2}r_{3}^{3}}{\delta'^2})\right\rvert\\
		&\ll \Big(\sum_{\N(a) \leq \frac{U}{\N(\delta)}} \frac{1}{\N(a)}
	 	\Big)^{\frac12}  \Big(\sum_{\N(a) \leq \frac{U}{\N(\delta)}}
	 	\mu^2(a) \left\lvert \G_{\delta^2\delta' \beta a d} ( \tfrac12+\varepsilon+it, \tfrac{r_{1}a^2d^{2}r_{2}^{2}r_{3}^{3}}{\delta'^2})\right\rvert^2
	  	\Big)^{\frac12} \\
	  	&\ll_{\varepsilon} X^{\varepsilon}  (\tfrac{U}{\N(\delta)})^{\frac12}  \N(r)^{\frac14+\varepsilon}(1+|t|)^{\frac32+\varepsilon}, 
	\end{align*}
	since for any $d\mid r_4^*$ and $\delta'\mid dr_2$,  we have $\N(\tfrac{r_{1}d^{2}r_{2}^{2}r_{3}^{3}}{\delta'^2}) \leq \N(r)$.
	Replacing above, we obtain 
	\begin{equation}\begin{split}\notag
		\int_{(2)} \G_{\beta}(s,r|r,a) \widetilde R_{X, Y}(s-1) \d s 
		&\ll  \frac{X^{\frac\nu2}}{Y}  (\N(r)X)^\varepsilon \sum_{\substack{d \mid r_4^{*}\\d \textnormal{ primary}}}\sum_{\substack{\N(\delta) \leq U \\
		\delta' \mid  d r_2}} \frac{1}{\N(d\delta^2\delta')^{\frac12}} (\tfrac{U}{\N(\delta)})^{\frac12}  \N(r)^{\frac14}\\
		& \ll   \frac{X^{\frac\nu2+\varepsilon} U^{\frac12} \N(r)^{\frac14+\varepsilon }}{Y} .
	\end{split}\end{equation}
	The contribution to the second integral term is bounded similarly by first moving the contour of the $s'$ integral to the line~$\re(s') = \varepsilon$ and noting that the integral over the variable~$s'$ is convergent.
\end{proof}

\section{Lindel{\"o}f on average} \label{section-LOA}

\subsection{Properties of the Dirichlet series of quartic Gauss sums}

For $0 \neq r\in \mathbb Z[i]$, $\beta \in \{ 1, 1 + \lambda^3 \}$, and $\ell \in \mathbb Z$, we define
$$\psi_\beta(s,r,\ell) :=\sum_{\substack{ c\in \mathbb Z[i] \\ c\equiv \beta \bmod 4}} \frac{g_4(r, c)}{\N(c)^s} \left(\frac{\overline{c}}{|c|}\right)^\ell.$$
We remark that $\G_\beta(s, r) = \psi_\beta(s+\frac12,r,1)$, where $\G_\beta(s, r)$ is defined in \eqref{generating-phi}.
Let us first recall the analytic properties of $\psi_\beta(s,r,\ell)$, following \cite{Diaconu, DDHL}. 
For this, we introduce some notation. For $\re(s)>1$, let 
\begin{align*}
	\zeta_{\Q(i), \lambda}(s, \ell) &= \sum_{m \equiv 1 \bmod \lambda^3} \frac{1}{\N(m)^s}  \left( \frac{\overline{m}}{|m|} \right)^{4\ell}
\end{align*}
denote the Hecke $L$-function (with local factor at $\lambda$ removed) associated with the Hecke character $( \frac{\overline{m}}{|m|} )^{4\ell}$.
For $i=1, \dots, 24$, let $$\psi_{(i ,1)}(s, r, \ell)$$ be the functions defined in \cite[Equation (4.11)]{DDHL}. 

Our functions $\psi_\beta(s,r,\ell)$ coincide, up to a constant multiple, with two of the Dirichlet series $\psi_{(i ,1)}(s, r, \ell)$. More precisely, there are some values $1 \leq i_1, i_2 \leq 24$ (see \cite[Equation (2.37)]{Diaconu}) such that
\begin{align}\label{Eq psi beta psi i}
	\mathcal{V} \cdot \psi_{(i_1,1)}(s,r, \ell) = (-1)^\ell \psi_1( s, r, \ell)  \quad \text{and} \quad
	\mathcal{V} \cdot \psi_{(i_2,1)}( s, r,  \ell) = (-1)^{\ell+1} \psi_{1+\lambda^{3}}( s, r, \ell) 
\end{align}
where $\mathcal{V}$ is an absolute constant ($\mathcal{V}$ is the volume of $\mathbb C/\lambda^4\mathbb C$ with respect to $\d z$). 
Furthermore, for any $1 \leq i \leq 24$, the function~$\psi_{(i,1)}(s, r, \ell)$ can be written as a $\C$-linear combination of $\psi_{1}(s, r, \ell)$ and $\psi_{1+\lambda^3}(s, r, \ell)$ of the form (see \cite[Remark 4.2]{DDHL} and \cite{Diaconu})
\begin{align} \label{convolution}
	\psi_{(i,1)}( s, r, \ell) = \sum_{\substack {(a, b) \in S_i(r)\\ \gamma \in \{ 1, 1 + \lambda^3 \} }} A_i(a,b,r,\ell,\gamma,s) \; \psi_{\gamma}( s, (-1)^a \lambda^{2b} r, \ell) 
\end{align}
where $S_i(r) \subseteq \Z_{\geq 0}^2$ is a finite set of size bounded linearly in terms of $\mathrm{ord}_{\lambda}(r)$, the~$\lambda$-adic valuation of~$r$, and the coefficients $A_i$ are polynomials in $2^{-s}$, in particular they are bounded in vertical strips.
	
We also define
\begin{align*}
	Z_{(i,1)}(s, r, \ell) &:= \zeta_{\Q(i), \lambda}(4s-3, \ell) \psi_{(i,1)}(s, r, \ell) \\
	\widehat\psi_{(i,1)}(s,r, \ell) &:= G_\infty(s,\ell) Z_{(i,1)}(s, r, \ell) \\
	G_\infty(s,\ell) &:= \Gamma_{\mathbb C}( s+\tfrac{|\ell|}{2}-\tfrac34 ) \Gamma_{\mathbb C}( s+\tfrac{|\ell|}{2}-\tfrac12 ) \Gamma_{\mathbb C}( s+\tfrac{|\ell|}{2}-\tfrac14 )\\
	\Gamma_{\mathbb C}(s) &:= 2 (2\pi)^{-s}\Gamma(s).
\end{align*}
The functions $\widehat\psi_{(i,1)}(s,r,\ell)$ for $\ell\in \mathbb Z$, $0 \neq r\in \mathbb Z[i]$, and $i= 1,\dots , 24$,
can be meromorphically continued to $\mathbb C$; if $\ell\neq 0$ they are entire, if $\ell=0$ they have at
most two possible (simple) poles at $s=3/4$ and $s=5/4$ (see \cite[Theorem 2.1]{Diaconu} and \cite[Proposition 4.3]{DDHL}).
Moreover, for each $i$, we have the functional equation \cite[(2.33), (2.48)]{Diaconu}
\begin{equation} \label{FE}
	\widehat\psi_{(i,1)}( s, r, \ell) 
	= \N(r)^{1-s} \left(\frac{\overline{r}}{|r|}\right)^\ell \sum_{j=1}^{24}  A_{(j,i)}(2^{-s},\ell) \widehat\psi_{(j,1)}( 2-s, r, -\ell),
\end{equation} 
where $A_{(j,i)}(2^{-s},\ell)$ are the local coefficients appearing in Diaconu's result.
By adapting the argument of \cite[Section~5]{Suzuki} with the added characters $r \mapsto \left(\frac{\overline{r}}{|r|}\right)^\ell$, it follows that the $A_{(j, i)}(2^{-s},\ell)$ are rational functions in $2^{-s}$ with coefficients in $\C$ with a denominator $(2^{4s-4} -1)$ or no denominator. In particular, for all $(j,i)$ and $\ell$ the function $s\mapsto A_{(j, i)}(2^{-s},\ell)$ is holomorphic on $\re(s) \neq 1$ and is  bounded in any vertical strip in $\frac{3}{4} \leq \re(s) < 1-\varepsilon$, where the implied constant depends only on $\varepsilon$.  In the same region, Stirling's formula gives (cf. \cite[Equation (6.14)]{DDDS})
$$ \frac{G_\infty(2-s,-\ell)}{G_\infty(s,\ell)} \ll |1+\im (s)|^{6-6\re (s)}. $$

Let
\begin{equation} \label{correct-Z-tilde}
	\widetilde{Z}(s, r, \ell) := \sum_{i=1}^{24} |Z_{(i,1)}(s, r, \ell)|^2 +  \sum_{i=1}^{24} |Z_{(i,1)}(s, r, -\ell)|^2.
\end{equation}
It then follows from the functional equation \eqref{FE} that
\begin{align}\label{bound-Z-tilde}
	\widetilde Z( s,r, \ell) \ll_{\re (s)} \N(r)^{2-2\re (s)} (1+|\im (s)|)^{12-12\re (s)} \widetilde Z(2-s,r,\ell),
\end{align} 
for $\frac34 \leq \re(s)<1$.

Finally, the convexity bound from~\cite[Proposition~4.3]{DDHL} yields a similar bound for $\widetilde{Z}$.
\begin{proposition}\label{convexity} 
	For $\ell \in \mathbb{Z}$,  $0 \neq r \in \mathbb{Z}[i]$,
	for $0<\varepsilon<1/100$ we have,
	\begin{align} \label{convexbd}
		\widetilde Z( s,r, \ell)& \ll_{\varepsilon, \mathrm{ord}_{\lambda}(r)} \N(r)^{\frac{3}{2}-\re(s)+\varepsilon} (|s|^2+\ell^2+1)^{3(\frac{3}{2}-\re(s))+\varepsilon}; \\
		& \quad \text{for} \quad 1+\varepsilon<\re(s)<\tfrac{3}{2}+\varepsilon \quad \text{and} \quad  |s- \tfrac{5}{4} |>\tfrac{1}{8}. \nonumber 
	\end{align}
\end{proposition}

\subsection{Proof of Proposition \ref{prop-Lindelof-on-average}}

We first prove the following technical lemma, which is the analogue of \cite[Lemma 6.4]{DDDS}.
\begin{lemma}\label{square_free_siev}
	Let $H: \R \rightarrow \R$ be a function such that $H(x) \ll (1+x^{2})^{-1}$.  Let $s  = \sigma + it \in \C$, with $\sigma \geq 1$.  For any $\ell \in \Z$, $X \geq 1$, $\varepsilon > 0$, and $k \in \Z[i]\setminus \{0\}$, we have that
		\begin{multline*}
		\Big\lvert \sum_{\substack{c \in \Z[i]\\c \equiv \beta \bmod{4}}}\frac{g_4(k,c)}{\N(c)^{s}} \left(\frac{\overline{c}}{|c|} \right)^\ell H\left(\frac{\N(c)}{X}\right)\Big\rvert^2 \\
		\ll_{\varepsilon} \max_{\gamma \in \{1,1+\lambda^3\}}\sum_{\substack{w \in \Z[i]\\w \equiv \gamma \bmod{4}\\w \mid k^{2}}}\frac{ \N(k)^{\varepsilon}}{\N(w)^{2 \sigma-2}}\Big\lvert\sum_{\substack{n \in \Z[i]\\ n \equiv \beta \gamma \bmod{4}}}\mu^2(n)\oquartic{kw^2}{n}\frac{g_{4}(1,n)}{\N(n)^{s}}
		\left( \frac{\overline{n}}{|n|} \right)^\ell H\left(\frac{\N(wn)}{X}\right)\Big\rvert^2.
	\end{multline*}
\end{lemma}

\begin{proof}
	Write $c = wn$, where $w \mid k^{\infty}$, where $(n,k)=1$, and $w,n$ are primary.  
	Applying \eqref{rel1} and \eqref{rel2} to rewrite the Gauss sum $g_4(k, wn)$, it follows that
	\begin{align*}
		&\Big\lvert \sum_{\substack{c \in \Z[i]\\c \equiv \beta \bmod{4}}}\frac{g_4(k,c)}{\N(c)^{s}}
		\left( \frac{\overline{c}}{|c|} \right)^\ell H\left(\frac{\N(c)}{X}\right)\Big\rvert^2\\
		&= \Big\lvert \sum_{\gamma \in \{1,1+\lambda^3\}}(-1)^{C(\gamma,\beta\gamma)}\sum_{\substack{w \in \Z[i]\\w \equiv \gamma \bmod{4}\\w \mid k^{\infty}}}\frac{g_{4}(k,w)}{\N(w)^{s}}
		\left( \frac{\overline{w}}{|w|} \right)^\ell 
		\sum_{\substack{n \in \Z[i]\\ n \equiv \beta \gamma \bmod{4}}}\oquartic{kw^2}{n}\frac{g_{4}(1,n)}{\N(n)^{s}}
		\left( \frac{\overline{n}}{|n|} \right)^\ell 
		H\left(\frac{\N(wn)}{X}\right)\Big\rvert^2\\
		&\ll \max_{\gamma \in \{1,1+\lambda^3\}}\Big\lvert\sum_{\substack{w \in \Z[i]\\w \equiv \gamma \bmod{4}\\w \mid k^{\infty}}}
		\frac{g_{4}(k,w)}{\N(w)^{s}}
		\left( \frac{\overline{w}}{|w|} \right)^\ell 
		\sum_{\substack{n \in \Z[i]\\ n \equiv \beta \gamma \bmod{4}}}\oquartic{kw^2}{n}\frac{g_{4}(1,n)}{\N(n)^{s}}
		\left( \frac{\overline{n}}{|n|} \right)^\ell H\left(\frac{\N(wn)}{X}\right)\Big\rvert^2.
	\end{align*}
	Since $w \mid k^{\infty}$, we may in fact restrict to $w \mid k^{2}$ by \eqref{rel4}.  It follows by the Cauchy--Schwarz inequality  that
	\begin{multline*}
		\Big\lvert \sum_{\substack{c \in \Z[i]\\c \equiv \beta \bmod{4}}}\frac{g_4(k,c)}{\N(c)^{s+\frac{1}{2}}}
		\left( \frac{\overline{c}}{|c|} \right)^\ell
		H\left(\frac{\N(c)}{X}\right)\Big\rvert^2\ll \max_{\gamma \in \{1,1+\lambda^3\}}\sum_{\substack{w \in \Z[i]\\w \equiv \gamma \bmod{4}\\w \mid k^{2}}}\Big\lvert\frac{g_{4}(k,w)}{\N(w)}\Big\rvert^2\\
		\times\sum_{\substack{w \in \Z[i]\\w \equiv \gamma \bmod{4}\\w \mid k^{2}}}\Big\lvert\frac{1}{\N(w)^{s-1}}\sum_{\substack{n \in \Z[i]\\ n \equiv \beta \gamma \bmod{4}}}\oquartic{kw^2}{n}\frac{g_{4}(1,n)}{\N(n)^{s}}
		\left( \frac{\overline{n}}{|n|} \right)^\ell
		H\left(\frac{\N(wn)}{X}\right)\Big\rvert^2,
	\end{multline*}
	and the lemma now follows upon noting that by trivially bounding $g_{4}(k,w) \leq \N(w)$, one has
	\[\sum_{\substack{w \in \Z[i]\\w \equiv \gamma \bmod{4}\\w \mid k^2}}\Big\lvert\frac{g_{4}(k,w)}{\N(w)}\Big\rvert^2 \ll \N(k)^{\varepsilon},\]
	for any $\varepsilon > 0$.
\end{proof}

We are now ready to prove the main result of this section.

\begin{proposition}[Lindel{\"o}f on average]\label{prop-Lindelof-on-average-2}
	For any $0 \neq h \in \Z[i]$, $\ell \neq 0 \in \Z$, $t\in \R$, $1<\sigma < \frac54$ and $U \geq 1$, we have the bound 
	\begin{equation}\notag
		\sum_{\N(m) \leq U} \mu(m)^2  |\psi_\beta(\sigma+it,hm^2,\ell)|^2
		\ll_{\sigma,\varepsilon} U^{1+\varepsilon} \N(h)^{\frac12+\varepsilon}(1+|t|)^{3+\varepsilon} 
	\end{equation}
	for every $\varepsilon>0$. 
\end{proposition}

\begin{remark}
	We remark that Proposition \ref{prop-Lindelof-on-average} follows immediately from Proposition \ref{prop-Lindelof-on-average-2}, since $\G_\beta(s, r) = \psi_\beta(s+\frac12,r,1)$.
\end{remark}

\begin{proof}[Proof of Proposition \ref{prop-Lindelof-on-average-2}.]
For any fixed $i \in \lbrace 1,\dots 24\rbrace$, and $r \in \Z[i]$,	using~\eqref{convolution},	we write
	\begin{align*} 
		Z_{(i,1)}(s, r, \ell) &= \zeta_{\Q(i), \lambda}(4s-3, \ell) \psi_{(i, 1)}( s, r, \ell) \\
		&= \sum_{\substack {(a, b) \in S_i(r)\\ \beta \in \{ 1, 1 + \lambda^3 \} }} A_i(a,b,r,\ell,\beta,s) 
		\sum_{\substack{c, d \in \Z[i] \\ c \equiv \beta \bmod d \\ d \text{primary}}}  
		\frac{g_4((-1)^a \lambda^{2b} r, c) \N(d)^3}{\N(cd^4)^s}  
		\left( \frac{\overline{c}}{|c|} \right)^\ell \left( \frac{\overline{d}}{|d|} \right)^{4\ell}
	\end{align*}
	and we define
	\begin{equation}
	 	\sum_{\substack{n=1}}^\infty a_n(r, \ell, \beta) n^{-s}:=
	 	\sum_{\substack{c, d \in \Z[i] \\ c \equiv \beta \bmod 4 \\d \; \text{primary}}}  
	 	\frac{g_4(r, c) \N(d)^3}{\N(cd^4)^s}  \left( \frac{\overline{c}}{|c|} \right)^\ell \left( \frac{\overline{d}}{|d|} \right)^{4\ell}.
	\end{equation}
	
	For $X\geq 1$ and $s=\sigma+it$ with $1<\sigma<\frac54$ and using Mellin's inversion, we obtain
	\begin{multline*}
		\sum_{\substack {(a, b) \in S_i(r)\\ \beta \in \{ 1, 1 + \lambda^3 \} }} A_i(a,b,r,\ell,\beta,s) 
		\sum_{n=1}^{\infty} \frac{a_n((-1)^a \lambda^{2b} r,\ell,\beta)}{n^s} e^{-n/X} 
		= \frac{1}{2\pi i} \int_{(2)} Z_{(i,1)}(s+w,r,\ell) X^w \Gamma(w) \d w \\
		= Z_{(i,1)}(s,r,\ell) +  \frac{1}{2\pi i} \int_{(2-2\sigma)} Z_{(i, 1)}(s+w,r,\ell) X^w \Gamma(w) \d w
	\end{multline*}
	by shifting the integral to $(2-2\sigma)$ and picking the pole at $w=0$. 
	We remark that  $Z_{(i,1)}(s+w,r,\ell)$ does not have a pole when $s+w = \tfrac54$ since $\ell \neq 0$.
	Following \cite[p 33-35]{DDDS}, we use the Cauchy--Schwarz inequality to write 
	\begin{multline*}
		|Z_{(i,1)}(s,r,\ell)|^2 
		\ll	 \sum_{\substack {(a, b) \in S_i(r)\\ \beta \in \{ 1, 1 + \lambda^3 \} }} 
		\Big\lvert\sum_{n=1}^{\infty} \frac{a_n((-1)^a \lambda^{2b} r,\ell,\beta)}{n^s} e^{-n/X} \Big\rvert^2  \\
		+ \int_{(2-2\sigma)} X^{2\re(w)}\lvert \Gamma(w)\rvert \lvert\d w\rvert  
		\int_{(2-2\sigma)} |Z_{(i,1)}(s+w,r,\ell)|^2 \lvert \Gamma(w)\lvert \lvert \d w\rvert.
	\end{multline*}
	and summing for $1 \leq i \leq 24$, and over the sign $\pm \ell$  as in \eqref{correct-Z-tilde}, we have
	\begin{align}\label{16july.1}
		\widetilde Z(s,r,\ell)
		&\ll \sum_{\substack{(a,b) \in S(r) \\ \beta,\pm}} \bigg| \sum_{n=1}^{\infty} 
		\frac{a_n((-1)^a \lambda^{2b} r,\pm\ell,\beta)}{n^s} e^{-n/X} \bigg|^2  
		+  X^{4 - 4\sigma} \int_{(2-2\sigma)} \widetilde Z(s+w,r,\ell) |\Gamma(w)| |\d w|,
	\end{align}
	where the sum is over $(a,b)\in S(r) := \bigcup_{i=1}^{24} S_i(r)$ which is again a finite set, with size bounded linearly in terms of the $\lambda$-adic valuation of $r$.
	Since in the integral we have $\frac34<\re(s+w) = 2-\sigma<1$, we apply Stirling's formula (see e.g.~\cite[(5.113)]{IK}) and  \eqref{bound-Z-tilde} to bound the second term on the right-hand side of \eqref{16july.1} and get
	\begin{align*}
		X^{4 - 4\sigma}	\int_{(2-2\sigma)} \widetilde Z(s+w,r,\ell) |\Gamma(w)| |\d w| 
		&\ll_{\sigma} \bigg( \frac{\N(r)(1+|t|)^6}{X^2}\bigg)^{2\sigma-2}  
		\int_{\mathbb R} \widetilde Z(\sigma-iy-it,r,\ell) e^{-|y|} \d y .
	\end{align*}
	Plugging this into \eqref{16july.1}, taking $r=hm^2$, and summing over square-free $m$ with $\N(m)\leq U$, 
	we obtain, using~\eqref{Eq psi beta psi i},
	\begin{multline}\label{13july.1}
		\mathcal{V}^{-2}\sum_{\N(m) \leq U} \mu(m)^2  
		|\psi_\beta(\sigma+it,hm^2,\ell)\zeta_{\mathbb{Q}(i)}(4(\sigma +it)-3,\ell)|^2 \\ 
		\leq \sum_{\substack{ \N(m) \leq U}} \mu(m)^2 \widetilde Z(\sigma+it,m^2h,\ell) 
		\ll_\sigma S_2+S_3,
	\end{multline}
	where
	\begin{align*}
		S_2 	&:= \sum_{\N(m) \leq U} \mu(m)^2 \sum_{\substack{(a,b) \in S(hm^2) \\ \gamma,\pm}}
		\bigg| \sum_{n=1}^{\infty} \frac{a_n((-1)^a \lambda^{2b} hm^2,\pm\ell,\gamma)}{n^s} e^{-n/X} \bigg|^2
	\end{align*}
	and
	\begin{align} \label{bound-S3}
		S_3 	&:= \bigg(\frac{U^2\N(h)(1+|t|^6)}{X^2}\bigg)^{2\sigma-2} 
		\sum_{\N(m) \leq U} \mu(m)^2 \int_{\mathbb R}  \widetilde Z(\sigma-it-iy,m^2h,\ell) e^{-|y|}  \d y.
	\end{align}
	
	We begin with bounding $S_2$. We have
	\begin{align*}
		S_2	&= \sum_{\N(m) \leq U} \mu(m)^2 \sum_{\substack{(a,b) \in S(hm^2) \\ \gamma,\pm}}
		\bigg| 	\sum_{\substack{c, d \in \Z[i] \\ c \equiv \gamma \bmod 4\\d \text{ primary}}}  
		\frac{g_4((-1)^a \lambda^{2b} hm^2, c) \N(d)^3}{\N(cd^4)^s}  \left( \frac{\overline{c}}{|c|} \right)^{\pm\ell} \left( \frac{\overline{d}}{|d|} \right)^{\pm4\ell}e^{-\N(cd^4)/X} \bigg|^2 \\
		&\ll  \sum_{\N(m) \leq U}  \sum_{\substack{(a,b) \in S(hm^2) \\ \gamma,\pm}}
		\sum_{d\equiv 1 \bmod \lambda^3 } \frac{1}{\N(d)^{\sigma}}  \mu(m)^2 
		\bigg| \sum_{c\equiv \gamma \bmod 4} \frac{g_4( (-1)^a \lambda^{2b} hm^2,c)}{\N(c)^{\sigma+it}} 
		\left( \frac{\overline{c}}{|c|} \right)^{\pm\ell} e^{-\N(cd^4)/X} \bigg|^2 ,
	\end{align*}
	applying the Cauchy--Schwarz inequality to the $d$-sum with $\alpha_d = \N(d)^{3 - \frac{7\sigma}{2}}$, 
	which implies that $\sum_{d} \alpha_d^2$ converges since $\sigma > 1$.
	Note that for $m$ square-free the size of the set $S(hm^2)$ can be bounded linearly in terms of $\mathrm{ord}_{\lambda}(h) +2 \ll \N(h)^\varepsilon$. 
	Using Lemma \ref{square_free_siev}, we get
	\begin{align*}
		S_2 &\ll_\varepsilon \max_{\gamma_1 \in \{ 1, 1+\lambda^3 \}} \max_{\pm} \max_{(a,b) \in \Z^2} (\N(h)U^2)^\varepsilon  
		\sum_{d\equiv 1 \bmod \lambda^3 } \frac{1}{\N(d)^{\sigma}}  \sum_{\N(m) \leq U} \mu(m)^2  \\
		& \times
		\sum_{\substack{ w \equiv \gamma_1 \bmod 4 \\ w \mid (hm^2)^2}} \frac{1}{\N(w)^{2\sigma-2}} 
		\bigg| \sum_{n \equiv \gamma \gamma_1 \bmod 4} \mu(n)^2 \frac{g_4(1,n)}{\N(n)^{\sigma+it}} 
		\overline{ \left( \frac{(-1)^a \lambda^{2b} hm^2w^2}{n} \right)_4} 
		\left( \frac{\overline{n}}{|n|} \right)^{\pm\ell}  e^{-\N(wnd^4)/X} \bigg|^2 
	\end{align*}

	Writing $w= \alpha \beta^2$ where $\alpha \equiv \gamma_1 \bmod 4$ and $\beta \equiv 1 \bmod{\lambda^3}$, 
	with $\mu(\alpha)^2=1$, the condition $w|(hm^2)^2$ implies $\alpha \beta |hm^2$ and consequently $q|m^2$, 
	where we denote $q = q(\alpha, \beta, h) := \alpha \beta/(h,\alpha \beta)$. 
	We write $m=Qv$ for $Q = Q(\alpha, \beta, h) :=\text{rad}(q)$ and $v\equiv 1 \bmod \lambda^3$, and obtain
	\begin{align*}
		S_2
		&\ll_\varepsilon (\N(h)U^2)^\varepsilon 
		\max_{\substack{\gamma_1, \gamma_2 \in \{ 1, 1+\lambda^3 \} \\ a, b, \pm}}
		 \sum_{\substack{\alpha \equiv \gamma_1 \bmod4\\ \beta  \equiv 1 \bmod \lambda^3}}
		\frac{\mu^2(\alpha)}{\N(\alpha)^{2\sigma-2}\N(\beta)^{4\sigma-4}} 
		\sum_{d\equiv 1 \bmod \lambda^3 } \frac{1}{\N(d)^{\sigma}} \\
		&\hspace{0.5cm} \times \sum_{\N(v) \leq U/\N(Q)} \mu(v)^2 
		\bigg| \sum_{n \equiv  \gamma_2 \bmod 4} \mu(n)^2 \frac{g_4(1,n)}{\N(n)^{\sigma+it}} 
		\overline{ \left( \frac{(-1)^a \lambda^{2b}  hQ^2v^2 \alpha^2 \beta^4}{n} \right)_4} 
		\left( \frac{n}{|n|} \right)^{\pm\ell}  e^{-\N(\alpha \beta^2nd^4)/X} \bigg|^2 .
	\end{align*}
	
	We now use the asymptotic large sieve to bound the sum over square-free $n$ and $v$, 
	exploiting the oscillation of the {\it quadratic} character
	\begin{align}\label{17july.1} 
		\Big( \frac{v^2}{n} \Big)_4 = \Big( \frac{v}{n} \Big)_2.
	\end{align}
	Namely, denoting
	\begin{align*}
		b_n = b_n(\sigma+it,a,b, h,\alpha,\beta) 
		:= \frac{g_4(1,n)}{\N(n)^{\sigma+it}} \overline{ \left( \frac{(-1)^{a} \lambda^{2b}  h Q^2 \alpha^2 \beta^4}{n} \right)_4} 
		\left( \frac{n}{|n|} \right)^{\pm\ell}  e^{-\N(\alpha \beta^2 n d^4)/X} ,
	\end{align*}
	and remarking that only the contribution from $\N(n \alpha \beta^2d^4)\leq X^{1+\varepsilon}$ is non-negligible, we have
	\begin{align*}
		S_2
		\ll_\varepsilon (\N(h)U^2)^\varepsilon 
		&\max_{\substack{\gamma_1, \gamma_2 \in \{ 1, 1+\lambda^3 \} \\ a, b, \pm}}
		\sum_{\substack{\N(\alpha \beta^2 d^4) \leq X^{1+\varepsilon} \\
		\alpha \equiv \gamma_1 \bmod 4 \\ \beta,d  \equiv 1 \bmod \lambda^3}}
		\frac{\mu^2(\alpha)}{\N(\alpha \beta^2)^{2\sigma-2} \N(d)^{\sigma} } \\
		&\times
		\sum_{\N(v) \leq U/\N(Q)} \mu(v)^2 \bigg| 
		\sum_{\substack{n \equiv  \gamma_2 \bmod 4 \\ \N(\alpha \beta^2nd^4)\leq X^{1+\varepsilon}}} \mu(n)^2 b_n  
		\left( \frac{v}{n} \right)_2  \bigg|^2 .
	\end{align*}
	An application of the quadratic large sieve \cite[Theorem 1]{HB} then yields 
	\begin{align*}
		\sum_{\N(v) \leq U/\N(Q)} \mu(v)^2 \bigg| \sum_{\N(n)\leq X^{1+\varepsilon}/\N(\alpha \beta^2d^4)} 
		\mu(n)^2 b_n { \left( \frac{v}{n} \right)_2} \bigg|^2  
		\ll_\varepsilon (XU)^\varepsilon \bigg(\frac{U}{\N(Q)} +  \frac{X}{\N(\alpha \beta^2d^4)} \bigg)
	\end{align*}
	for every $\varepsilon>0$, since for $\sigma >1$ 
	\begin{align*}
		\sum_{\N(n)\leq X^{1+\varepsilon}/\N(\alpha \beta^2d^4)} |b_n|^2 
		\ll  \sum_{\N(n) \leq X^{1+\varepsilon}/\N(\alpha \beta^2d^4)} \frac{1}{\N(n)^{2\sigma-1}} \ll_\sigma 1.
	\end{align*}
	
	Therefore, 
	\begin{align*}
		S_2
		&\ll_\varepsilon (\N(h)XU^2)^{2\varepsilon} \sum_{\N(\alpha \beta^2) \leq X^{1+\varepsilon}}		
		\frac{ \mu^2(\alpha) }{\N(\alpha \beta^2)^{2\sigma-2}} 
		\sum_{\N(d^4) \leq X^{1+\varepsilon}/\N(\alpha \beta^2)} \frac{1}{\N(d)^{\sigma}} 
		\bigg(\frac{U}{\N(Q)} +  \frac{X}{\N(\alpha\beta^2d^4)}  \bigg)\\
		&\ll_\varepsilon U (\N(h)XU^2)^{2\varepsilon}   \sum_{\N(\alpha \beta^2) \leq X^{1+\varepsilon}}		
		\frac{ \mu^2(\alpha) }{\N(\alpha \beta^2)^{2\sigma-2} \N(Q)} + X (\N(h)U^2)^{2\varepsilon} 
	\end{align*}
	since the sum over $d$ converges. Now,
	$$ \N(Q) = \N\bigg(\text{rad}\bigg(\frac{\alpha\beta}{(h,\alpha\beta)}\bigg)\bigg) 
	\geq  \N\bigg(\mathrm{rad}\bigg(\frac{\alpha\beta}{(h,\alpha)(h,\beta)}\bigg) \bigg),$$
	and we consider
	\begin{align*}
		C(h) := \sum_{\alpha, \beta \in \Z[i]}
		\frac{\mu^2(\alpha)  }{\N(\alpha \beta^2)^{2\sigma-2} \, 
		\N\bigg(\mathrm{rad}\bigg( \frac{\alpha}{(h,\alpha)} \frac{\beta}{(h,\beta)}\bigg)\bigg)}.
	\end{align*}
	Since $C(h)$ is a doubly-multiplicative sum, we can write it as an Euler product.
	For each prime ideal $\fp$ of $\Z[i]$,  we have $v_{\fp}(\alpha) \in \lbrace0,1\rbrace$, and 
	writing $v_{\fp}(\beta) = n \geq 0$, we obtain for any $\varepsilon > 0$ 
	\begin{align*}
		C(h) =& \prod_{\fp \nmid h} \left( 1  + \sum_{n=1}^\infty  \frac{1}{\N(\fp)^{1 + 2 n (2\sigma - 2)}}  
		+ \sum_{n = 0}^\infty  \frac{1}{\N(\fp)^{1 +(2 n+1) (2\sigma - 2)}} \right) \\
		&  \times  \prod_{\fp \mid h} \left( 1  
		+ \sum_{n=1}^{v_\fp(h)}  \frac{1}{\N(\fp)^{2 n (2\sigma - 2)}}  
		+ \sum_{n=v_\fp(h)+1}^\infty  \frac{1}{\N(\fp)^{1 + 2 n (2\sigma - 2)}}  
		\right)\left( 1 + \frac{1}{\N(\fp)^{(2\sigma -2)}} \right) \\
		\ll&_{\sigma,\varepsilon}  \N(h)^{\varepsilon}
	\end{align*}
	since $\sigma > 1$.
	This gives
	\begin{equation}\begin{split}\label{13july.2}
		S_2 &\ll_{\sigma,\varepsilon} (\N(h)XU^2)^{2\varepsilon} \Big(U +  X  \Big).
	\end{split}\end{equation}
	
	Hence, Equations \eqref{13july.1}, \eqref{bound-S3} and \eqref{13july.2} yield
	\begin{equation}\begin{split}\label{13july.3}
			\mathcal Z_U(\sigma+it,h,\ell) 
			:= & \sum_{\N(m) \leq U} \mu(m)^2 \widetilde Z(\sigma +it,m^2h,\ell) \\
			\leq & \; C_2(\sigma,\varepsilon)\cdot (\N(h)XU^2)^{2\varepsilon} \cdot \Big(U +  X  \Big)\\
			&+ C_3(\sigma) \cdot \Big(U^2\N(h)(1+|t|^6)X^{-2}\Big)^{2\sigma-2} \cdot \mathcal I_U(\sigma+it,h,\ell)
	\end{split}\end{equation}
	where
	\begin{align*}
		\mathcal I_U(\sigma+it,h,\ell) :=\int_{\mathbb R} e^{-|y|} \mathcal Z_U(\sigma-it-iy,h, \ell)  dy.
	\end{align*}
	
	We introduce the  quantity
	\begin{align*}
		\sup_{y \in \mathbb R} \frac{\mathcal Z_U(\sigma+iy,h,\ell)}{(1+|y|)^3} 
	\end{align*}
	which is finite by the convexity bound~\eqref{convexbd}, and is attained for a real $\tilde t = \tilde t(\sigma,h,\ell;U) $. 
	Then, for every $y\in\mathbb R$,
	\begin{align}\label{18july.1}
		\frac{\mathcal Z_U(\sigma+i\tilde t,h,\ell)}{(1+|\tilde t|)^3} \geq \frac{\mathcal Z_U(\sigma+iy,h,\ell)}{(1+|y|)^3} .
	\end{align}
	As a consequence,
	\begin{align*}
		\mathcal I_U(\sigma+i\tilde t,h,\ell) 
		\leq \mathcal Z_U(\sigma+i\tilde t,h, \ell) \int_{\mathbb R} e^{-|y|} 
		\bigg( \frac{1+|\tilde t+y|}{1+|\tilde t|}\bigg)^3  dy \leq C_1 \cdot \mathcal Z_U(\sigma+i\tilde t,h, \ell).
	\end{align*}
	Choosing $t=\tilde t$ and $X= (2C_1C_3(\sigma))^{1/(4\sigma-4)} U\sqrt{\N(h)} (1+|\tilde t|)^3$, 
	Equation \eqref{13july.3} then reads
	\begin{equation}\begin{split}\notag
		\mathcal Z_U(\sigma+i\tilde t,h,\ell) 
		\leq & \; C_2(\sigma,\varepsilon)\cdot (\N(h)XU^2)^{2\varepsilon} \cdot (U +  X  )\\
		&+ C_1 C_3(\sigma) \cdot \Big(U^2\N(h)(1+|\tilde t|)^6X^{-2}\Big)^{2\sigma-2} 
		\cdot \mathcal Z_U(\sigma+i\tilde t,h, \ell)\\
		\leq & \; C_2(\sigma,\varepsilon)(2C_1C_3(\sigma))^{\frac{2\varepsilon}{4\sigma-4}} 
		\cdot (\N(h)^{\frac{3}{2}}U^3  (1+|\tilde t|)^3)^{2\varepsilon} U
		\big(1 +  (2C_1C_3(\sigma))^{\frac{1}{4\sigma-4}} \sqrt{\N(h)} (1+|\tilde t|)^3 \big)\\
		&+ \frac12  \mathcal Z_U(\sigma+i\tilde t,h, \ell)
	\end{split}\end{equation}
	and then
	\begin{equation}\begin{split}\notag
		\mathcal Z_U(\sigma+i\tilde t,h,\ell) 
		\ll_{\sigma,\varepsilon} & \Big(\N(h)U(1+|\tilde t|)\Big)^{6\varepsilon}  U\sqrt{\N(h)}(1+|\tilde t|)^3,
	\end{split}\end{equation}
	which, by \eqref{18july.1} gives
	\begin{equation}\notag
		\mathcal Z_U(\sigma+i t,h,\ell) 
		\ll_{\sigma,\varepsilon}  \Big(\N(h)U(1+|t|)\Big)^{6\varepsilon}  U\sqrt{\N(h)}(1+|t|)^3 
	\end{equation}
	for every $t\in \mathbb R$. Combining this and~\eqref{13july.1} with the bound $\lvert \zeta_{\mathbb Q(i),\lambda} (\sigma + it,\ell) \rvert \geq \frac{\zeta_{\mathbb Q(i),\lambda}(2\sigma,0)}{\zeta_{\mathbb Q(i),\lambda}(\sigma,0)}\gg_\sigma 1$ for $s=\sigma+it$ and $\sigma>1$ finishes the proof of Proposition~\ref{prop-Lindelof-on-average-2}.
\end{proof}

\section{Calculating Root Numbers} \label{root-number}

In this section, we compute the sign of the functional equation for the $L$-functions $L(s,\xi_d)$ as $d$ takes odd fourth-power free values, giving a proof for Lemma~\ref{lemma-sign-s1}. Recall that a Hecke character $\xi$ on $\mathbb{Z}[i]$ with conductor $\mathfrak{f}$ can be given as a product of an infinite part $\xi_\infty$, which is a multiplicative character on $\mathbb C^\times$ and a finite part $\xi_{\mathrm{fin}}$, which is a character of the finite group  $(\mathbb Z[i]/\mathfrak{f})^\times$.
Following~\cite[$(3.85)$, $(3.86)$]{IK}, we will use the formula
\begin{align}\label{Eq Formula sign}
	W(\xi_{d}) = -i \N(\mathfrak{f}_{d})^{-\frac12} \xi_{d,\infty}(\gamma_{d}) \sum_{x \in \mathbb{Z}[i]/\mathfrak{f}_{d}}\xi_{d,\mathrm{fin}}(x) e^{2\pi i \tr\left(\tfrac{x}{\gamma_{d}}\right)}
\end{align}
where $\gamma_{d} \in \Z[i]$ is any generator of the ideal $2 \mathfrak{f}_{d}$.
Similarly to ~\cite[Lemma~2.3]{DDW}, we can write
$\xi_d = \xi_{\infty} \xi_{d,\mathrm{fin}}$ with
$  \xi_{\infty} : \alpha \mapsto \frac{\alpha}{\lvert \alpha \rvert}$ and
\[\xi_{d,\mathrm{fin}} =
\overline{\chi}_{(d)}\cdot \eta_{d},
\]
where $\eta_d$ is a character modulo a power of the ramified prime $(\lambda)$, depending on the congruence class of $d \bmod 8$.
We introduce some notations for quadratic and quartic characters modulo powers of $(\lambda)$.
Let
$\chi_{(2)}:\left(\Z[i]/(2)\right)^{\times} \longrightarrow \{\pm 1\}$ be the character given by
\begin{align} \label{def-2}
	\quad  \chi_{(2)}(\alpha) &:= \begin{cases} 1 &\text{ for } \alpha \equiv 1 \bmod 2 \\ -1 &\text{ for }  \alpha \equiv i \bmod 2. \end{cases}
\end{align}
Let $\chi_{(\lambda^3)}:\left(\Z[i]/(\lambda^3)\right)^{\times} \longrightarrow \{\pm 1, \pm i\}$ be the character given by
\begin{align}
	\label{def-2+2i}
	\chi_{(\lambda^3)}(\alpha)&:= \begin{cases} 1 &\text{ for }  \alpha \equiv 1 \bmod {\lambda^3} \\
		-1 &\text{ for }  \alpha \equiv  -1 \bmod {\lambda^3} \\
		-i &\text{ for }  \alpha \equiv i \bmod {\lambda^3} \\
		i &\text{ for }  \alpha \equiv -i  \bmod {\lambda^3}.
	\end{cases}
\end{align}
And let $\chi_{(4)}:\left(\Z[i]/(4)\right)^{\times} \longrightarrow \{\pm 1, \pm i\}$ be the character given by
\begin{align}\label{def-4}
	\chi_{(4)}(\alpha) &:= \begin{cases}
		1 &\text{ for }  \alpha \equiv 1 \text{ or } -1 + \lambda^3\bmod 4\\
		-1 &\text{ for }  \alpha \equiv -1 \text{ or } 1 + \lambda^3 \bmod 4\\
		i &\text{ for }  \alpha \equiv i \text{ or } -i + \lambda^3 \bmod 4\\
		-i &\text{ for }  \alpha \equiv -i \text{ or } i + \lambda^3 \bmod 4.
	\end{cases}
\end{align}

The characters $\chi_{(2)}, \chi_{(\lambda^3)}$ and $\chi_{(4)}$ may be extended to $\mathbb Z[i]$ where they are primitive with conductors given by $(2)$, $(\lambda^3)$, and $(4)$, respectively.  We further remark that $\chi_{(2)} = \chi_{(\lambda^3)}^2 = \chi_{(4)}^2$ and that for~$\alpha \in \Z[i]$, with $(\alpha,\lambda) = 1$, the primary generator of $(\alpha)$ is given by
\begin{equation}\label{boldalpha}
	\boldsymbol{\alpha} =  {\chi}_{(\lambda^3)}(\alpha)  \; \alpha.
\end{equation}
With these notations, the proof of~\cite[Lemma~2.3]{DDW} counting prime divisors with multiplicities gives
\begin{align}\label{Eq def eta}
	\eta_{d} &:=
	\begin{cases}
		\chi_{(\lambda^3)} & \textnormal{ when } d\equiv 1 \bmod 8\\
		\chi_{(\lambda^3)}\chi_{(2)} & \textnormal{ when } d\equiv 5 \bmod 8 \\
		\overline{\chi}_{(4)}\chi_{(2)} & \textnormal{ when } d\equiv 3 \bmod 8 \\
		\overline{\chi}_{(4)} & \textnormal{ when }d\equiv 7 \bmod 8.
	\end{cases}
\end{align}
Observe that $\eta_{d}$ is then a primitive character modulo $(g)$ with  $$g:=\begin{cases}
	\lambda^3 & \textnormal{ when } d\equiv 1\bmod 4\\
	4 & \textnormal{ when } d\equiv 3 \bmod 4.\end{cases}$$
We now have all the necessary information to use the formula~\eqref{Eq Formula sign} and prove Lemma~\ref{lemma-sign-s1}.		
	
	\begin{proof}[Proof of Lemma \ref{lemma-sign-s1}]	

		Write $\rad(d) = \prod_j p_j >0$, and $d= \sgn(d)\prod_j p_j^{e_j}$, where $p_j$ run through the distinct positive {\it rational} primes dividing~$d$. 
		Since $(g,d)=1$, by the Chinese remainder theorem there exists a ring isomorphism
		\[\Z[i]/(g) \times \prod_{p_{j}|d}\Z[i]/(p_{j}) \rightarrow \Z[i]/(g\cdot \rad(d)) \quad (x_{0},(x_{j})_{j}) \mapsto u x_{0}+\sum_{j}v_{j}x_{j},\]
		where $u \equiv 1\bmod g$ and $u \equiv 0 \bmod{\rad(d)}$, while $v_{j} \equiv 1 \bmod {p_{j}}$  and  $v_{j} \equiv 0 \bmod {\frac{g\cdot \rad(d)}{p_{j}}}$ for all $j$. Upon choosing $\gamma_{d} = 2g\cdot \rad(d)$, 
		we find that
		\begin{align}
			\begin{split}
				&\sum_{x \in \mathbb{Z}[i]/\mathfrak{f}_{d}}\xi_{d,\mathrm{fin}}(x) e^{2\pi i \tr\left(\tfrac{x}{\gamma_{d}}\right)} = \sum_{x \in \mathbb{Z}[i]/(g\cdot \rad(d))}\overline{\chi}_{(d)}(x)\cdot \eta_{d}(x) e^{2\pi i \tr\left(\tfrac{x}{2g\cdot \rad(d)}\right)}\\
				&= \sum_{x_{0} \in \mathbb{Z}[i]/(g)}\prod_{p_{j}|d}\sum_{x_{j} \in \mathbb{Z}[i]/(p_{j})} \overline{\chi}^{e_j}_{(p_{j})}\Big(ux_{0}+\sum_{j}v_{j} x_{j}\Big)\eta_{d}\Big(ux_{0}+\sum_{j}v_j x_{j}\Big)  e^{2\pi i \tr\left(\tfrac{u x_{0} + \sum_{j}v_{j}x_{j}}{2g\cdot \rad(d)}\right)}\\
				&= \sum_{x_{0} \in \mathbb{Z}[i]/(g)}\eta_{d}\left(ux_{0}\right)e^{2\pi i \tr\left(\tfrac{u x_{0}}{2g\cdot \rad(d)}\right)}\prod_{p_{j}|d}\sum_{x_{j} \in \mathbb{Z}[i]/(p_{j})} \overline{\chi}^{e_j}_{(p_{j})}\left(v_{j} x_{j}\right)  e^{2\pi i \tr\left(\tfrac{v_{j}x_{j}}{2g\cdot \rad(d)}\right)}.
			\end{split}
		\end{align}
		
		Applying the change of variables $\alpha = u x_{0}/\rad(d)$ and $\beta_j =
		v_j x_{j}p_{j}/g\cdot \rad(d)$, we then find that
		\begin{align}\label{root_number_split}
			\begin{split}
				W(\xi_{d}) &=  \frac{-i}{|g\cdot \rad(d)|} 
				\frac{g}{|g|}
				\eta_{d}(\rad(d))\sum_{\alpha \in \mathbb{Z}[i]/(g)} \eta_{d}(\alpha) e^{2\pi i \tr\left(\tfrac{\alpha}{2g}\right)} \\
				&\phantom{=}\times  \prod_{\substack{p_{j}|d\\ p_{j} > 0}}
				\overline{\chi}^{e_j}_{(p_{j})}\left(\frac{g\cdot \rad(d)}{p_{j}}\right)\sum_{\beta_j \in \mathbb{Z}[i]/(p_{j})}\overline{\chi}^{e_j}_{(p_{j})}(\beta_j) e^{2\pi i \tr\left(\tfrac{\beta_j}{2p_{j}}\right)} \\
				&= W(\eta_{d}, 2) \times \eta_{d}(\rad(d)) \times \overline{\chi}_{(d)}(g) 
				\times \prod_{\substack{p_{j}|d\\ p_{j} > 0}} W(\overline{\chi}^{e_j}_{(p_{j})}, p_j) 
			\end{split}
		\end{align}
		where
		\begin{align*}
			W(\eta_{d}, 2) &:= \frac{-i}{|g|} 
			\frac{g}{|g|}
			\sum_{\alpha \in \mathbb{Z}[i]/(g)} \eta_{d}(\alpha) e^{2\pi i \tr\left(\tfrac{\alpha}{2g}\right)} \\
			W(\overline{\chi}^{e}_{(p)}, p) &:= \frac{1}{p}\sum_{x \in \mathbb{Z}[i]/(p)}\overline{\chi}^{e}_{(p)}(x) e^{2\pi i \tr\left(\tfrac{x}{2p}\right)},
		\end{align*}
		and where we have used  \eqref{chi_at_integers}.
		
		We first study the contribution of the factor at~$2$.
	If  $d \equiv 1 \bmod 4$, then we compute
		\begin{align} 
		 W(\eta_{d}, 2)
			=  \begin{cases}
				1 & \textnormal{ if }  d\equiv 1 \bmod 8 \\
				-i & \textnormal{ if }  d\equiv -3 \bmod 8.
			\end{cases}
		\end{align}
		Similarly when $d\equiv 3 \bmod 4$, we have
		\begin{equation} \label{gauss_sum_3mod4}
			  W(\eta_{d}, 2) = 1. 
		\end{equation}
		Let us now study the contribution at a rational odd prime $p$. 
	We have
		\begin{align}
			W(\overline{\chi}^{e}_{(p)}, p) =\begin{cases}
				(-1)^{e}& \textnormal{ when }p \equiv \pm 3 \bmod 8\\
				1 & \textnormal{ when }p \equiv \pm 1 \bmod 8,
			\end{cases}
		\end{align}
		and therefore for odd fourth-power-free $d \in \mathbb{Z}$,
		\begin{equation}\label{root_number_product}
			\prod_{\substack{p_{j}|d\\ p_{j} > 0}} W(\overline{\chi}^{e_j}_{(p_j)}, p_j)= \begin{cases}
				-1& \textnormal{ when }d \equiv \pm 3 \bmod 8\\
				1 & \textnormal{ when }d \equiv \pm 1 \bmod 8.
			\end{cases}
		\end{equation}
	
	  When  $d \equiv 1 \bmod 8$, it follows from
		\eqref{Eq def eta}, \eqref{root_number_split} and 
		\eqref{root_number_product}, and~\eqref{Eq suppl quartic law} that
		\begin{align*}
			W(\xi_{d})	& =W(\eta_{d}, 2) \times \eta_{d}(\rad(d)) \times \overline{\chi}_{(d)}(g) 
			\times \prod_{\substack{p_{j}|d\\ p_{j} > 0}} W(\overline{\chi}^{e_j}_{(p_{j})}, p_j) \\
		&	= \chi_{(\lambda^3)}(\rad(d))  \overline{\chi}_{(d)}(\lambda^3)  \\
			&= (-1)^{\frac{\rad(d) -1}{2}}
			i^{\frac{(d-1)}{4}}.
		\end{align*}
		
		Similarly, when $d \equiv 5 \bmod 8$, we find that
		\begin{align*}
			W(\xi_{d})	&= i \cdot \chi_{(\lambda^3)}(\rad(d))\chi_{(2)}(\rad(d))			\overline{\chi}_{(d)}(\lambda^3)\\
			&= (-1)^{\frac{\rad(d) -1}{2}}  i^{\frac{d+3}{4}}\cdot \\
		\end{align*}
		Finally, when $d \equiv 3 \bmod 8$, we note by \eqref{root_number_split}, \eqref{gauss_sum_3mod4}, \eqref{root_number_product}, that
		\begin{align*}
			W(\xi_{d})&= -\overline{\chi}_{(4)}(\rad(d))\chi_{(2)}(\rad(d)) =(-1)^{\frac{\rad(d)+1}{2}}
		\end{align*}
		and if $d \equiv 7 \bmod 8$, we have
		\begin{align*}
			W(\xi_{d})	= \overline{\chi}_{(4)}(\rad(d)) = (-1)^{\frac{\rad(d)-1}{2}}.
		\end{align*}
		One can then split according to the congruence class of $d \bmod 16$ to obtain the statement of Lemma~\ref{lemma-sign-s1}. 
	\end{proof}

\end{document}